%% file: Rigidity.tex
\documentclass[14pt,a4paper,reqno]{amsart}
\usepackage{vmargin,color}
\usepackage[latin1]{inputenc}
\usepackage{amsmath,amsfonts,amsthm,epsfig,graphicx}
\usepackage{pstricks,pst-plot,multido}
\usepackage{caption}
\usepackage{esint} 
\usepackage{pgfplots}
\usepackage{graphicx,tikz,subfigure,color,calc}
\usetikzlibrary{decorations.pathreplacing}

\pgfplotsset{compat=1.10}
\usepgfplotslibrary{fillbetween}
\usepgfplotslibrary{polar}
\usetikzlibrary{patterns}
\usepackage{relsize}

\usepackage{bigints}

\definecolor{darkgreen}{rgb}{0.0,0.5,0.0}
\definecolor{darkblue}{rgb}{0.0,0.0,0.3}
\definecolor{nicosred}{rgb}{0.65,0.1,0.1}
\definecolor{light-gray}{gray}{0.6}
\definecolor{really-light-gray}{gray}{0.8}

\usepackage{marginnote}

\def\R{\mathbb R}

\newcommand{\diver}{\mathrm{div}} 
\newcommand{\mres}{\mathbin{\vrule height 1.6ex depth 0pt width  
0.13ex\vrule height 0.13ex depth 0pt width 1.3ex}}    

\newcommand{\Upp}{^{\vee}}
\newcommand{\Low}{^{\wedge}}

\DeclareMathOperator*{\aplim}{ap\,lim}

\newtheorem{theorem}{Theorem}[section]
\newtheorem{remark}[theorem]{Remark}
\newtheorem{definition}[theorem]{Definition}

\newtheorem{proposition}[theorem]{Proposition}
\newtheorem{lemma}[theorem]{Lemma}
\newtheorem{corollary}[theorem]{Corollary}

\newtheorem{open}[theorem]{Open Problem}

\numberwithin{equation}{section}
\numberwithin{figure}{section}

\begin{document}

\title{Minimally singular functions and the rigidity problem for Steiner's perimeter inequality}

\author[M. Perugini]
 {Matteo Perugini}
 \address[Matteo Perugini]{Universit\'a degli Studi di Milano, Dipartimento di Matematica, Milano}
 \email{matteo.perugini@unimi.it}

\begin{abstract}
{\rm Let $n\geq 1$, and let $\Omega\subset \mathbb{R}^n$ be an open and connected set with finite Lebesgue measure. Among functions of bounded variation in $\Omega$ we introduce the class of \emph{minimally singular} functions. Inspired by the original theory of Vol'pert of one-dimensional restrictions of $BV$ functions, we provide a geometric characterization for this class of functions via the introduction of a pseudometric that we call \emph{singular vertical distance}. As an application, we present a characterization result for \emph{rigidity} of equality cases for Steiner's perimeter inequality. By \emph{rigidity} we mean that the only extremals for Steiner's perimeter inequality are vertical translations of the Steiner symmetric set.} 
\end{abstract}

\maketitle

\section{Introduction}
\subsection{Overview}
Firstly introduced by Jacob Steiner in \cite{Steiner}, Steiner symmetrisation is a well known rearrangement procedure that given any set $E$ and an hyperplane $H$, it produces the so called \emph{Steiner symmetric} set $E^s$ of $E$ with respect to $H$. The set $E^s$ has the property that, every straight line $L$ orthogonal to $H$ intersects $E$ if and only if it also intersects $E^s$, the length of $L\cap E$ coincides with the length of $L\cap E^s$, and $L\cap E^s$ is a segment whose mid point belongs to $H$. By construction it can be proved that $E^s$ has the same volume of the initial set $E$, while its perimeter, that we denote with $P(E^s)$, is never greater than the perimeter of $E$, namely
\begin{align}\label{intro_Steiner inequality}
P(E^s)\leq P(E).
\end{align}
The above relation is called \emph{Steiner's inequality} and although Steiner symmetrisation has been widely used in analysis for more than one hundred years, the study of the equality cases of \eqref{intro_Steiner inequality} and in particular the study of \emph{rigidity} of the equality cases for such inequality, is a difficult and important problem that still attracts the interest of the mathematical community. By \emph{rigidity} of equality cases we mean the situation where the only extremals of \eqref{intro_Steiner inequality} are equivalent to vertical translations of $E^s$. 
The first results in this direction are due to De Giorgi, who studied the equality cases of \eqref{intro_Steiner inequality} in its proof of the isoperimetric inequality (see \cite{DeGiorgi58ISOP,DeGiorgiSelected}). After some decades, the study of the \emph{rigidity} problem has been resumed and deeply studied by Chleb\'ik, Cianchi, and Fusco in \cite{ChlebikCianchiFuscoAnnals05}, where general sufficient conditions for \emph{rigidity} to hold true were found. More recently Cagnetti, Colombo, De Philippis, and Maggi in \cite{CagnettiColomboDePhilippisMaggiSteiner}, after finding a complete characterization for the equality cases of \eqref{intro_Steiner inequality}, were able to provide even more general sufficient conditions for \emph{rigidity}, and to fully characterize it for some specific situations. Thus, apart from  some cases, the \emph{rigidity} problem was left open (for a precise account of the results so far obtained we refer to \cite{CagnettiPortugaliae}).

\noindent
Let us mention that in the context of Steiner's inequality for the anisotropic perimeter, in \cite{PeruginiAnisotropo} the \emph{rigidity} problem has been studied and general sufficient conditions for which \emph{rigidity} in the anisotropic setting is equivalent to \emph{rigidity} in the Euclidean setting are presented. The \emph{rigidity} problem, in the context of perimeter inequalities, has been fully characterized for other types of symmetrisation procedures. We refer indeed to \cite{CagnettiPeruginiStoger}, and \cite{DomazakisSchwarz} for the spherical and Schwarz symmetrisation respectively, while in the context of Ehrhard's symmetrization inequality for Gaussian perimeter we cite \cite{ccdpmGAUSS}. In the context of Schwartz's
symmetrization inequality for the Dirichlet-type integral functionals, Brothers and Ziemer in \cite{brothersziemer} provided general sufficient conditions under which \emph{rigidity} in that setting holds true. Lastly, in \cite{CagnettirigidityPolyaSzego} Cagnetti proved that those sufficient conditions found by Brothers and Ziemer are actually also necessary.

\noindent
Goal of the present work is to tackle the \emph{rigidity} problem for Steiner's inequality in its full generality.

\noindent
With the aim of getting a clear geometric insight of the problem, we decided to simplify the complexity of the setting by localizing the characterization result for the equality cases of \eqref{intro_Steiner inequality} obtained in \cite[Theorem 1.9]{CagnettiColomboDePhilippisMaggiSteiner} over open and connected sets. This approach lead us to the introduction of a class of functions of bounded variation that we call \emph{minimally singular} (see Definition \ref{def_minimally singular functions}).


\noindent
Our main result is a complete geometric characterization of \emph{minimally singular} functions via the introduction of a pseudometric that we call \emph{singular vertical distance} (see Definition \ref{def: SVD}, and Theorem \ref{thm_characterization minimally singular via SVD}). As a by-product we establish a characterization result of \emph{rigidity} of equality cases for Steiner's inequality over open and connected sets (see Theorem \ref{thm_rigidity by perugini}).

\noindent
Let us start by introducing some of the main definitions about Steiner symmetrisation.
\subsection{Steiner's inequality: main definitions}\label{subsection_intro:equality cases}
Let $n \in \mathbb{N}$, with $n \geq 1$. We decompose $\mathbb{R}^{n+1}$ as $\mathbb{R}^{n} \times \mathbb{R}$, 
and we write a generic point of $\mathbb{R}^{n+1}$ as $(x,t)$, with $x \in \mathbb{R}^{n}$ and $t \in \mathbb{R}$. For every set $E \subset \mathbb{R}^{n+1}$ and for every $x\in \mathbb{R}^n$ we define the one dimensional section of $E$ at $x$ as
\[
E_{x} := \{ t \in \mathbb{R}:\, (x,t) \in E \}.
\]
Let $v : \mathbb{R}^{n} \to [0, \infty)$ be a Lebesgue measurable function. We say that $E\subset\mathbb{R}^{n+1}$ is \textit{$v$-distributed} if 
\[
v (x) = \mathcal{H}^1 (E_{x}) \quad \text{ for $\mathcal{L}^{n}$-a.e. } x 
\in \mathbb{R}^{n},
\]
where $\mathcal{H}^k$, and $\mathcal{L}^k$ stand for the $k$-dimensional Hausdorff measure and the $k$-dimensional Lebesgue measure respectively with $k\in \mathbb{N}\cup \{ 0 \}$. We define the set $F[v]\subset \mathbb{R}^{n+1}$ as 
\begin{align}\label{def_F[v]}
F[v] := \left\{ (x, t) \in  \mathbb{R}^{n+1} :\,  
 |t|<\frac{1}{2} v(x) \right\}.
\end{align}
If $E\subset \mathbb{R}^{n+1}$ is a $v$-distributed set, then $F[v]$ is its Steiner symmetric set with respect to the hyperplane 
\( \{ (x, t) \in \mathbb{R}^{n+1} : t = 0 \} = \mathbb{R}^{n}\times\{ 0 \} \). Our choice to use the notation $F[v]$ rather than $E^s$ is motivated by the fact that when studying the cases of equality of \eqref{intro_Steiner inequality}, we are interested in describing the analytical properties of the function $v$ rather than focus on a particular $v$-distributed set $E$. Relation \eqref{intro_Steiner inequality} in particular says that if $E$ is a $v$-distributed set of finite perimeter, then also $F[v]$ is a set of finite perimeter (for more details about sets of finite perimeter we refer to Section \ref{section_fundamentals}). Let us observe that thanks to \cite[Lemma 1.1]{ChlebikCianchiFuscoAnnals05} in case $E$ is a set of finite perimeter with $\mathcal{L}^n(E)=\infty$, then $F[v]$ coincides with the whole $\mathbb{R}^{n+1}$ whose perimeter is zero. So, motivated by this last consideration, since we are interested in the characterization of equality cases of \eqref{intro_Steiner inequality}, it is not restrictive to focus our attention to those  Lebesgue measurable functions $v:\mathbb{R}^{n}\to [0,\infty)$ such that both the perimeter and the volume of $F[v]$ are finite, namely $P(F[v])<\infty$ and $\mathcal{L}^n(F[v])<\infty$. In particular, we have the following characterization result (see \cite[Lemma 1.3]{CagnettiPortugaliae}, and \cite[Proposition 3.2]{CagnettiColomboDePhilippisMaggiSteiner}).
\begin{proposition} 
Let $v:\mathbb{R}^n\to [0,\infty)$ be a Lebesgue measurable function. Then $F[v]\subset \mathbb{R}^{n+1}$ is a set of finite perimeter and finite volume if and only if the following conditions hold true
\begin{align}\label{cond_ F[v] has finite perimeter}
v\in BV(\mathbb{R}^n;[0,\infty)) \textnormal{ and } \mathcal{L}^n(\{ v>0 \})<\infty.
\end{align}
\end{proposition}
\noindent
We denote with $\mathcal{M}(v)$ the class of all $v$-distributed sets $E$ having finite perimeter and attaining equality in \eqref{intro_Steiner inequality}, namely
\begin{align}\label{def_M(v)}
\mathcal{M}(v):=\left\{E\subset \mathbb{R}^{n+1}:\, E \textnormal{ is $v$-distributed and }P(F[v])=P(E)  \right\}.
\end{align}
We are now ready to define rigorously the concept of \emph{rigidity} in the context of Steiner's inequality. Let $v:\mathbb{R}^n\to [0,\infty)$ be a Lebesgue measurable function satisfying condition \eqref{cond_ F[v] has finite perimeter}. Then, we say that \emph{rigidity} holds true if and only if 
\begin{align}\label{rigidity classic definition}
\forall\,E \in \mathcal{M}(v)\textnormal{ there exists } t \in \mathbb{R}\textnormal{ so that }\mathcal{H}^{n+1}\left( E \Delta (t e_{n+1}+ F[v]) \right)=0,   
\end{align}
where $(e_1,\dots,e_{n+1})$ stands for the canonical basis of $\mathbb{R}^{n+1}$, and $A\Delta B$ stands for the symmetric difference of the sets $A,B\subset \mathbb{R}^{n+1}$. In other words, given $v:\mathbb{R}^n\to [0,\infty)$ a Lebesgue measurable function satisfying \eqref{cond_ F[v] has finite perimeter} consider the following geometric variational problem 
\begin{align*}
\inf \left\{P(E):\, E\subset \mathbb{R}^{n+1} \textnormal{ is $v$-distributed}     \right\}.
\end{align*}
Then Steiner's inequality \eqref{intro_Steiner inequality} guarantees the existence of at least a minimizer, that is $F[v]$, while studying \emph{rigidity} amounts to investigate uniqueness (up to vertical translation) for such minimizer. Thus for us the \emph{rigidity} problem for Steiner's inequality consists in finding the necessary and sufficient conditions that the function $v$ has to satisfy so that \eqref{rigidity classic definition} holds true.

\subsection{What we know about \emph{rigidity}}\label{subsection_intro:rigidity}

We are now ready to discuss more specifically about the \emph{rigidity} problem. First of all, we cannot expect that \emph{rigidity} holds true if we don't assume a certain notion of connectedness for the projection of the symmetric set $F[v]$ on $\mathbb{R}^n$, that is $\{ v>0 \}$ (see Figure \ref{fig_serveconnessione}). In the following $v\Upp$ and $v\Low$ are the approximate upper and lower limits of $v$, respectively (for more details we refer to Section \ref{section_fundamentals}).

\begin{figure}[!htb]
\centering
\def\svgwidth{11cm}
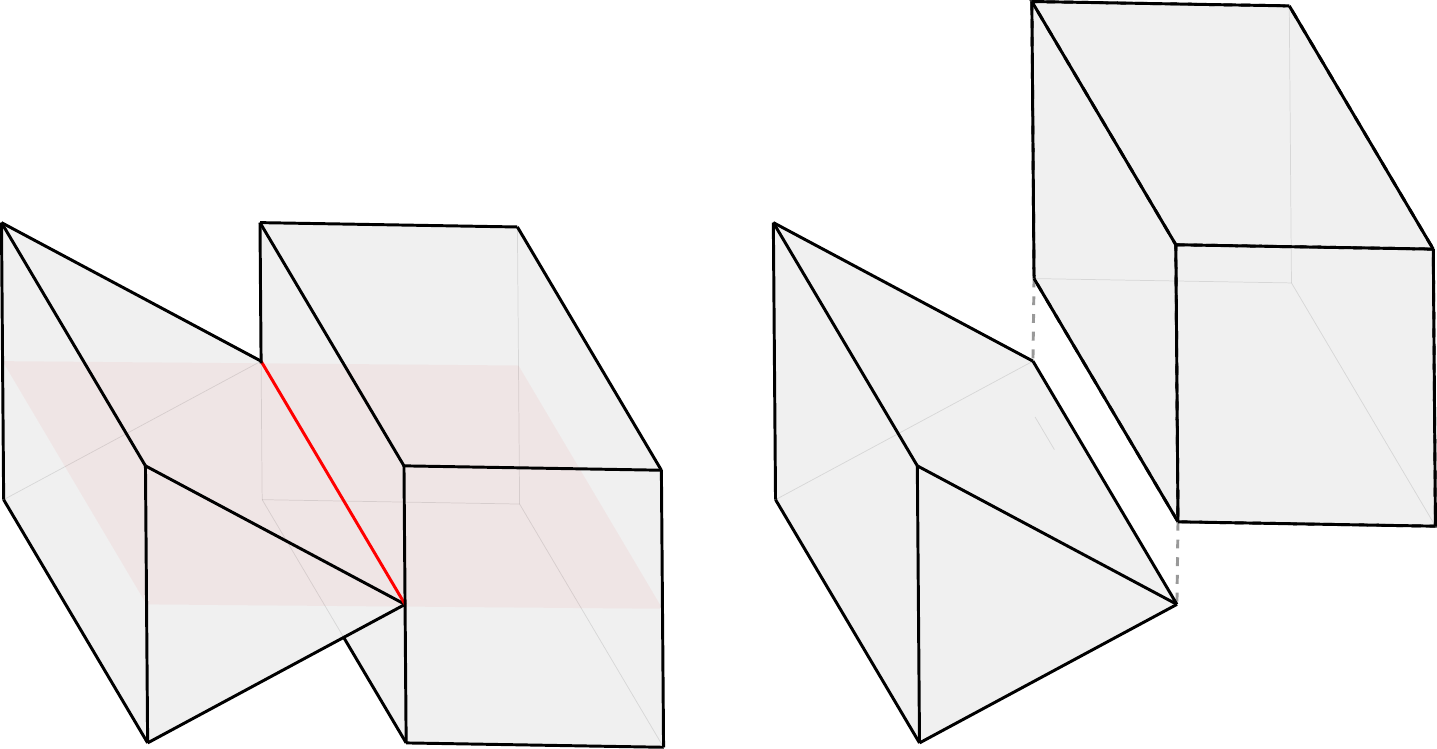
\caption{A pictorial representation in $\mathbb{R}^3$ of $\{ v\Low=0 \}$ (see the red line on the left) that essentially disconnects the projection of $F[v]$ on $\mathbb{R}^2$ that is $\{ v>0 \}$ (see the light red region on the left). Indeed the set $E$ on the right is a $v$-distributed set having the same perimeter of $F[v]$, thus in this situation \emph{rigidity} fails. Let us observe that there is no limitation on how much we could have lifted up the parallelepiped-block of the set $E$.}
\label{fig_serveconnessione}
\end{figure} 
\noindent
What is happening in the example presented in Figure \ref{fig_serveconnessione} is that the set $\{  v\Low=0\}$ \emph{essentially disconnects} the projection $\{ v>0 \}$ allowing to separately shift one of the two disconnected parts of $F[v]$ without affecting the perimeter. This notion of connectedness valid for Borel sets, was firstly introduced in \cite{ccdpmGAUSS} while studying the rigidity problem in the context of Ehrhard's symmetrization inequality for Gaussian perimeter. To be more precise, let $K, G \subset \mathbb{R}^m$, with $m\in \mathbb{N}$ be Borel sets. Then we say that $K$ \emph{essentially disconnects} $G$ if and only if there exists a non trivial Borel partition $\{ G_+,G_- \}$ of $G$ such that 
\begin{align*}
\mathcal{H}^{m-1}\left( (G^{(1)}\cap \partial ^{e}G_+ \cap \partial ^{e}G_-)\setminus K  \right)=0,
\end{align*}
where for any Lebesgue measurable set $A\subset \mathbb{R}^n$ we write $A^{(t)}$, with $0\leq t\leq 1$ to indicate the set of points of density $t$ of $A$, and we denote  with $\partial ^e A = \mathbb{R}^n\setminus (A^{(0)}\cup A^{(1)})$ the essential boundary of $A$, while by non trivial Borel partition we mean that the two sets $G_+,G_-$ satisfy the following conditions
\begin{align}\label{non trivial Borel partition}
\mathcal{H}^m(G_+ \cap G_- )=0,\quad \mathcal{H}^m(G \Delta (G_+\cup G_-))=0,\quad \mathcal{H}^m(G_+)\mathcal{H}^m(G_-)>0.
\end{align}
Conversely, we say that $K$ \emph{does not essentially disconnect} $G$ if and only if for every non trivial Borel partition  $\{ G_+,G_- \}$ of $G$ we have that 
\begin{align}\label{def_does not essentially disconnect}
\mathcal{H}^{m-1}\left( (G^{(1)}\cap \partial ^{e}G_+ \cap \partial ^{e}G_-)\setminus K  \right)>0.
\end{align}
Lastly, we say that $G$ is \emph{essentially connected} if the empty set does not essentially disconnect $G$. Thus, the following condition 
\begin{align}\label{cond_ F[v] is indecomposable}
\{ v\Low =0 \}\textnormal{ does not essentially disconnect }\{ v>0 \}
\end{align}
is a necessary condition for \emph{rigidity} to hold true (the reader interested on how this notion of connectedness relates to the notion of  \emph{indecomposability} for sets of finite perimeter can check \cite[Remark 2.3]{ccdpmGAUSS}, and \cite[Section 4D]{CagnettiColomboDePhilippisMaggiSteiner}). 

\noindent
Another aspect that we need to take into account when studying the \emph{rigidity} problem, is the role of the singular part of the distributional derivative of $v$, namely $D^s v$. Let us recall that with $Dv$, namely the distributional gradient  of $v$, we denote the $\mathbb{R}^n$-valued Radon measure on $\mathbb{R}^n$, and $Dv= D^a v + D^s v$, where $D^a v$ and $D^s v$ are the absolutely continuous and singular part of $Dv$, respectively. We also recall that $D^s v= D^j v + D^c v$ where $D^j v$ and $D^c v$ are the jump and Cantor part of $D^s v$, respectively. As it is also clear from the formula of the perimeter of $F[v]$ (see \cite[Theorem 3.1]{CagnettiColomboDePhilippisMaggiSteiner}) the total variation of $D^s v$ evaluated over $\mathbb{R}^n$, namely $|D^s v|(\mathbb{R}^n)$ coincides with the $\mathcal{H}^{n}$-measure of the parts of the reduced boundary $\partial^* F[v]$ of $F[v]$ having the measure-theoretic outer unit normal $\nu^{F[v]}$ that is parallel with respect to the symmetrisation hyperplane: roughly speaking we refer to these parts of the reduced boundary as ``the vertical parts'' (see for instance Figure \ref{fig_vertical parts}).

\begin{figure}[!htb]
\centering
\def\svgwidth{11cm}
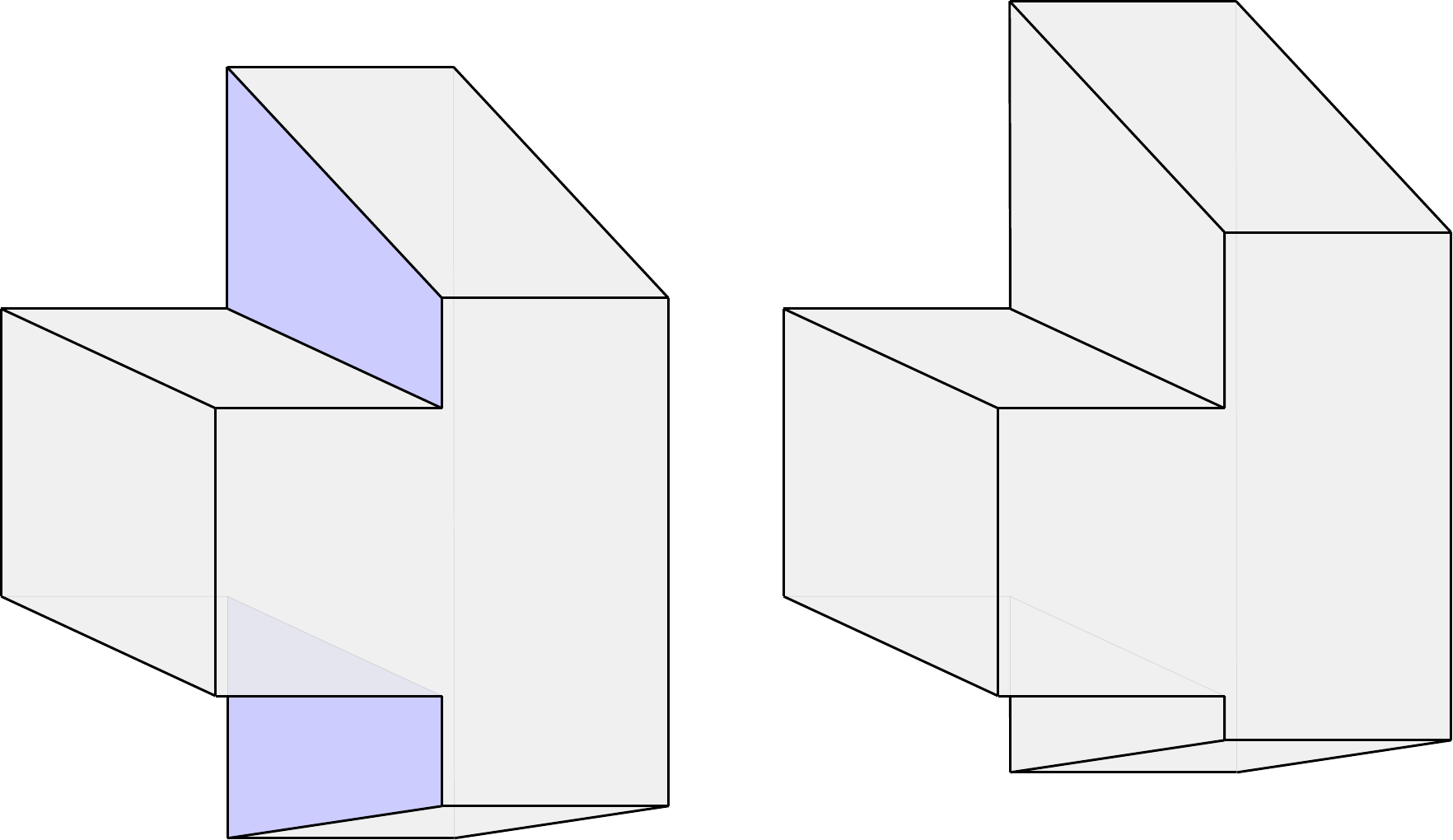
\caption{A pictorial representation in $\mathbb{R}^3$ of how some of the vertical parts of the reduced boundary of $F[v]$ (see the light blue parts on the left) may be used to create counterexamples to \emph{rigidity} (see the set $E$ on the right). Note that the set $E$ has been obtained by vertically shifting upward an entire piece of $F[v]$. In this case $P(F[v])=P(E)$, thus $E\in \mathcal{M}(v)$, but since $E$ and $F[v]$ are not $\mathcal{H}^{n+1}$-equivalent up to vertical translations, \emph{rigidity} fails.}
\label{fig_vertical parts}
\end{figure} 
\noindent
Roughly speaking we can say that, in the example presented in Figure \ref{fig_vertical parts}, the vertical parts of $\partial^* F[v]$ that appear in light blue, create a sort of vertical surface, along which we can slide an entire chunk of $F[v]$ without effecting the perimeter.

\noindent
In the seminal paper \cite{ChlebikCianchiFuscoAnnals05},  Chleb\'ik, Cianchi and Fusco proved the following important result (see \cite[Theorem 1.3]{ChlebikCianchiFuscoAnnals05}, and \cite[Theorem B]{CagnettiColomboDePhilippisMaggiSteiner}) which provides general sufficient conditions under which \emph{rigidity} holds true. 

\begin{theorem}\label{thm_rigidity by CCF}
Let $v:\mathbb{R}^n\to [0,\infty)$ be a Lebesgue measurable function satisfying \eqref{cond_ F[v] has finite perimeter}. Let $\Omega\subset \mathbb{R}^n$ be an open and connected set such that both the following conditions hold true, namely
\begin{align*}
D^s v \mres \Omega= 0,\quad\textnormal{and}\quad v\Low(x)>0\textnormal{ for }\mathcal{H}^{n-1}\textnormal{-a.e. }x\in \Omega.
\end{align*}
Then for every $E\in \mathcal{M}(v)$ there exists $t \in \mathbb{R}$ such that $\mathcal{H}^{n+1}\left( \left(E \Delta (t e_{n+1}+ F[v])\right) \cap (\Omega\times \mathbb{R}) \right)=0$. Moreover, if there exists an $\Omega\subset \mathbb{R}^n$ as above such that $\mathcal{L}^n( \{ v>0 \}\setminus \Omega)=0$ then for every $E\in \mathcal{M}(v)$ there exists $t \in \mathbb{R}$ such that $\mathcal{H}^{n+1}\left( \left(E \Delta (t e_{n+1}+ F[v])\right) \right)=0$.
\end{theorem}  
\noindent
The condition $v\Low (x)>0$ for $\mathcal{H}^{n-1}$-a.e. $x\in \Omega$ guarantees that $\{ v\Low =0 \}$ does not essentially disconnect $\Omega$, while requiring that $D^s v \mres \Omega=0$ (i.e. having that $v\in W^{1,1}(\Omega)$) guarantees that $\partial^* F[v]\cap (\Omega\times\mathbb{R})$ has no ``vertical parts''. For $n\geq 2$ these two conditions are far from being necessary conditions for \emph{rigidity} (the situation is indeed different for $n=1$ as shown in Theorem \ref{thm_rigidity when n=1}). In the attempt of weakening such strong requirements, in \cite{CagnettiColomboDePhilippisMaggiSteiner} the authors proved the following result that significantly improves Theorem \ref{thm_rigidity by CCF} (see \cite[Theorem 1.11]{CagnettiColomboDePhilippisMaggiSteiner}). In the following $S_v$ stands for the jump set of $v$.

\begin{theorem}\label{thm_rigidity CCDFM suffcient}
Let $v:\mathbb{R}^n\to [0,\infty)$ be a Lebesgue measurable function satisfying \eqref{cond_ F[v] has finite perimeter}. Assume in addition that $D^c v$ is concentrated in a Borel set $K$ such that
\begin{align}\label{eq_CCDFM suff rigidity}
S_v\cup K \cup \{ v\Low=0 \} \textnormal{ does not essentially disconnect }\{ v>0 \}.
\end{align}
Then for every $E\in \mathcal{M}(v)$ there exists $t \in \mathbb{R}$ such that $\mathcal{H}^{n+1}\left( \left(E \Delta (t e_{n+1}+ F[v])\right) \right)=0$.
\end{theorem}
\noindent
Roughly speaking the above result does not ask for the singular part of $D v$ to be zero, but it requires  that the set where $D^s v$ is concentrated together with $\{ v\Low = 0 \}$, namely $S_v \cup K \cup \{ v\Low =0 \}$, to be ``not too big'' in the sense of not essentially disconnecting $\{ v>0 \}$. 
\noindent
Although the above result represents a remarkable improvement in the understanding of the \emph{rigidity} problem, the conditions assumed in the statement are still not necessary, but merely sufficient. In order to obtain proper characterization results of \emph{rigidity}, in \cite{CagnettiColomboDePhilippisMaggiSteiner}  the authors restricted their study to different specific scenarios depending on the choice of the function $v$ and on the space dimension $n\geq 1$. In particular for $n=1$ they obtained the following complete characterization result (see \cite[Theorem 1.30]{CagnettiColomboDePhilippisMaggiSteiner} for the original statement).


\begin{theorem}\label{thm_rigidity when n=1}
Let $v:\mathbb{R}\to [0,\infty)$ be a Lebesgue measurable function satisfying \eqref{cond_ F[v] has finite perimeter}. Then the following statements are equivalent:
\begin{itemize}
\item[i)] for every $E \in \mathcal{M}(v)$ there exists $t \in \mathbb{R}$ so that $\mathcal{H}^{2}\left( E \Delta (t e_{2}+ F[v])\right) =0$;\vspace*{0.1cm}
\item[ii)] $\{ v>0 \}$ is $\mathcal{H}^1$-equivalent to a bounded open interval $(a,b)$, $v\in W^{1,1}((a,b))$, and $v\Low (x)>0$ for every $x\in (a,b)$.
\end{itemize}
\end{theorem}

\begin{remark}
Let us note that having $v(x)\Low >0$ for every $x \in (a,b)$ is equivalent to condition \eqref{cond_ F[v] is indecomposable}.
\end{remark}

\noindent
Instead for $n\geq 2$, a complete characterization of \emph{rigidity} was still possible only for those functions of bounded variation $v:\mathbb{R}^n\to [0,\infty)$ such that $D^c v= 0$, and $S_v\cap \{ v\Low >0 \}$ is locally $\mathcal{H}^{n-1}$-finite (for more details in this direction we suggest to see \cite[Theorem 1.13, Theorem 1.20, and Theorem 1.29]{CagnettiColomboDePhilippisMaggiSteiner}). Thus, apart from some specific situations, the characterization of \emph{rigidity} for Steiner's inequality was left open.

\subsection{A localized version of \emph{rigidity}}

As observed so far, what is really important in the process of characterizing \emph{rigidity}, is getting a clear picture of the following two problems: on the one hand understanding in which way the set $\{ v\Low = 0 \}$ disconnects the projection $\{ v>0 \}$ of $F[v]$, while on the other hand comprehending the role of the ``vertical parts'' of $\partial^*F[v]$
in disconnecting the remaining part of $\partial ^* F[v]$. If the role of $\{ v\Low=0 \}$ disconnecting $\{ v>0 \}$ was made clear through the introduction of essential connectedness, less is known about the second aspect regarding the role played by the ``vertical parts'' of $\partial^*F[v]$, which as we noticed before, are deeply connected with $D^s v$.

\noindent
Thus, in order to have a much more clear insight of the geometric nature of \emph{rigidity}, we decided to restrict our study on this latter aspect. Motivated by this task, we decided to localize the study of \emph{rigidity} over a particular class of sets. 
\noindent
Before entering into details, let us recall the localized version of Steiner's inequality, and correspondingly the localized version of \emph{rigidity}.  Let $E\subset \mathbb{R}^{n+1}$ be a set of finite perimeter, and let $G\subset \mathbb{R}^{n+1}$ be a Borel set, we denote with $P(E;G)$ the relative perimeter of $E$ with respect to $G$ (see Section \ref{section_fundamentals} for more details). If $G=\mathbb{R}^n$ we simply write $P(E)$ instead of  $P(E;\mathbb{R}^n)$. Let $v:\mathbb{R}^n\to [0,\infty)$ be a Lebesgue measurable function satisfying \eqref{cond_ F[v] has finite perimeter}. Then, as proved in \cite[Lemma 3.4]{ChlebikCianchiFuscoAnnals05} (see also \cite[Theorem 5.9]{PeruginiCircolare}) for every $v$-distributed set $E\subset \mathbb{R}^{n+1}$ of finite perimeter we have the following localized version of Steiner's inequality
\begin{align}\label{Steiner's inequality}
P(F[v];B\times\mathbb{R}) \leq P(E;B\times\mathbb{R})\quad \forall\, B\subset \mathbb{R}^n \textnormal{ Borel.}
\end{align}
Given $B\subset \mathbb{R}^n$ Borel set, we set
\begin{align}\label{def_M_Omega(v)}
\mathcal{M}_{B}(v):= \left\{E\subset \mathbb{R}^{n+1}:\, E\textnormal{ is }v\textnormal{-distributed and }P(F[v];B\times\mathbb{R}) = P(E;B\times\mathbb{R})   \right\}.
\end{align} 
If $B=\mathbb{R}^n$ we simply write $\mathcal{M}(v)$ instead of $\mathcal{M}_{\mathbb{R}^n}(v)$, which coincides with \eqref{def_M(v)}. Now, 
let $\Omega\subset \mathbb{R}^n$ be an open set.
Motivated by the notation we have just introduced we say that \emph{rigidity} over $\Omega$ holds true if and only if 
\begin{align}\label{rigidity over Omega}
\forall\,E \in \mathcal{M}_\Omega(v)\textnormal{ there exists } t \in \mathbb{R}\textnormal{ so that }\mathcal{H}^{n+1}\left( \left(E \Delta (t e_{n+1}+ F[v])\right) \cap (\Omega\times \mathbb{R}) \right)=0. 
\end{align}
A first step in the study of the \emph{rigidity} problem, is the characterization of the equality cases for Steiner's inequality. Available in the literature there are two characterization results of the extremals of \eqref{Steiner's inequality} a geometric one and an analytical one. The geometric one (see \cite[Theorem 5.9]{PeruginiCircolare}) characterizes every set $E\in\mathcal{M}_B(v)$, with $B\subset \mathbb{R}^n$ any Borel set, in terms of the symmetric properties of its measure-theoretic inner unit normal defined over its reduced boundary, while the analytical one (see \cite[Theorem 1.9]{CagnettiColomboDePhilippisMaggiSteiner}) characterizes every set $E\in\mathcal{M}(v)$ in terms of the fine properties of its barycenter function. We found the latter characterization result particularly useful in understanding \emph{rigidity}, and for this reason in the following we are going to present a localized version of that result.
Let us start by introducing the barycenter function. As already shown by Ennio De Giorgi (see \cite{DeGiorgi58ISOP,DeGiorgiSelected}), a necessary condition for a set $E$ to belong to $\mathcal{M}(v)$ is to have its vertical sections equivalent to segments. This property holds true also in the localized setting (see for instance \cite[Theorem 5.9]{PeruginiCircolare}), namely
given $\Omega\subset \mathbb{R}^n$ open set, for every $E\in \mathcal{M}_\Omega(v)$ we have that
\begin{align}\label{sections as segments}
E_x\textnormal{ is }\mathcal{H}^1\textnormal{-equivalent to a segment for }\mathcal{L}^n\textnormal{-a.e. }x\in \Omega.
\end{align}
Sets satisfying the above condition can be uniquely associated with a function that describes the position of the mid point of their vertical slices. To be more precise let $\Omega\subset \mathbb{R}^n$ be an open set,
 and let $E$ be a $v$-distributed set satisfying \eqref{sections as segments}, then we define its barycenter function $b_{E,\Omega}:\Omega\to \mathbb{R}$ as 
\begin{align}\label{baricenter}
b_{E,\Omega}(x):=
\begin{cases}
\displaystyle\frac{1}{v(x)} \int_{E_x}t\,dt \quad &\mbox{if }x\in \Omega\cap \{0<v<\infty\},\vspace*{0.2cm}\\
0\quad &\mbox{otherwise}.
\end{cases}
\end{align}
If $\Omega=\mathbb{R}^n$ we simply write $b_E$ instead of $b_{E,\mathbb{R}^n}$. The barycenter function is a powerful tool for the study of the \emph{rigidity} problem. Indeed  \emph{rigidity} over $\Omega$ holds true if and only if for every $E\in \mathcal{M}_\Omega(v)$ its barycenter function is constant in $\Omega\cap \{ v>0 \}$, namely
\begin{align}\label{rigidity via the baricenter}
\forall\, E\in\mathcal{M}_\Omega(v) \textnormal{ there exists }c\in\mathbb{R} \textnormal{ such that } b_{E,\Omega}(x)=c \textnormal{ for }\mathcal{L}^n\textnormal{-a.e. }x\in\Omega\cap \{ v>0 \}. 
\end{align}
As we mentioned at the beginning of this section,  an important ingredient of our strategy to understand the role that $D^s v$ plays in spoiling \emph{rigidity} is to focus our analysis over a particular class of open sets, more precisely we consider those sets $\Omega$ such that
\begin{align}\label{CCF condition for Omega}
\Omega\subset \mathbb{R}^n \textnormal{ is open and connected, and }   v\Low (x)>0\textnormal{ for }\mathcal{H}^{n-1}\textnormal{-a.e. }x\in \Omega.
\end{align} 
Indeed, working over a set $\Omega$ as in \eqref{CCF condition for Omega}, reduces the complexity of the characterization result for sets in $\mathcal{M}_\Omega(v)$, allowing to get a better insight on \emph{rigidity}. Before presenting the localized version of an important result proved in \cite{CagnettiColomboDePhilippisMaggiSteiner}, let us briefly introduce some notation (for the precise definitions we refer to Section \ref{section_fundamentals}). We denote with $[b_{E,\Omega}]:=b_{E,\Omega}\Upp-b_{E,\Omega}\Low$ the jump of $b_{E,\Omega}$, and $S_{b_{E,\Omega}}$ is the jump set of $b_{E,\Omega}$. We recall that a function $f:\Omega\to \mathbb{R}$ belongs to $GBV(\Omega)$, where $GBV(\Omega)$ stands for the space of generalised functions of bounded variation, if and only if for every $M>0$ the truncated function $f^M(x):=\max \{-M, \min\{ M,f(x) \} \}$ is a function of locally bounded variation in $\Omega$. If $f\in GBV(\Omega)$ we denote with $\nabla^* f$ its approximate gradient.  A careful inspection of the proofs of both  \cite[Theorem 1.7, Theorem 1.9]{CagnettiColomboDePhilippisMaggiSteiner}  leads to the following result (see \cite[Theorem 1.7, Theorem 1.9]{CagnettiColomboDePhilippisMaggiSteiner} for the original statements with $\Omega=\mathbb{R}^n$).
\begin{proposition}\label{prop_equality cases over open sets}
Let $v:\mathbb{R}^n\to [0,\infty)$ be a Lebesgue measurable function satisfying \eqref{cond_ F[v] has finite perimeter}, and let $\Omega\subset \mathbb{R}^n$ be as in \eqref{CCF condition for Omega}. Then $E\in \mathcal{M}_\Omega(v)$ if and only if the following conditions hold true
\begin{align}
&i)\; E_x \textnormal{ is } \mathcal{H}^1\textnormal{-equivalent to a segment for }\mathcal{L}^n\textnormal{-a.e. }x\in \Omega,\textnormal{ and }b_{E,\Omega}  \in GBV(\Omega); \\
&ii)\;\nabla^* b_{E,\Omega}(x)=0,\quad \textnormal{for } \mathcal{L}^{n}\textnormal{-a.e. }x\in\Omega;\label{eq_nabla b for E = 0 PROP}\\
&iii)\;[b_{E,\Omega}]\leq 1/2 [v], \quad \mathcal{H}^{n-1}\textnormal{-a.e. on }\Omega; \label{eq_Salto b< 1/2 salto v for E PROP}\\
&iv)\;\textnormal{there exists a Borel function } f:\Omega\rightarrow [-1/2,1/2]\textnormal{ such that }  \nonumber\\
 & \quad D^c\left( b_{E,\Omega}^M \right)(B)= \int_{B\cap \{ |b_{E,\Omega}|<M \}^{(1)}}f(x)\,dD^c v(x),\label{eq_Cantor b< 1/2 Cantor v for E PROP}\quad \forall\,B\subset\Omega \textnormal{ bounded Borel set,}\\ &\quad \textnormal{ and for }\mathcal{L}^1\textnormal{-a.e. } M>0\nonumber.
\end{align}
\end{proposition}

\subsection{Minimally singular functions and main results}\label{subsection_intro:main results}
Thanks to the above result, and motivated by the characterization of \emph{rigidity} in terms of the properties of the barycenter function of a generic set in $\mathcal{M}_\Omega(v)$ (see \eqref{rigidity via the baricenter}) we decided to introduce the following class of functions. From now on it will be often convenient to consider $\Omega\subset \mathbb{R}^n$ to be a set satisfying the following properties
\begin{align}\label{assumptions on Omega}
\Omega\subset \mathbb{R}^n \textnormal{ is open and connected with }\mathcal{L}^n(\Omega)<\infty.
\end{align}
\begin{definition}\label{def_minimally singular functions}
Let $\Omega\subset \mathbb{R}^n$ be as in \eqref{assumptions on Omega}, let $u \in BV(\Omega)$, and let $b\in GBV(\Omega)$ be any function that satisfies the following conditions 
\begin{align}
&i)\;\nabla^* b(x)=0,\quad \textnormal{for } \mathcal{L}^{n}\textnormal{-a.e. }x\in\Omega;\label{eq_nabla b = 0}\\
&ii)\;[b]\leq [u], \quad \mathcal{H}^{n-1}\textnormal{-a.e. on }\Omega; \label{eq_Salto b< 1/2 salto v}\\
&iii)\;\textnormal{there exists a Borel function } f:\Omega\rightarrow [-1,1]\textnormal{ such that }  \nonumber\\
 & \quad D^c\left( b^M \right)(B)= \int_{B\cap \{ |b|<M \}^{(1)}}f(x)\,dD^cu(x),\label{eq_Cantor b< 1/2 Cantor v}\quad \forall\,B\subset\Omega \textnormal{ bounded Borel set,}\\ &\quad \textnormal{ and for }\mathcal{L}^1\textnormal{-a.e. } M>0.\nonumber
\end{align}
Then we say that $u\in BV(\Omega)$ is \emph{minimally singular} if and only if for every $b\in GBV(\Omega)$ satisfying \eqref{eq_nabla b = 0}, \eqref{eq_Salto b< 1/2 salto v} and \eqref{eq_Cantor b< 1/2 Cantor v},
there exists $c\in\mathbb{R}$ such that $b(x)=c$ for $\mathcal{L}^n$-a.e. $x\in\Omega$.
\end{definition}
\noindent
Roughly speaking, we say that $u\in BV(\Omega)$ is \emph{minimally singular} if and only if $D^s u$ is ``small enough'' to prevent the existence of any  non-constant function $b \in GBV(\Omega)$ whose weak distributional gradient $D b= D^s b$ is controlled, in the sense of the conditions listed in Definition \ref{def_minimally singular functions}, by $D^su$ (see for instance Figure \ref{fig_not minimally singular} and Figure \ref{fig_minimally singular} for an example of a not \emph{minimally singular} and of a \emph{minimally singular} function, respectively). In the following, given $\Omega\subset \mathbb{R}^n$ open set and $u:\Omega \to \mathbb{R}$ a Lebesgue measurable function, we denote by $\Sigma^u$ the subgraph of $u$, which is defined as 
$$
\Sigma^u:=\left\{(x,t)\in \Omega\times\mathbb{R}:\, t<u(x) \right\}.
$$
Recall that if $u\in BV(\Omega)$ then $\Sigma^u$ is a set of locally finite perimeter in $\Omega\times\mathbb{R}$, and $\partial^* \Sigma^u\subset \Omega\times\mathbb{R}$ denotes its reduced boundary.

\begin{figure}[!htb]
\centering
\def\svgwidth{14cm}
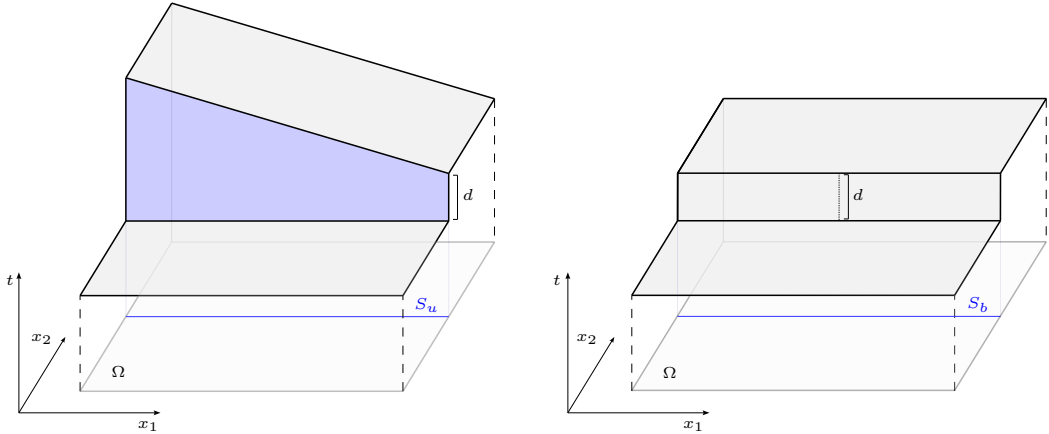
\caption{On the left, a pictorial representation in $\mathbb{R}^3$ of the subgraph of a function $u\in BV(\Omega)$ which is not \emph{minimally singular}. In light blue we have the vertical parts of $\partial^* \Sigma^u$, while in grey we have the remaining part of $\partial^* \Sigma^u$. On the right, in grey we can see $\partial^* \Sigma^b$, where $b\in BV(\Omega)$. Note that $b$ is a non constant function satisfying the three conditions listed in Definition \ref{def_minimally singular functions} thus $u$ is not \emph{minimally singular}. Note that the vertical parts of $\partial^* \Sigma^u$ somehow ``disconnect'' the remaining part of $\partial^* \Sigma^u $.}
\label{fig_not minimally singular}
\end{figure}

\noindent
Let us immediately observe that in case $\{ v\Low > 0 \}\subset \mathbb{R}^n$ is $\mathcal{H}^{n-1}$-equivalent to an open and connected set $\Omega$ it can be proved that \emph{rigidity} over $\mathbb{R}^n$ holds true if and only if $v\in BV(\Omega)$ is \emph{minimally singular} (a proof of this fact is contained in the argument of Theorem \ref{thm_rigidity by perugini}).

\noindent
The geometric intuition behind the ideas that will follow is based on the fact that if we want \emph{rigidity} over $\Omega$ to hold true we basically have to require that the vertical parts of $\partial^* F[v]\cap (\Omega\times\mathbb{R})$ ``do not disconnect'' the remaining part of $\partial^* F[v]\cap (\Omega\times\mathbb{R})$.


\noindent
Based on the original ideas of Vol'pert of restricting $BV$ functions over $1$-dimensional sections (see \cite{Volpert}, and \cite[Chapter 3.11]{AFP}), given $u\in BV(\Omega)$, with $\Omega\subset \mathbb{R}^n$ as in \eqref{assumptions on Omega}, we introduce a notion of distance between two distinct points of $\Omega$ with respect to the function $u$. Before presenting this definition, let us briefly introduce some notation. In the following $\gamma: I_\gamma \to \Omega$ is a polygonal chain (see relation \eqref{piecewise affine curves}), where $ I_\gamma\subset [0,\infty)$ is a closed and bounded interval, and $I^\circ_\gamma$ is the interior of $I_\gamma$. We define $u_\gamma: I_\gamma \to [0,\infty)$ as the restriction of $u$ to $\gamma$ namely $u_\gamma(t):= u\Low(\gamma(t))$ for every $t\in I_\gamma$ (for more details we refer to Section \ref{section_restriction BV over curves}, and to Section \ref{section_SVD}). Let $\Omega\subset \mathbb{R}^n$ be as in \eqref{assumptions on Omega}, and let $u\in BV(\Omega)$. Given $x_1, x_2 \in \Omega$, we call \emph{singular vertical distance} between $x_1$ and $x_2$ with respect to $u$ in $\Omega$, the quantity defined as
\begin{align*}
\textnormal{SVD}_{u,\Omega}(x_1,x_2):=
\begin{cases}
\inf \left\{|D^s u_\gamma|(I_\gamma^\circ):\, \gamma(I_\gamma)\subset \Omega,\,\gamma\textnormal{ connects }x_1,x_2  \right\}\quad &\mbox{if } x_1\neq x_2;\vspace*{0.2cm}\\
0 \quad &\mbox{if }x_1=x_2;
\end{cases}
\end{align*}
where the infimum is taken over all polygonal chains $\gamma$ connecting  $x_1$ and $x_2$ such that $u_\gamma\in BV(I_\gamma^\circ)$ (for a precise definition we refer to Definition \ref{def: SVD}, see also Figure \ref{fig_minimally singular} for a pictorial representation of the ideas so far explained). 

\begin{figure}[!htb]
\centering
\def\svgwidth{12cm}
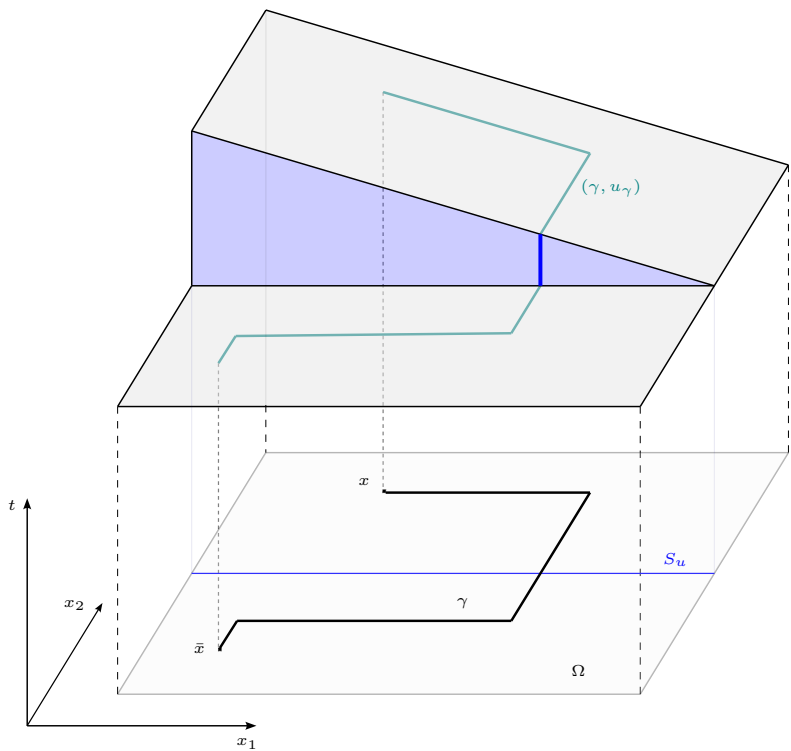
\caption{A pictorial representation in $\mathbb{R}^3$ of the idea at the hearth of the notion of \emph{singular vertical distance} between $\bar{x}$ and $x$ for a given $u\in BV(\Omega)$. In bold black we can see a polygonal chain $\gamma$ connecting $\bar{x}$ and $x$ in $\Omega$. In green  (appearing on $\partial^* \Sigma^u$) we drew $(\gamma, u_\gamma)$ where $u_\gamma$ is the restriction of $u\Low$ to $\gamma$. Note that the length of the bold blue segment equals $|D^s u_\gamma|(I_\gamma^\circ)$. It is not difficult to convince ourselves that $\textnormal{SVD}_{u,\Omega}(\bar{x},x)=0$ for $\mathcal{L}^n$-a.e. $x\in\Omega$, and in particular thanks to Theorem \ref{thm_characterization minimally singular via SVD}, we have that $u$ is \emph{minimally singular}. Let us stress that the condition of having $\textnormal{SVD}_{u,\Omega}(\bar{x},x)=0$ for $\mathcal{L}^n$-a.e. $x\in\Omega$ can be geometrically interpreted as having that the vertical parts of $\partial^* \Sigma^u $ ``do not disconnect'' the remaining part of $\partial^* \Sigma^u$.}
\label{fig_minimally singular}
\end{figure}

\noindent
Our main result is the following characterization of \emph{minimally singular} functions via the notion of \emph{singular vertical distance}. 

\begin{theorem}\label{thm_characterization minimally singular via SVD}
Let $\Omega\subset\mathbb{R}^n$ be as in  \eqref{assumptions on Omega} and let $u\in BV(\Omega)$. Then $u$ is \emph{minimally singular} if and only if there exists $\bar{x}\in \Omega$ such that $\textnormal{SVD}_{u,\Omega}(\bar{x},x)=0$ for $\mathcal{L}^n$-a.e. $ x\in \Omega$.
\end{theorem}
\noindent
In other words, we have that $u$ is \emph{minimally singular} in $\Omega$ if and only if there exists a point $\bar{x}\in \Omega$ with the following property. For $\mathcal{L}^n$-a.e. $x\in \Omega$ there exists a sequence of polygonal chains connecting $\bar{x}$ and $x$ along which the total variation of the singular part of the distributional derivative of the restriction of $u\Low$ to those curves, is arbitrary small. Let us immediately note that if $n=1$ then $u\in BV(\Omega)$ is \emph{minimally singular} if and only if $u$ is absolutely continuous (see Proposition \ref{prop_n=1}).

\noindent
We now apply the ideas so far explained to provide a characterization result of \emph{rigidity} of equality cases for Steiner's inequality. Our result is the following.

\begin{theorem}\label{thm_rigidity by perugini}
Let $v:\mathbb{R}^n\to [0,\infty)$ be a Lebesgue measurable function satisfying \eqref{cond_ F[v] has finite perimeter}, and let $\Omega \subset \mathbb{R}^n$ be as in \eqref{CCF condition for Omega}. Then the following statements are equivalent:
\begin{itemize}
\item[i)] for every $E \in \mathcal{M}_\Omega(v)$ there exists $t \in \mathbb{R}$ so that $\mathcal{H}^{n+1}\left( \left(E \Delta (t e_{n+1}+ F[v])\right) \cap (\Omega\times \mathbb{R}) \right)=0$;\vspace*{0.1cm}
\item[ii)] $v\in BV(\Omega)$ is \textnormal{minimally singular}, that is there exists $\bar{x}\in \Omega$ such that $\textnormal{SVD}_{u,\Omega}(\bar{x},x)=0$ for $\mathcal{L}^n$-a.e. $ x\in \Omega$.
\end{itemize}
Moreover, if there exists $\Omega\subset \mathbb{R}^n$ as in \eqref{CCF condition for Omega} such that $\mathcal{H}^{n-1}(\{ v\Low>0 \}\setminus \Omega  )=0$, then statement $i)$ can be substituted with the following one
\begin{itemize}
\item[$\textnormal{i}^*)$] for every $E \in \mathcal{M}(v)$ there exists $t \in \mathbb{R}$ so that $\mathcal{H}^{n+1}\left( E \Delta (t e_{n+1}+ F[v]) \right)=0$.
\end{itemize}
\end{theorem}
\noindent
Let us stress that, likewise in the statement of Theorem \ref{thm_rigidity by CCF}, requiring that $\Omega$ is as in \eqref{CCF condition for Omega} guarantees that $\{ v\Low = 0 \}$ does not essentially disconnect $\Omega$.

\begin{remark}
As already mentioned above, in \cite{CagnettiColomboDePhilippisMaggiSteiner} the authors managed to prove a characterization result of \emph{rigidity} for Steiner's inequality whenever $v:\mathbb{R}^n \to [0,\infty)$ is a special function of bounded variation that satisfies \eqref{cond_ F[v] has finite perimeter} and it is such that $S_v\cap \{ v\Low > 0 \}$ is $\mathcal{H}^{n-1}$-locally finite. Such characterization was given in terms of a geometric property that the function $v$ had to verify which they called \emph{mismatched stairway property} (see \cite[Definition 1.26]{CagnettiColomboDePhilippisMaggiSteiner}). Thus, if $\{ v\Low >0 \}$ is an open and connected set with $\mathcal{L}^n(\{ v\Low >0 \})< \infty$, combining Theorem \ref{thm_rigidity by perugini} with their characterization result  \cite[Theorem 1.29]{CagnettiColomboDePhilippisMaggiSteiner} we get that the function $v$ satisfies the \emph{mismatched stairway property} if and only if $v$ is \emph{minimally singular} in $\{ v\Low>0 \}$.
\end{remark}

\noindent
Finally, we show that \emph{minimally singular} functions with null gradient are necessarily $\mathcal{L}^n$-equivalent to constant functions (see Proposition \ref{prop_characterization of constant functions}, and Corollary \ref{cor_rigidity for constant functions}). Let us conclude with some final comments.
\subsection{Final comments}\label{subsection_intro:final comments}
The existence of a set $\Omega$ satisfying the conditions mentioned in the second part of Theorem \ref{thm_rigidity by perugini} namely,
\begin{align}\label{cond_conjecture on Omega}
\Omega\subset \mathbb{R}^n \textnormal{ open and connected}, \textnormal{ and }\mathcal{H}^{n-1}(\{ v\Low >0 \}\Delta \Omega)=0,
\end{align}
is unfortunately something that one cannot take for granted. Next result shows that if $n=1$ such a set $\Omega$ actually exists, providing in this way, together with Theorem \ref{thm_rigidity by perugini}, an alternative proof of Theorem \ref{thm_rigidity when n=1}.

\begin{lemma}\label{lem_open sets exists for n=1}
Let $v:\mathbb{R}\to [0,\infty)$ be a Lebesgue measurable function such that $\mathcal{L}^1(\{  v>0\})<\infty$, and such that \eqref{cond_ F[v] is indecomposable} holds true. Then there exists an open and bounded interval $(a,b)\subset \mathbb{R}$ such that $\{ v\Low > 0 \}= (a,b)$.
\end{lemma}

\noindent
If for $n=1$, as shown by Lemma \ref{lem_open sets exists for n=1}, the existence of a set $\Omega$ as in \eqref{cond_conjecture on Omega} is a necessary condition for \emph{rigidity} to hold true, for $n\geq 2$ this is not true any more. More precisely we have the following result.

\begin{proposition}\label{prop_no open set equivalency}
Let $n\geq 2$. Then there exists a Lebesgue measurable function $v:\mathbb{R}^n \to [0,\infty)$ satisfying both \eqref{cond_ F[v] has finite perimeter} and \eqref{cond_ F[v] is indecomposable}, with the property that \emph{rigidity} holds true, and for every open set $\Omega\subset \mathbb{R}^n$ we have that $\mathcal{L}^n(\{ v>0 \}\Delta \Omega)>0$.
\end{proposition}




\noindent
However, let us note that condition \eqref{cond_ F[v] is indecomposable} together with asking that the topological boundary of $\{ v\Low >0 \}$ is contained in $\{ v\Low = 0 \}$ up to a set of $\mathcal{H}^{n-1}$-measure zero, namely 
\begin{align}\label{cond_la chiusa di v = 0 meno v = 0 minuscola}
\mathcal{H}^{n-1}(\partial\{ v\Low > 0 \}\setminus \{ v\Low =0 \})=0
\end{align}
is sufficient to get the existence of a set $\Omega$ satisfying \eqref{cond_conjecture on Omega}. Indeed, we have the following simple result.

\begin{lemma}\label{lem_sufficient condition for the existence of the open set}
Let $v:\mathbb{R}^n\to [0,\infty)$ be a Lebesgue measurable function such that $\mathcal{L}^n(\{ v>0 \})<\infty$ and satisfying both \eqref{cond_ F[v] is indecomposable}, and \eqref{cond_la chiusa di v = 0 meno v = 0 minuscola}. Then the set $\Omega:=\mathbb{R}^n\setminus \overline{\{ v\Low =0 \}}$ satisfies \eqref{cond_conjecture on Omega}.
\end{lemma}
\noindent
All in all, despite Theorem \ref{thm_rigidity by perugini} provides a characterization of  \emph{rigidity} under quite general conditions, Proposition \ref{prop_no open set equivalency} tells us that there's plenty of situations where still \emph{rigidity} holds true, but Theorem \ref{thm_rigidity by perugini} cannot be applied. Thus the characterization of \emph{rigidity} in its full generality still remains an open problem.

\subsection{Structure of the paper}\label{subsection_intro:structure of the paper}
In Section \ref{section_fundamentals} we recall some important definitions from geometric measure theory and the theory of functions of bounded variation. In Section \ref{section_restriction BV over curves} inspired by the work of Vol'pert originally done for characteristic functions, we construct a family of polygonal curves in $\Omega$ with the property that the restriction of $u$ to those curves is a function of bounded variation of one variable, and we prove some technical results that we will need in the subsequent sections. In Section \ref{section_SVD} we give a precise definition of the notion of \emph{singular vertical distance}, and we prove some important intermediate results that we will need for the characterization result of \emph{minimally singular} functions. Finally, we conclude the section with a couple of open problems (see Open Problem \ref{open 1} and \ref{open 2}). In Section \ref{section_minimally singular characterized via SVD} we prove Theorem \ref{thm_characterization minimally singular via SVD}. In the first part of Section \ref{section_application to Steiner} we apply the results obtained in the previous sections to prove Theorem \ref{thm_rigidity by perugini}. Lastly, in the remaining part of Section \ref{section_application to Steiner} we prove Lemma \ref{lem_open sets exists for n=1}, Proposition \ref{prop_no open set equivalency}, and Lemma \ref{lem_sufficient condition for the existence of the open set}.

\subsection*{Acknowledgements}
The author would like to thank Filippo Cagnetti for his valuable comments on the first draft of this work.

\section{Fundamentals of geometric measure theory}\label{section_fundamentals}
\noindent
In this section we recall some tools from Geometric Measure Theory and functions of bounded variation. For more details we refer to \cite{AFP,GMSbook1,maggiBOOK}. We divide this section into few subsections.

\subsection{Density points}
Let $n\in\mathbb{N}$. Let $x\in\mathbb{R}^n$, and $\rho>0$. We denote with  $B_\rho(x)$ the open ball in $\mathbb{R}^n$ of radius $\rho$ and center $x$, while $\mathbb{S}^{n-1}:=\{x\in\mathbb{R}^n:\, |x|=1  \}$ where $|\cdot|$ stands for the Euclidean norm in $\mathbb{R}^n$. If $x=0$ we will simply write $B_\rho$ instead of $B_\rho(x)$. Let $B \subset \R^n$ be a Lebesgue measurable set and let $x\in\R^n$. 
The upper and lower $n$-dimensional densities of $B$ at $x$ are defined as
\begin{eqnarray*}
  \theta^*(B,x) :=\limsup_{\rho\to 0^+}\frac{\mathcal{H}^n(B\cap B_\rho(x))}{\omega_n\,\rho^n}\,,
  \qquad
  \theta_*(B,x) :=\liminf_{\rho\to 0^+}\frac{\mathcal{H}^n(B\cap B_\rho(x))}{\omega_n\,\rho^n}\,,
\end{eqnarray*}
respectively, where $\omega_n\rho^n=\mathcal{L}^n(B_\rho(x))$.
It turns out that $x \mapsto  \theta^*(B,x)$ and $x \mapsto  \theta_*(B,x)$
are Borel functions that agree $\mathcal{L}^n$-a.e. on $\mathbb{R}^n$. 
Therefore, the $n$-dimensional density of $B$ at $x$
\[
\theta(B,x) := \lim_{\rho\to 0^+}\frac{\mathcal{H}^n(B\cap B_\rho(x))}{\omega_n\,\rho^n}\,,
\]
is defined for $\mathcal{L}^n$-a.e. $x\in\mathbb{R}^n$, and $x \mapsto  \theta (B,x)$
is a Borel function on $\mathbb{R}^n$. Given $t \in [0,1]$, we set
$$
B^{(t)} :=\{x\in\R^n:\theta(B,x)=t\}.
$$
By the Lebesgue differentiation theorem, the pair $\{B^{(0)},B^{(1)}\}$ 
is a partition of $\mathbb{R}^n$, up to a $\mathcal{L}^n$-negligible set. 
The set $\partial^{\mathrm{e}} B :=\mathbb{R}^n\setminus(B^{(0)}\cup B^{(1)})$ is called the \textit{essential boundary} of $B$.

\subsection{Functions of bounded variation and sets of finite perimeter} 
Let $f:\mathbb{R}^n  \to\R$ be a Lebesgue measurable function, 
and let $\Omega\subset \mathbb{R}^n$ be an open set, such that $f\in L^1(\Omega)$. Then we say that $f$ is of bounded variation in $\Omega$, and we write $f\in BV(\Omega)$ if and only if 
\begin{align}\label{def: total variation of Df, con f in BV}
\sup\Big\{\int_\Omega\,f(x)\,\diver\,T(x)\,dx:\,T\in C^1_c(\Omega;\R^n)\,,|T|\le 1\Big\}<\infty,
\end{align}
where $C^1_c(\Omega;\R^{n})$ is the set of $C^1$ functions from $\Omega$ to $\R^n$ with compact support. More in general, we say that $f\in BV_{\textnormal{loc}}(\Omega)$ if $f\in BV(\Omega')$ for every open set $\Omega'$ compactly contained in $\Omega$. If $f\in BV_{\textnormal{loc}}(\Omega)$ then there exists a $\mathbb{R}^n$-valued Radon measure on $\Omega$, which we call the distributional derivative of $f$ and we denote it with $Df$, with the following property. For every $T\in C^1_c(\Omega;\mathbb{R}^n)$ we have
\begin{align*}
\int_{\Omega}f(x)\textnormal{div}\,T(x)\,dx = -\int_{\Omega}T(x)\cdot dDf(x).
\end{align*} 
Moreover, if $u\in BV(\Omega)$ the total variation $|Df|$ of $Df$ in $\Omega$ is finite and its value $|Df|(\Omega)$ coincides with \eqref{def: total variation of Df, con f in BV}. One can write the  Radon--Nykodim decomposition of $Df$ with respect to $\mathcal{L}^n$
as $Df=D^af+D^sf$, where $D^sf$ and $\mathcal{L}^{n}$ are mutually singular, and $D^af\ll\mathcal{L}^{n}$. We denote the density of $D^af$ with respect to $\mathcal{L}^n$ by $\nabla f$, so that $\nabla\,f\in L^1(\Omega;\R^n)$ with $D^af=\nabla f\,\mathcal{L}^{n}$. Moreover, for $\mathcal{L}^{n}$-a.e. $x\in \Omega$, $\nabla f(x)$ is the approximate differential of $f$ at $x$. 

\noindent
Let us now briefly recall some definitions about the theory of functions of bounded variation of one variable (for more details we refer to \cite[Chapter 3.2]{AFP}). Let $a,b \in \mathbb{R}\cup \{ \pm \infty \}$ with $a<b$, and let $f: (a,b)\to \mathbb{R}$. Then we define the \emph{pointwise variation} $\textnormal{pV}(f,(a,b))$ of $f$ in $(a,b)$ as 
\begin{align*}
\textnormal{pV}(f,(a,b)):=\sup\left\{\sum_{i=1}^{m-1} |f(t_{i+1})-f(t_i)|:\, m\geq 2,\, a<t_1<\dots<t_m<b  \right\}.
\end{align*} 
Note that if $\textnormal{pV}(f,(a,b))$ is finite then $f$ is bounded in $(a,b)$. If $I^\circ\subset \mathbb{R}$ is a generic open set then $\textnormal{pV}(f,I^\circ)$ coincides with the sum of the pointwise variation of $f$ on each connected components of $I^\circ$. We also define the \emph{essential variation} $\textnormal{eV}(f,I^\circ)$ of $f$ in $I^\circ$ as
\begin{align*}
\textnormal{eV}(f,I^\circ):=\inf \left\{\textnormal{pV}(g,I^\circ):\, g(t)=f(t)\textnormal{ for }\mathcal{L}^1\textnormal{-a.e. }t\in I^\circ \right\}.
\end{align*}
It can be proved that if $f\in L^1 (I^\circ)$ and $\textnormal{eV}(f,I^\circ)<\infty$ then $f\in BV(I^\circ)$ with $|Df|(I^\circ)=\textnormal{eV}(f,I^\circ)$.

\noindent
Let us now introduce some basic definitions about the theory of sets of finite perimeter. Let $E\subset\R^n$ be a Lebesgue measurable set, and let  $\Omega\subset \mathbb{R}^n$ be an open set. We say that  $E\subset \mathbb{R}^n$ is a set of finite perimeter in $\Omega$ if and only if 
\begin{align}\label{def: locally finite perimeter of a set in R^k}
\sup\left\{\int_{\mathbb{R}^n}\chi_E(x)\textnormal{div}\,T(x)\,dx:\, T\in C^1_c(\Omega;\mathbb{R}^n)  \right\}<\infty,
\end{align}
where $\chi_E$ is the characteristic function of $E$. If $E\subset \mathbb{R}^n$ is a set of finite perimeter in $\Omega$, we denote with $P(E;\Omega)$ its relative perimeter in $\Omega$, where $P(E;\Omega)$ coincides with the quantity in \eqref{def: locally finite perimeter of a set in R^k}. If $P(E):=P(E;\mathbb{R}^n)<\infty$ we say that $E$ is a set of finite perimeter, while more generally if $P(E;\Omega')<\infty$ for every $\Omega'\subset\subset \Omega$, we say that $E$ is a set of locally finite perimeter in $\Omega$. If $E\subset \mathbb{R}^n$ is a set of finite perimeter and finite volume in $\Omega$, then we have that $\chi_E\in BV(\Omega)$, while in general if $E\subset \mathbb{R}^n$ is a set of finite perimeter in $\Omega$ then $\chi_E\in BV_{\textnormal{loc}}(\Omega)$. Moreover, if $E\subset \mathbb{R}^n$ is a set of locally finite perimeter in $\Omega$ we define the \emph{reduced boundary} $\partial^*E\subset\Omega\subset \mathbb{R}^n$ of $E$ as the set of those points of $\Omega$ such that
\begin{align*}
\nu^E(x):=\lim_{\rho\to 0^+}\frac{D\chi_E(B_\rho(x))}{|D\chi_E|(B_\rho(x))},
\end{align*}
exists and belongs to $\mathbb{S}^{n-1}$. The Borel function $\nu^E:\partial^*E\to \mathbb{S}^{n-1}$ 
is called the {\it (measure-theoretic) inner unit normal} to $E$. Given $E\subset \mathbb{R}^n$ set of locally finite perimeter in $\Omega$, we have that $D\chi_E=\nu^E\mathcal{H}^{n-1}\mres \partial^*E$ and,
\begin{align*}
\int_{\mathbb{R}^n}\chi_E(x)\textnormal{div}\,T(x)\,dx=-\int_{\partial^*E\cap \Omega}T(x)\cdot \nu^E(x)\,d\mathcal{H}^{n-1}(x),\quad \forall\,T\in C^1_c(\Omega;\mathbb{R}^n).
\end{align*}
The relative perimeter of $E$ in $B\subset \Omega$ is then defined by
$$
P(E;B):=|D\chi_E|(B)=\mathcal{H}^{n-1}(\partial^*E\cap B)
$$
for every Borel set $B \subset \Omega$. If $E$ is a set of locally finite perimeter in $\Omega$, it turns out that 
\begin{equation*}
  \label{inclusioni frontiere}
  \partial^*E  \subset (E^{(1/2)}\cap \Omega) \subset (\partial^{\mathrm{e}} E\cap \Omega).
\end{equation*}
Moreover, {\it Federer's theorem} holds true (see \cite[Theorem 3.61]{AFP} and \cite[Theorem 16.2]{maggiBOOK}), namely
\[
\mathcal{H}^{n-1}((\partial^{\mathrm{e}} E\cap \Omega)\setminus \partial^*E)=0.
\]
\subsection{General facts about measurable functions}
Let $f:\mathbb{R}^n\to \mathbb{R}$ be a Lebesgue measurable function. We define the {\it approximate upper limit} $f\Upp(x)$ and the {\it approximate lower limit} 
$f\Low(x)$ of $f$ at $x\in\R^n$ as 
\begin{eqnarray}
  \label{def fvee}
  f\Upp(x)=\inf\Big\{t\in\mathbb{R}: x\in \{f>t\}^{(0)}\Big\}\,,
  \\
  \label{def fwedge}
  f\Low (x)=\sup\Big\{t\in\mathbb{R}: x\in \{f>t\}^{(1)}\Big\}\,.
\end{eqnarray}
We observe that $f\Upp$ and $f\Low$ are Borel functions that are defined at 
{\it every} point of $\R^n$, with values in $\R\cup\{\pm\infty\}$.
Moreover, if $f_1: \R^n \to \R$ and $f_2: \R^n \to \R$
are measurable functions satisfying $f_1=f_2$ $\mathcal{L}^n$-a.e. on $\R^n$, 
then $f_1\Upp=f_2\Upp$ and $f_1\Low=f_2\Low$ {\it everywhere} on $\R^n$. 
We define the {\it approximate discontinuity} set $S_f$ of $f$ as $S_f: =\{f\Low<f\Upp\}$.
\noindent
Note that, by the above considerations, it follows that $\mathcal{H}^n(S_f)=0$. 
Although $f^\wedge$ and $f^\vee$ may take infinite values on $S_f$, 
the difference $f^\vee(x)-f^\wedge(x)$ is well defined in $\R\cup\{\pm\infty\}$ for every $x\in S_f$.
Then, we can define the {\it approximate jump} $[f]$ of $f$ as the Borel function $[f]:\R^n\to[0,\infty]$ given by
  \begin{eqnarray*}
    [f](x):=\left\{\begin{array}{l l}
      f\Upp(x)-f\Low(x)\,&\mbox{if $x\in S_f$}\,,
      \vspace{.2cm} \\
      0\,&\mbox{if $x\in \R^n\setminus S_f$}\,.
    \end{array}
    \right .
  \end{eqnarray*}

\noindent
For every Lebesgue measurable function $f:\R^n \rightarrow \R$ and for every $t\in \R$  we can prove the following relations
\begin{align}
\{|f|\Upp < t  \} &= \{-t<f\Low  \} \cap \{f\Upp < t   \},\label{eq: 2.7 Filippo page 15}\\
\{ f\Upp < t \}&\subset \{f < t  \}^{(1)} \subset \{ f\Upp \leq t  \},\label{eq: 2.8 Filippo page 15}\\
\{f\Low > t  \} &\subset \{ f> t  \}^{(1)} \subset \{f \Low \geq t  \}.\label{eq: 2.9 Filippo page 15}
\end{align}
Furthermore, if $f,g:\R^n \rightarrow \R$ are Lebesgue measurable functions and $f=g$ $\mathcal{L}^n$-a.e. on a Borel set $B$, then
\begin{align}\label{eq: 2.10 Filippo pag 16}
f\Upp(x)= g\Upp(x), \quad f\Low(x)=g\Low(x),\quad [f](x)=[g](x), \quad \forall x\in B^{(1)}.
\end{align} 
Let $A\subset\R^n$ be a Lebesgue measurable set. 
We say that $t\in\R\cup\{\pm\infty\}$ is the weak approximate limit of $f$ at $x$ with respect to $A$, 
and write $t=\aplim (f,A,x)$, if
\begin{eqnarray}
  &&\theta\Big(\{|f-t|>\varepsilon\}\cap A;x\Big)=0\,,\qquad\forall\varepsilon>0\,,\hspace{0.3cm}\qquad (t\in\R)\,,
  \\
  &&\theta\Big(\{f<M\}\cap A;x\Big)=0\,,\qquad\hspace{0.6cm}\forall M>0\,,\qquad (t=+\infty)\,,
  \\
  &&\theta\Big(\{f>-M\}\cap A;x\Big)=0\,,\qquad\hspace{0.3cm}\forall M>0\,,\qquad (t=-\infty)\,.
\end{eqnarray}
For $x\in\R^n$ and $\nu\in \mathbb{S}^{n-1}$, we denote by $H_{x,\nu}\subset \mathbb{R}^n$ the hyperplane passing through $x$ and orthogonal to $\nu$, namely
\begin{align*}
H_{x,\nu}:=\left\{z\in\mathbb{R}^n:\, (z-x)\cdot \nu=0   \right\}.
\end{align*}
If $x=0$ we simply write $H_\nu$ instead of $H_{0,\nu}$.
Moreover, we denote with $H_{x,\nu}^+$ and $H_{x,\nu}^-$ the closed half-spaces respectively whose boundaries are orthogonal to $\nu$, namely
\begin{eqnarray}\label{Hxnu+}
  H_{x,\nu}^+:=\Big\{z\in\R^n:(z-x)\cdot\nu\ge 0\Big\},\quad  H_{x,\nu}^-:=\Big\{z\in\R^n:(z-x)\cdot\nu\le 0\Big\}\,.
\end{eqnarray}
We say that $x\in S_f$ is a \textit{jump point} of $f$ if there exists $\nu\in \mathbb{S}^{n-1}$ such that
\[
f^\vee(x)=\aplim(f,H_{x,\nu}^+,x)> f^\wedge(x)=\aplim(f,H_{x,\nu}^-,x).
\]
If this is the case, we say that $\nu_f(x):= \nu$ is the approximate jump direction of $f$ at $x$.
We denote by $J_f$ the set of approximate jump points of $f$, moreover we have that $J_f\subset S_f$ 
and $\nu_f:J_f\to \mathbb{S}^{n-1}$ is a Borel function. Given $f:\R^n\rightarrow \R$ Lebesgue measurable function, we say that $f$ is \emph{approximately differentiable} at $x\in S_f^c$ provided $f\Low(x)=f\Upp(x)\in \R$ and there exists $\xi \in \R^n$ such that
\begin{align*}
\aplim(g,\R^n,x)=0,
\end{align*} 
where $g(z)= (f(z)-\tilde{f}(x)-\xi\cdot(z-x))/|z-x|$ for $z \in \mathbb{R}^n\setminus\{x \}$.
If this is the case, then $\xi$ is uniquely determined, we set $\xi=\nabla f(x)$, and call $\nabla f(x)$ the \emph{approximate differential} of $f$ at $x$. The localization property (\ref{eq: 2.10 Filippo pag 16}) holds true also for the approximate differentials, namely if $g,f:\R^n\rightarrow \R$ are Lebesgue measurable functions, $f=g$ $\mathcal{L}^n$-a.e. on a Borel set $B$, and $f$ is approximately differentiable $\mathcal{L}^n$-a.e. on $B$, then so it is $g$ $\mathcal{L}^n$-a.e. on $B$ with
\begin{align}\label{eq: 2.12 FIlippo pag 16}
\nabla f(x)=\nabla g(x) \quad \textit{\emph{for $\mathcal{L}^n$-a.e. $x\in B$}}.
\end{align}
Let us recall some useful properties we will need on the next sections (see \cite[Lemma 2.2, Lemma 2.3]{CagnettiColomboDePhilippisMaggiSteiner} for further details).
\begin{lemma}\label{lem: cantor total variation of v on v=0 is null}
If $f\in BV(\R^n)$, then $|D^cf|(\{f\Low=0  \})=0$. In particular, if $f=g$ $\mathcal{L}^n$-a.e. on a Borel set $B\subset \mathbb{R}^n$, then $D^cf\mres B^{(1)}= D^cg\mres B^{(1)}$.
\end{lemma}
\begin{lemma}\label{lem: 2.3 FIlippo page 18}
If $f,g \in BV(\R^n)$, $E$ is a set of finite perimeter in $\mathbb{R}^n$ and $f=\chi_E\,g  $, then
\begin{align}
\nabla f&=\chi_E \nabla g,\quad \quad \quad \mathcal{L}^n\textnormal{-a.e. on }\mathbb{R}^n,\\
D^cf&= D^cg\mres E^{(1)},\\
S_f\cap E^{(1)} &= S_g \cap E^{(1)}. 
\end{align}
\end{lemma}
\noindent
A Lebesgue measurable function $f:\mathbb{R}^n\rightarrow \mathbb{R}$, it's called of \emph{generalized bounded variation} on $\mathbb{R}^n$, shortly $f\in GBV(\mathbb{R}^n)$ if and only if $f^M\in BV_{loc}(\mathbb{R}^{n})$ for every $M>0$ where $f^M(x):=\max \{-M, \min\{ M,f(x) \} \}$. Let us observe that the structure theorem of BV-functions holds true for GBV-functions too. Indeed, given $f\in GBV(\mathbb{R}^n)$, then, (see \cite[Theorem 4.34]{AFP}) $\{ f>t \}$ is a set of finite perimeter for $\mathcal{L}^1$-a.e. $t\in \mathbb{R}$, $f$ is approximately differentiable $\mathcal{L}^n$-a.e. on $\mathbb{R}^n$, and the usual coarea formula takes the form 
\begin{align*}
\int_{\R}P(\{ f> t \};B)\,dt = \int_{B}|\nabla f|d\mathcal{H}^n + \int_{B\cap S_f}[f]d\mathcal{H}^{n-1} + |D^cf|(B),
\end{align*}
where for every Borel set $B\subset \R^n$, $|D^c f|$ denotes the Borel measure on $\mathbb{R}^n$ defined as
\begin{align}\label{def: Cantor part GBV}
|D^c f|(B)=\lim_{M\rightarrow +\infty}|D^c(f^M)|(B)=\sup_{M>0}|D^c (f^M)|(B).
\end{align}

\subsection{k-dimensional restrictions of BV-functions}\label{subsec_k-dimensional restrictions}
For more details about the following topics we refer to \cite[Chapter 3.11]{AFP}. Let $\Omega\subset \mathbb{R}^n$ be an open set. 
Let $k,n\in \mathbb{N}$ with $n\geq 2$, and $1\leq k \leq n-1$, and let $(\nu_i)_i\subset \mathbb{S}^{n-1}$ with $i=1,\dots,k$ be linear independent vectors which up to a single orthogonal transformation $R:\mathbb{R}^n\to\mathbb{R}^n$ they coincide with the first $k$ vectors of the canonical basis of $\mathbb{R}^n$, namely $R(\nu_i)=e_i$ for $i=1,\dots, k$. We denote with $\Omega_{\nu_1,\dots,\nu_k}\subset \mathbb{R}^n$ the set defined as
\begin{align*}
\Omega_{\nu_1,\dots,\nu_k}:=\left\{x\in \mathbb{R}^n:\, x\cdot\nu_i=0\, \forall\,i=1,\dots,k,\, \exists (t_1,\dots,t_k)\in\mathbb{R}^k \textnormal{ s.t. }x+t_1\nu_1+\dots +t_k\nu_k \in \Omega   \right\}.
\end{align*} 
For every $x\in \Omega$ we set with $x_{\nu_1,\dots, \nu_k}\in \mathbb{R}^n$ the unique point  in $\Omega_{\nu_1,\dots, \nu_k}$ with the property that there exists $(t_1,\dots,t_k)\in \mathbb{R}^k$ such that 
$$ x= x_{\nu_1,\dots, \nu_k} + t_1\nu_1+ \dots + t_k\nu_k. $$
For every $y\in \Omega_{\nu_1,\dots,\nu_k}$ we define the set $\Omega_y^{\nu_1,\dots,\nu_k}\subset\mathbb{R}^k$ as
\begin{align}\label{def_k dimensional slice}
\Omega_y^{\nu_1,\dots,\nu_k}:=\left\{(t_1,\dots,t_k)\in \mathbb{R}^k:\, y+t_1\nu_1+\dots +t_k\nu_k \in \Omega    \right\}.
\end{align}
In the following we will often refer to $\Omega_y^{\nu_1,\dots,\nu_k}$ as the \emph{$k$-dimensional slice} of $\Omega$ (at $y$). Let us stress that up to this point there was no need of working with open sets, indeed the notation introduced so far could have been defined also for generic Borel sets. If $u:\Omega\to \mathbb{R}$ is a Lebesgue measurable function we set with $u^{\nu_1,\dots,\nu_k}_y$ the restriction of $u$ to $\Omega^{\nu_1,\dots,\nu_k}_y$ namely
\begin{align*}
u^{\nu_1,\dots,\nu_k}_y(t_1,\dots,t_k):= u(y+ t_1\nu_1+ \dots +t_k\nu_k)\quad \forall\, (t_1,\dots, t_k)\in \Omega^{\nu_1,\dots, \nu_k}_y.
\end{align*}
As explained at the end of \cite[Section 3.11]{AFP} the theory of 1-dimensional restriction of BV function can be extended to higher dimensional sets without significance change in the proofs. In particular, as a consequence of a higher dimensional version of \cite[Theorem 3.103]{AFP} we have the following result.


\begin{proposition}\label{prop_k-dimensional BV slicing}
Let $u\in BV(\Omega)$, and let $G_{\nu_1,\dots,\nu_k}\subset \Omega_{\nu_1,\dots,\nu_k}$ be defined as
\begin{align*}
G_{\nu_1,\dots,\nu_k}:=\left\{y\in \Omega_{\nu_1,\dots,\nu_k}:\, u^{\nu_1,\dots,\nu_k}_{y}\in BV(\Omega^{\nu_1,\dots,\nu_k}_{y})   \right\}.
\end{align*}
Then $\mathcal{L}^{n-k}(\Omega_{\nu_1,\dots,\nu_k}\setminus G_{\nu_1,\dots,\nu_k})=0$.
\end{proposition}

\noindent
For $k=1$, we can state the above result in a more refined way (see \cite[Theorem 3.107, Theorem 3.108, and Theorem 4.35]{AFP}). In the following given $u\in BV(\Omega)$, for any $\nu\in\mathbb{S}^{n-1}$ we call \emph{directional distributional derivative} of $u$ along $\nu$, and we denote it with $D_\nu u$ the finite Radon measure on $\mathbb{R}^n$ defined as
$$
D_\nu u(B):=\int_{B} \nu\cdot dDu(x)\quad \forall\,B\subset \mathbb{R}^n\textnormal{ Borel}.
$$
\begin{proposition}\label{prop_GBV slicing}
Let $u\in L^1(\Omega)$. Then $u\in BV(\Omega)$ if and only if for every $\nu\in\mathbb{S}^{n-1}$ and for every $y\in G_\nu$ the $1$-dimensional section $u^\nu_y$ belongs to $BV(\Omega^\nu_y)$ and $|D_\nu u|(\Omega)=\int_{\Omega_\nu}|Du^\nu_y|(\Omega^\nu_y)\,dy<\infty$; moreover for every $y\in G_\nu$ the following equalities hold:
\begin{itemize}
\item[i)]$\nabla u(y+t \nu)\cdot \nu = \nabla u^\nu_y (t)$\quad for $\mathcal{L}^1$-a.e. $t\in\Omega^\nu_y$;\vspace*{0.2cm}
\item[ii)]$S_{u^\nu_y}=\{t\in \Omega^\nu_y:\, y+t \nu \in S_u  \}$;\vspace*{0.2cm}
\item[iii)]$(u^\nu_y)\Low(t)=u\Low(y+ t \nu)$\quad and \quad $(u^\nu_y)\Upp(t)=u\Upp(y+ t \nu)$\quad for every $t\in\Omega^\nu_y$.
\end{itemize}
\end{proposition}

\section{Restriction of BV functions to one-dimensional curves}\label{section_restriction BV over curves}
\noindent
Let $\Omega\subset \mathbb{R}^n$ be as in \eqref{assumptions on Omega}. Let $u\in BV(\Omega)$. The goal of this section is twofold. First, we introduce some notation about a particular family of curves associated to $u$ and $\Omega$. Second, we study the  properties of the restriction of $u$ over curves in that family. We divide this section in few subsections. 
\subsection{Some notation about curves}\label{subsec_A class of one-dimensional curves}
In the following, we consider curves of the type
\begin{align}\label{piecewise affine curves}
\gamma:I_\gamma\to \Omega\textnormal{ continuous, piecewise affine, and injective},
\end{align}
where $I_\gamma\subset\mathbb{R}$ is a closed bounded interval, and by \emph{continuous piecewise affine} here we mean that $\gamma(I_\gamma)\subset \Omega$ is a simple polygonal chain in $\Omega$, i.e. it is a curve of finite length which is the union of a finite number of segments inside $\Omega$ that intersect each other if and only if they are consecutive, and only at their extreme points. More precisely throughout this paper for every $\gamma:I_\gamma\to \Omega$ as in \eqref{piecewise affine curves} there exist $m=m(\gamma)\in\mathbb{N}$ and points $a_i\in [0,\infty)$ for $i=0,\dots,m$ with the property that $a_0 < a_1 < \dots < a_m$,  $I_\gamma=[a_0,a_m]$, and setting with $\eta_i:= (\gamma(a_i)-\gamma(a_{i-1}))/|\gamma(a_i)-\gamma(a_{i-1})|$ for $i=1,\dots, m$ we have that
\begin{align*}
\gamma'(t)=\eta_i \in \mathbb{S}^{n-1}\quad \forall\, t \in (a_{i-1},a_i), \textnormal{ for } i=1,\dots,m.
\end{align*}
Thus as a consequence, we have that
\begin{align}\label{gamma di t preliminary}
\gamma(t)= \gamma(a_{i-1}) + (t-a_{i-1})\eta_i \quad \forall\, t\in [a_{i-1},a_i], \textnormal{ for } i=1,\dots,m.
\end{align} 
\noindent
According with the notation introduced in Section \ref{subsec_k-dimensional restrictions}, we set $y_i=(\gamma(a_{i-1}))_{\eta_i}\in \Omega_{\eta_i}$ i.e. the orthogonal projection of $\gamma(a_{i-1})\in \Omega$  along the direction $\eta_i$ on the hyperplane $H_{\eta_i}$, and we call with $t_i\in \Omega^{\eta_i}_{y_i}$ the real number such that $\gamma(a_{i-1})= y_i +t_i\,\eta_i$ (see Figure \ref{fig_poligonal chains}). Thus, $[t_i,t_i+ (a_i-a_{i-1})]\subset \Omega^{\eta_i}_{y_i}$ and 
\begin{align*}
\gamma([a_{i-1},a_i])=\left\{y_i + t\, \eta_i:\, t\in [t_i,t_i+ (a_i-a_{i-1})] \right\},\quad \textnormal{ for } i=1,\dots,m.
\end{align*}


\noindent
In particular, we can rewrite \eqref{gamma di t preliminary} in the following way
\begin{align}\label{gamma di t intermediate}
\gamma(t)= y_i + (t+ (t_i -a_{i-1}))\,\eta_i \quad \forall\,t\in [a_{i-1},a_i],\textnormal{ for }i=1,\dots,m.
\end{align}
\noindent
We conclude this section with a definition.
\begin{definition}\label{def_gamma connects two points}
Given any two points $x_1, x_2 \in \Omega$, we say that $\gamma$ as in \eqref{piecewise affine curves} \emph{connects} $x_1, x_2$, if and only if $\gamma(a_0)=x_1$, and $\gamma(a_m)=x_2$ where $a_0,a_m$ are associated to $\gamma$ as described above.
\end{definition}

\begin{figure}[!htb]
\centering
\def\svgwidth{10cm}
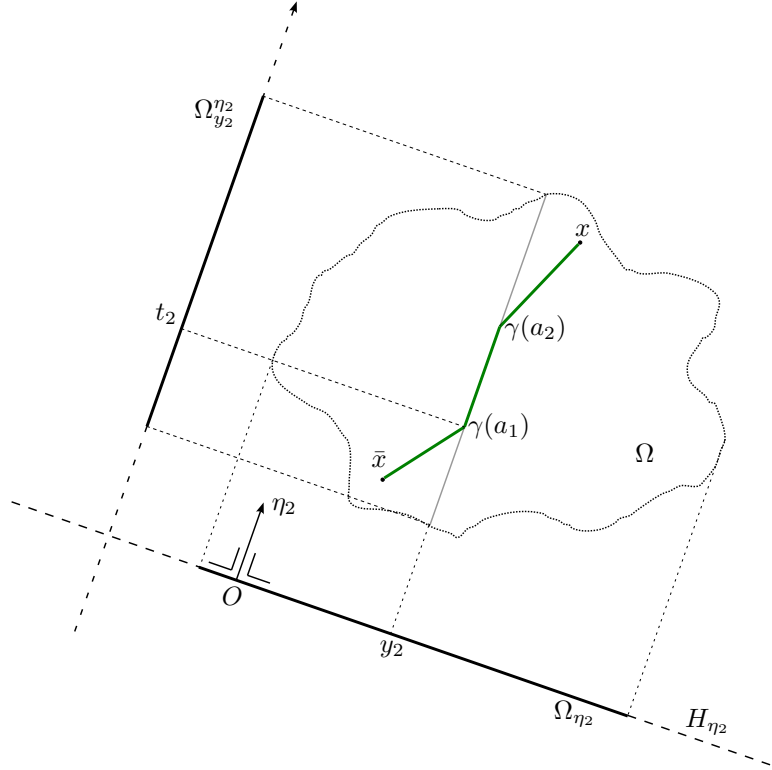
\caption{A pictorial representation in $\mathbb{R}^2$ of part of the notation introduced in Section \ref{subsec_k-dimensional restrictions} and in Section \ref{subsec_A class of one-dimensional curves}. In particular, $O$ stands for the origin in $\mathbb{R}^2$, and for $\gamma$ as in the above picture we have that $m=3$ with $\gamma(a_0)=\bar{x}$ and $\gamma(a_3)=x$.}
\label{fig_poligonal chains}
\end{figure}

\subsection{A family of one-dimensional curves}
Given $u\in BV(\Omega)$, we introduce the following family of curves in $\Omega$
\begin{align}\label{def_Gamma_Omega(u)}
\Gamma_{\Omega}(u):=\left\{ \gamma:I_\gamma \to \Omega:\, \gamma \textnormal{ is as in Section } \ref{subsec_A class of one-dimensional curves},\textnormal{ and } y_i \in G_{\eta_i}\,\forall\,i=1,\dots,m \right\},
\end{align}
where $(\eta_i)_i$ and $(y_i)_i$ are uniquely associated to each $\gamma$ as described in Section \ref{subsec_A class of one-dimensional curves}, while the set $G_{\eta_i}$ was introduced in Proposition \ref{prop_k-dimensional BV slicing}. Let us stress that, by definition of $\Gamma_\Omega(u)$ we have that $\gamma\in \Gamma_\Omega(u)$ if and only if $u^{\eta_i}_{y_i}\in BV(\Omega^{\eta_i}_{y_i})$ for every $i=1,\dots,m$. In the first part of this section we will show that it is possible to connect ``almost'' every couple of points in $\Omega$ with curves in $\Gamma_\Omega(u)$ (see Proposition \ref{prop_si puo connettere via Gamma}). In the remaining part of the section we will show that for every $\gamma\in \Gamma_\Omega(u)$, the function $u_\gamma:I_\gamma \to \mathbb{R}$ defined as 
\begin{align}\label{def_u_gamma}
u_\gamma(t):=u\Low(\gamma(t)),\quad \forall\,t\in I_{\gamma}
\end{align}
is a function of bounded variation in $I_\gamma^\circ$ and we will prove a result that is a sort of ``polygonal chain counterpart'' of Proposition \ref{prop_GBV slicing} (see Proposition \ref{prop_fine properties of ugamma}). Let us now begin with a definition.

\begin{definition}\label{def_suitable for connection}
A point $x\in \Omega$ is said to be \emph{suitable for connections (with respect to the function $u$}) if there exist $(n-1)$ linear independent vectors $(\nu_i)_i\subset \mathbb{S}^{n-1}$ which up to a single orthogonal transformation $R:\mathbb{R}^n\to\mathbb{R}^n$ they coincide with the first $(n-1)$ vectors of the canonical basis of $\mathbb{R}^n$, namely $R(\nu_i)=e_i$ for $i=1,\dots, n-1$, such that 
$$
x_{\nu_1,\dots,\nu_k} \in G_{\nu_1,\dots, \nu_k}\quad \forall\, k=1,\dots,n-1,
$$
where the point $x_{\nu_1,\dots,\nu_k}$ and the set $G_{\nu_1,\dots,\nu_k}$ were defined in  Section \ref{subsec_k-dimensional restrictions}, and in Proposition \ref{prop_k-dimensional BV slicing} respectively.
\end{definition}

\noindent
Roughly speaking, a point in $\Omega$ is \emph{suitable for connections} if and only if for every $k=1,\dots,n-1$, we can find a $k$-dimensional slice such that $x$ can be written as
$$
x=x_{\nu_1,\dots,\nu_k} + t_1\nu +\dots + t_k\nu_k\quad \textnormal{for some } (t_1,\dots,t_k)\in \Omega^{\nu_1,\dots,\nu_k}_{x_{\nu_1,\dots,\nu_k}},\, \forall\,k=1,\dots,n-1,
$$ 
and the function $u$ restricted to each of those $k$-dimensional slices is of bounded variation in the sense of Proposition \ref{prop_k-dimensional BV slicing}, namely $u^{\nu_1,\dots,\nu_k}_{x_{\nu_1,\dots,\nu_k}}\in BV\left(\Omega^{\nu_1,\dots,\nu_k}_{x_{\nu_1,\dots,\nu_k}}\right)$.

\begin{lemma}\label{lem_admissable for connections}
Let $u\in BV(\Omega)$, then $\mathcal{L}^n$-a.e. $x\in \Omega$ is suitable for connections.
\end{lemma}
\begin{proof}
Let $(\nu_i)_i\subset \mathbb{S}^{n-1}$ with $i=1,\dots,n-1$ be a sequence of $n-1$ linear independent vectors which up to a single orthogonal transformation $R:\mathbb{R}^n\to\mathbb{R}^n$ they coincide with the first $n-1$ vectors of the canonical basis of $\mathbb{R}^n$, namely $R(\nu_i)=e_i$ for $i=1,\dots, n-1$. Let $k\in \mathbb{N}$ with $1\leq k \leq n-1$, and let us call with $\Omega_k\subset \Omega$ the set so defined
$$
\Omega_k:=\left\{x\in \Omega:\, x_{\nu_1,\dots,\nu_k} \in G_{\nu_1,\dots,\nu_k}  \right\}.
$$
Let us show that $\mathcal{L}^n(\Omega\setminus \Omega_k)=0$ for every $k=1,\dots, n-1$. Indeed, by Fubini Theorem we get
\begin{align*}
\int_{\Omega}1\,dx &= \int_{\Omega_{\nu_1,\dots,\nu_k}}d\mathcal{H}^{n-k}(y)\int_{\Omega^{\nu_1,\dots,\nu_k}_y}1\,dt_1\dots dt_k =  \int_{G_{\nu_1,\dots,\nu_k}}d\mathcal{H}^{n-k}(y)\int_{\Omega^{\nu_1,\dots,\nu_k}_y}1\,dt_1\dots dt_k\\
&+ \int_{\Omega_{\nu_1,\dots,\nu_k}\setminus G_{\nu_1,\dots,\nu_k}}d\mathcal{H}^{n-k}(y)\int_{\Omega^{\nu_1,\dots,\nu_k}_y}1\,dt_1\dots dt_k\\
&= \int_{G_{\nu_1,\dots,\nu_k}}d\mathcal{H}^{n-k}(y)\int_{\Omega^{\nu_1,\dots,\nu_k}_y}1\,dt_1\dots dt_k=\int_{\Omega_k}1\,dx
\end{align*}
where for the second last equality sign we used that $\mathcal{H}^{n-k}(\Omega_{\nu_1,\dots,\nu_k}\setminus G_{\nu_1,\dots,\nu_k})=0$ (see  Proposition \ref{prop_k-dimensional BV slicing}). Thus $\mathcal{L}^n(\Omega\setminus \Omega_k)=0$ for every $k=1,\dots,n-1$, and so every $x\in \cap_{k=1}^{n-1}\Omega_k$ is admissible for connections. Since $\mathcal{L}^n(\Omega\setminus \cap_{k=1}^{n-1}\Omega_k)=0$ we conclude.
\end{proof}

\noindent
Next result provides a positive answer to the following question: given any two points $x_1, x_2\in \Omega$ that are both suitable for connections, is there always a curve $\gamma\in \Gamma_\Omega(u)$ that connects them? 

\begin{proposition}\label{prop_si puo connettere via Gamma}
Let $x,\bar{x}\in\Omega$ be suitable for connections with $x\neq \bar{x}$. Then there exists $\gamma\in \Gamma_\Omega(u)$ that connects them in the sense of Definition \ref{def_gamma connects two points}.
\end{proposition}
\noindent
The proof of this result is quite technical and it shows in a detailed way how to construct a polygonal chain as required by the statement. 
\begin{proof}[Proof of Proposition \ref{prop_si puo connettere via Gamma}]
We divide the proof in several steps.\\
\noindent
\textbf{Step 1.} Let $x\in\Omega$ be a point suitable for connections in the sense of Definition \ref{def_suitable for connection}. Up to apply the orthogonal transformation $R$ and up to translate everything to the origin we can consider $x=0$, and the vectors $(v_i)_i$ to be the ones of the canonical basis. Thus,
$$
0 \in G_{e_1,\dots, e_k}\quad \forall\, k=1,\dots,n-1,
$$
where we used the fact that $0_{e_1,\dots, e_k}=0$ for every $k=1,\dots, n-1$. In this first step we will prove that for every $\epsilon>0$ such that $B_\epsilon\subset \Omega$ (where we recall $B_\epsilon$ stands for the open ball in $\mathbb{R}^n$ centred at the origin and having radius equal to $\epsilon$) we can connect (in the sense of Definition \ref{def_gamma connects two points}) the origin with $\mathcal{L}^n$-a.e. $x\in B_\epsilon$. In order to simplify a bit the exposition, we  set $B=B_\epsilon$, and we consider the function $u$ as a function in $BV(B)$. In the following we denote with $B^k$ the $k$-dimensional ball in $\mathbb{R}^k$ of radius $\epsilon$, namely 
\begin{align*}
B^k&=\left\{ (x_1,\dots, x_k)\in \mathbb{R}^k:\, \left( x_1^2+\dots +x_k^2  \right)^{1/2}  <\epsilon \right\}\quad \textnormal{for }k=1,\dots,n.
\end{align*}
Let us notice that by definition of the $k$-dimensional slice (see \eqref{def_k dimensional slice}) we have that 
$$  
B^k=B^{e_1,\dots,e_k}_0\subset \Omega^{e_1,\dots,e_k}_0 \quad \forall\, k=1,\dots,n-1,\\
$$
thus we also set
\begin{align*}
u^k &= u^{e_1,\dots, e_k}_0 \quad \forall\, k=1,\dots,n-1.
\end{align*}
Moreover, with a little abuse of notation we  call with $(e_i)_i\subset \mathbb{S}^{k-1}$ for $i=1,\dots,k$,  where $k\in \mathbb{N}$ and $1\leq k\leq n$ the vectors of the canonical basis of $\mathbb{R}^k$. In this setting, let us note that 
\begin{align}\label{eq_B^k+1_k+1=B^k}
(B^{k+1})_{e_{k+1}}=\left\{(x_1,\dots,x_k,0)\in \mathbb{R}^{k+1}:\, (x_1,\dots,x_k)\in B^k   \right\}\quad \textnormal{for }k=1,\dots,n-1,
\end{align} 
where we recall that the definition of $(B^{k+1})_{e_{k+1}}$ was given in Section \ref{subsec_k-dimensional restrictions} for open sets in $\mathbb{R}^n$. We divide this step in sub-steps. See Figure \ref{fig_admissable} for a pictorial representation of part of the notation introduced in this step and in step 1.1. \\
\textbf{Step 1.1.} In this step we define in a recursive way a sequence of sets $V^k\subset \mathbb{R}^k$ for $k=1,\dots, n-1$, each one somehow related to the set $G_{e_1,\dots,e_k}\subset \Omega_{e_1,\dots,e_k}\subset \mathbb{R}^n$, and we state some preliminary technical  observations. More precisely, let $k\in\mathbb{N}$ with $1 \leq k\leq  n-1$, then by assumptions we know that $0\in G_{e_1,\dots,e_k}$ which by Proposition \ref{prop_k-dimensional BV slicing} implies that $u^k\in BV(\Omega_0^{e_1,\dots,e_k})$, and so keeping in mind the notation introduced in step 1 we have that
\begin{align*}
u^k\in BV(B^k).
\end{align*}
Let us now apply Proposition \ref{prop_k-dimensional BV slicing} in $\mathbb{R}^k$ for $u^k \in BV(B^k)$ for $k=2,\dots, n-1$. Let us slice the function $u^k$ in the direction $e^k$, namely by Proposition \ref{prop_k-dimensional BV slicing} there exists a set $G^k_{e_k}\subset (B^k)_{e_k}\subset \mathbb{R}^k$ such that
\begin{align}\label{eq_Volper set copre la palla}
\mathcal{L}^{k-1}((B^k)_{e_k}\setminus G^k_{e_k})=0\quad \textnormal{for }k=2,\dots, n,
\end{align}
and for every $y\in G^k_{e_k}$ we have that 
\begin{align*}
\left(u^k\right)^{e_k}_y \in BV((B^k)^{e_k}_y).
\end{align*}
For every $2\leq k\leq n$ we define the orthogonal projection $\pi_k:\mathbb{R}^k\to \mathbb{R}^{k-1}$  as
\begin{align*}
\pi_k(x)=(x_1,\dots,x_{k-1})\quad \forall\, x\in \mathbb{R}^k.
\end{align*}
Note that $\pi_{k+1}\left( (B^{k+1})_{e_{k+1}} \right)= B^k$ for every $k=1,\dots,n-1$. We set
\begin{align}\label{eq_V^k TILDE}
\tilde{V}^{k}&:=\pi_{k+1}(G^{k+1}_{e_{k+1}})=\left\{x\in B^{k}:\,(x_1,\dots,x_{k},0)\in G^{k+1}_{e_{k+1}}    \right\}\quad \textnormal{for }k=1,\dots ,n-1,
\end{align}
where in case $k=n-1$ we use the abuse of notation $G^n_{e_n}= G_{e_n}$ where the set $G_{e_n}\subset (B)_{e_n}$ is the set given by Proposition \ref{prop_k-dimensional BV slicing} when slicing the function $u\in BV(B)$ with respect to the direction $e_n$.
By construction of the sets $\tilde{V}^k$, we have that
\begin{align}\label{eq_pi(B^k)-pi(tilde V^k) =0}
\pi_k(\tilde{V}^{k})\subset \pi_k(B^{k})\quad \textnormal{for }k=2,\dots, n,
\end{align}
while thanks to both \eqref{eq_B^k+1_k+1=B^k}, and  \eqref{eq_Volper set copre la palla} we get that
\begin{align}\label{eq_the real volpert copre la palla AUXILIARY}
\mathcal{L}^k(B^k\setminus \tilde{V}^k)=0\quad \textnormal{for }k=1,\dots, n-1.
\end{align}

\begin{figure}[!htb]
\centering
\def\svgwidth{14cm}
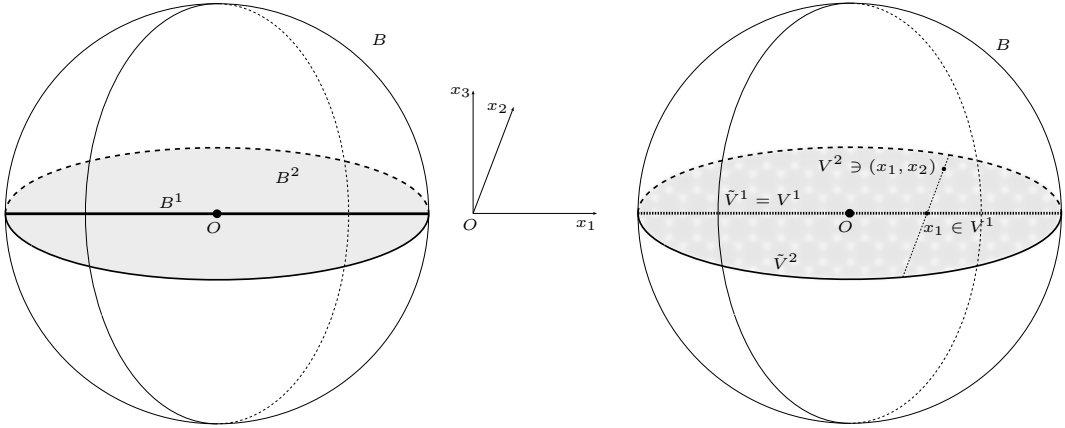
\caption{A pictorial representation in $\mathbb{R}^3$ of parts of the notation introduced at the beginning of step 1 and in step 1.1 of the proof of Proposition \ref{prop_si puo connettere via Gamma}. On the left we can see $B=B_\epsilon(0)\subset \mathbb{R}^3$ together with $B^2$ (in grey) and $B^1$. On the right picture we can see a dotted line replacing $B^1$ that represents the set $\tilde{V}^1=V^1$ (which thanks to \eqref{eq_the real volpert copre la palla AUXILIARY} has the same $\mathcal{L}^1$-measure of $B^1$), while in blurry grey we have the set $\tilde{V}^2$ (which again thanks to \eqref{eq_the real volpert copre la palla AUXILIARY} has the same $\mathcal{L}^2$-measure of $B^2$). Observe that a point $(x_1,x_2)\in \tilde{V}^2$  belongs to $V^2$ if and only if $x_1 \in V^1$.}
\label{fig_admissable}
\end{figure}

Starting from the sets $\tilde{V}^k$ (see for instance Figure \ref{fig_admissable}) we define
\begin{align}
V^1&:=\tilde{V}^1, \label{eq_V^1}\\
V^{k}&:=\left\{x\in \tilde{V}^{k}:\, \pi_k(x)\in V^{k-1}\right\}\quad \textnormal{for }k=2,\dots,n-1. \label{eq_V^k}
\end{align}
\textbf{Step 1.2.} 
Let us show that 
\begin{align}\label{eq_the real volpert copre la palla}
\mathcal{L}^k(B^k\setminus V^k)=0\quad \textnormal{for }k=1,\dots, n-1.
\end{align}
Thanks to \eqref{eq_the real volpert copre la palla AUXILIARY} it will be sufficient to show that 
\begin{align}\label{eq_real Volpert copre fake Volpert}
\mathcal{L}^k(\tilde{V}^k\setminus V^k)=0\quad \textnormal{for }k=1,\dots, n-1.
\end{align}
For $k=1$ thanks to \eqref{eq_V^1} there's nothing to prove. Let us now continue by induction on the dimension, namely let us assume that \eqref{eq_real Volpert copre fake Volpert} holds true for $k=m-1\leq n-2$, then by Fubini we get
\begin{align*}
\mathcal{L}^m(\tilde{V}^m) &= \int_{\pi_m(\tilde{V}^m)\cap V^{m-1}} \mathcal{H}^1\left( (\tilde{V}^m)^{e_m}_{(x_1,\dots,x_{m-1},0)} \right) dx_1\dots dx_{m-1} \\&+ \int_{\pi_m(\tilde{V}^m) \setminus V^{m-1}} \mathcal{H}^1\left( (\tilde{V}^m)^{e_m}_{(x_1,\dots,x_{m-1},0)} \right) dx_1\dots dx_{m-1}\\
&\stackrel{\eqref{eq_real Volpert copre fake Volpert} + \eqref{eq_the real volpert copre la palla AUXILIARY}}{=} \int_{\pi_m(\tilde{V}^m)\cap B^{m-1}} \mathcal{H}^1\left( (\tilde{V}^m)^{e_m}_{(x_1,\dots,x_{m-1},0)} \right) dx_1\dots dx_{m-1} \\&+ \int_{\pi_m(\tilde{V}^m) \setminus V^{m-1}} \mathcal{H}^1\left( (\tilde{V}^m)^{e_m}_{(x_1,\dots,x_{m-1},0)} \right) dx_1\dots dx_{m-1}\\
&\stackrel{\eqref{eq_V^k}}{=}\int_{\pi_m(\tilde{V}^m)\cap B^{m-1}} \mathcal{H}^1\left( (\tilde{V}^m)^{e_m}_{(x_1,\dots,x_{m-1},0)} \right) dx_1\dots dx_{m-1} + \mathcal{L}^m(\tilde{V}^m \setminus V^m)\\
& = \int_{\pi_m(\tilde{V}^m)\cap \pi_m(B^m)} \mathcal{H}^1\left( (\tilde{V}^m)^{e_m}_{(x_1,\dots,x_{m-1},0)} \right) dx_1\dots dx_{m-1} + \mathcal{L}^m(\tilde{V}^m \setminus V^m)\\
& \stackrel{\eqref{eq_pi(B^k)-pi(tilde V^k) =0}}{=}\mathcal{L}^m(\tilde{V}^m)  + \mathcal{L}^m(\tilde{V}^m \setminus V^m),
\end{align*} 
Thus, $\mathcal{L}^m(\tilde{V}^m \setminus V^m)= 0$ which proves \eqref{eq_real Volpert copre fake Volpert} for $k=m$. Thus, by induction on $m$, relation \eqref{eq_real Volpert copre fake Volpert} holds true, and this concludes step 1.2.\\
\noindent
\textbf{Step 1.3.} We conclude step 1, namely we show that we can connect the origin with $\mathcal{L}^n$-a.e. $x\in B$, where we recall $B=B_\epsilon$. Let us observe that thanks to \eqref{eq_the real volpert copre la palla} 
\begin{align}\label{eq_sta nelle coordinate di Volpert}
\textnormal{for }\mathcal{L}^n\textnormal{-a.e. }x\in B \textnormal{ we have that } (x_1,\dots,x_{n-1})\in V^{n-1}.
\end{align}
Let us show that for every $x\in B$ such that \eqref{eq_sta nelle coordinate di Volpert} holds true, there exists $\gamma\in \Gamma_\Omega(u)$ that connects (in the sense of Definition \ref{def_gamma connects two points}) the origin with the point $x$, concluding in this way Step 1. Let us fix any $k\in \mathbb{N}$ with $1\leq k \leq n$ then, in order to conclude it will be sufficient showing that there exists $\gamma\in \Gamma_\Omega(u)$ that connects $(x_1,\dots, x_{k-1},0,\dots,0)$ with $(x_1,\dots,x_k,0,0,\dots,0)$ (in the sense of Definition \ref{def_gamma connects two points}) whenever $|x_k|>0$, in particular we will show that such a $\gamma$ is indeed the segment connecting those two points. Indeed, we set with $\gamma_k:[0,|x_k|] \to \Omega$ the curve so defined
\begin{align*}
\gamma_k(t)=(x_1,\dots,x_{k-1},0,\dots,0) + t\, \textnormal{sign}(x_k)e_k,\quad \forall\,t\in [0,|x_k|].
\end{align*}
We want to show that $\gamma_k\in \Gamma_\Omega(u)$, which by \eqref{def_Gamma_Omega(u)} it means that $(\gamma_k(0))_{\textnormal{sign}(x_k)e_k} \in G_{\textnormal{sign}(x_k)e_k}$, that is equivalent to ask that
\begin{align}\label{eq_il segmento connette i due punti}
(x_1,\dots,x_{k-1},0,\dots,0)\in G_{e_k}.
\end{align}
Indeed, thanks to \eqref{eq_sta nelle coordinate di Volpert}, and by relation \eqref{eq_V^k} we know that $(x_1,\dots,x_{k-1})\in V^{k-1}$ which by \eqref{eq_V^k TILDE} implies that 
$(x_1,\dots,x_{k-1},0)\in (G^k)_{e_k}$. Then, for every $t\in (B^k)^{e_k}_{(x_1,\dots,x_{k-1},0)}$ we have that
\begin{align*}
&(u^k)^{e_k}_{(x_1,\dots,x_{k-1},0)}(t)=u^k((x_1,\dots,x_{k-1},0)+ t e_k)= u^k((x_1,\dots,x_{k-1},t))\\
&= u^{(e_1,\dots,e_{k-1},e_k)}_0((x_1,\dots,x_{k-1},t))=u(0 + x_1 e_1 + \dots + t e_k)\\
&=u((x_1,\dots,x_{k-1},0,\dots,0)+t e_k)= u^{e_k}_{(x_1,\dots,x_{k-1},0,\dots,0)}(t).
\end{align*}
Since $(x_1,\dots,x_{k-1},0)\in (G^k)_{e_k}$ we know that $(u^k)^{e_k}_{(x_1,\dots,x_{k-1},0)} \in BV\left((B^k)^{e_k}_{(x_1,\dots,x_{k-1},0)}\right)$, and so thanks to the above calculations we have also that $u^{e_k}_{(x_1,\dots,x_{k-1},0,\dots,0)}\in BV\left(B^{e_k}_{(x_1,\dots,x_{k-1},0,\dots,0)}\right)$ which implies that $(x_1,\dots,x_{k-1},0,\dots,0)\in G_{e_k}$. This proves \eqref{eq_il segmento connette i due punti}, and by generality of $k$ we conclude the first step.\\
\noindent
\textbf{Step 2.} In this step we introduce some notation. We divide this step in few sub steps.\\
\noindent
\textbf{Step 2.1.} In this step we construct a polygonal chain that connects $x$ and $\bar{x}$. By connectedness of $\Omega$ we can find a continuous function $\gamma_0: [a_0,b_0] \to \Omega$ such that $\gamma_0(a_0)=x$ and $\gamma_0(b_0) = \bar{x}$. Since $\gamma_0(I_0)$ is a compact set in $\Omega$, there exists $\epsilon_0>0$ such that
$$
B_{\epsilon_0}(\gamma_0(t))\subset \Omega\quad \forall\, t\in I_0.
$$ 
Thus, there exits a polygonal chain $\gamma: I_\gamma \to \Omega$ as described as in Section \ref{subsec_A class of one-dimensional curves} that connects $x$ and $\bar{x}$, and such that 
$$
\gamma(I_\gamma)\subset \bigcup_{t\in I_0}B_{\epsilon_0}(\gamma_0(t)).
$$
It is not restrictive to choose $\gamma$ with the following property, namely that $R(\eta_1)= e_n$, and $\bar{R}(\eta_m)=e_n$ where $R:\mathbb{R}^n \to \mathbb{R}^n$, and $\bar{R}:\mathbb{R}^n \to \mathbb{R}^n$ are those rotations given by Definition \ref{def_suitable for connection} for $x$ and $\bar{x}$ respectively, while the vectors $(\eta_i)_i\subset \mathbb{S}^{n-1}$ are associated to $\gamma$ as described in Section \ref{subsec_A class of one-dimensional curves}. So, arguing as before, by compactness of $\gamma(I_\gamma)$ in $\Omega$, there exists $\epsilon>0$ such that 
$$
B_{\epsilon}(\gamma(t))\subset \Omega\quad \forall\, t\in I_\gamma.
$$ 
Let us choose $\epsilon>0$ such that $B_{\epsilon}(\gamma(a_i))\cap B_{\epsilon}(\gamma(a_j))=\emptyset$ for every $i\neq j$ with $i,j=0,\dots,m$, where the points $(a_i)_i\subset [0,\infty)$ are associated to $\gamma$ as described in Section \ref{subsec_A class of one-dimensional curves}. This concludes step 2.1.\\
\noindent
\textbf{Step 2.2.}
In this step we fix some notation. We set
\begin{align*}
B_i=B_\epsilon(\gamma(a_i))\quad \textnormal{for } i=0,\dots,m,
\end{align*}
We call with $D_0\subset B_0$, and $D_m\subset B_m$ the set of those points that thanks to Step 1, they can be connected to $x$ and $\bar{x}$ respectively. We also set 
$$
D_i:=\left\{x\in B_i:\, x_{\eta_{i+1}}\in (G_i)_{\eta_{i+1}}   \right\},\quad \textnormal{for }i=1,\dots,m-1
$$
where $(G_i)_{\eta_{i+1}}\subset (B_i)_{\eta_{i+1}}$ is the set given by Proposition \ref{prop_k-dimensional BV slicing} for the function $u$ restricted to the set $B_i$. Let us note that by definition we get
\begin{align}\label{eq_D_i occupa tutta B_i}
\mathcal{L}^{n}(B_i\setminus D_i)=0\quad \textnormal{for }i=1,\dots,m.
\end{align}
We call with $\psi_i:\mathbb{R}^n \to \mathbb{R}^n$ the diffeomorphism defined as
\begin{align*}
\psi_i(x)&:= x + (a_i-a_{i-1})\,\eta_i\quad \textnormal{for }i=1,\dots,m,\\
\psi^{-1}_i(x) &:= x- (a_i-a_{i-1})\,\eta_i\quad \textnormal{for }i=1,\dots,m.
\end{align*}
Note that 
$$
\psi_i(B_{i-1})=B_i,\qquad \psi^{-1}_i(B_i)=B_{i-1}\quad \textnormal{for }i=1,\dots,m.
$$
Lastly, we set 
\begin{align*}
\psi_{j,\dots,i}(x)&:=\psi_j\left(\dots \psi_i(x)   \dots \right) \quad \textnormal{for } 1\leq i < j\leq m,\\
\psi_{j,\dots,i}^{-1}(x)&:=\psi_i^{-1}\left(\dots \psi_j^{-1}(x)   \dots \right) \quad \textnormal{for }1\leq i < j\leq m.
\end{align*}
Note that 
\begin{align*}
\psi_{j,\dots,i}(B_i)=B_j,\qquad \psi_{j,\dots,i}^{-1}(B_j)=B_i\quad \textnormal{for }1\leq i < j\leq m.
\end{align*}
With this we conclude step 2.2.\\
\noindent
\textbf{Step 3.} We conclude the proof. We  divide this step in sub steps.\\
\noindent
\textbf{Step 3.1.} In this step we prove that if $x_0\in D_0$ then $\psi_1(x_0)$ can be connected with $x$, moreover we prove that if $x_m\in D_m$ then $\psi^{-1}_m(x_m)$ can be connected with $\bar{x}$. Indeed, up to apply the rotation $R$ associated to the point $x$ as described in Definition \ref{def_suitable for connection}, and up to translate $x$ to the origin, we obtain the same setting as in step 1, so that $B_0=B$. Moreover, if $x_0\in D_0\subset B_0$ then the point $\psi_1(x_0)\in B_1$, and by assumption on $\eta_1$ (see end of step 2.1) we have that the first $n-1$ components of $\psi_1(x_0)$ coincide with the first $n-1$ components of $x_0$, namely
$$
(\psi_1(x_0)_1,\dots,\psi_1(x_0)_{n-1})=((x_0)_1,\dots,(x_0)_{n-1}).
$$
Furthermore, since $x_0\in D_0$, then \eqref{eq_sta nelle coordinate di Volpert} is in force, and thus
\begin{align*}
(\psi_1(x_0)_1,\dots,\psi_1(x_0)_{n-1})\stackrel{\eqref{eq_sta nelle coordinate di Volpert}}{\in} V_{n-1}.
\end{align*}
Thus, since 
\begin{align}\label{eq_segment}
\left\{x_0+(t-a_0)e_n:\, \forall\, t\in [a_0,a_1]\right\}\subset \Omega,
\end{align}
arguing as in step 1.3 we conclude that also the point $\psi_1(x_0)$ can be connected with the origin. Let us point out that, thanks to the particular choice of $\epsilon>0$ (see end of step 2.1) the segment defined in \eqref{eq_segment} intersects the curve connecting $(\psi_1(x_0)_1,\dots,\psi_1(x_0)_{n-1},0)$ with the origin, only in $(\psi_1(x_0)_1,\dots,\psi_1(x_0)_{n-1},0)$. A similar argument can be used to prove that if $x_m\in D_{m}$, then also $\psi^{-1}_m(x_m)$ can be connected with $\bar{x}$. This concludes step 3.1.\\
\noindent
\textbf{Step 3.2.} We conclude in case $m=1,2$. Let's start with $m=1$. Referring to the notation introduced in step 2.2, we define
\begin{align*}
\bar{B}_1:=\psi_1(D_0)\cap D_1.
\end{align*}
By properties of the diffeomorphism $\psi_i$  we have that $\mathcal{L}^n(B_1\setminus \bar{B}_1)=0$. Let $x_1\in \bar{B}_1$, then  since $x_1\in \psi_1(D_0)$ and thanks to step 3.1 we can connect $x_1$ with $x$, while since $x_1\in D_1$ we can also connect $x_1$ with $\bar{x}$, so we conclude for $m=1$. In case $m=2$ we define
\begin{align*}
\bar{B}_1:=\psi_1(D_0)\cap D_1 \cap \psi^{-1}_2(D_2).
\end{align*}
By properties of the diffeomorphism $\psi_i$  we have that $\mathcal{L}^n(B_1\setminus \bar{B}_1)=0$. Similarly to what we have done for $m=1$, let $x_1\in \bar{B}_1$, then thanks to step 3.1, since $x_1\in \psi_1(D_0)$  we can connect $x_1$ with $x$, while since $x_1\in \psi^{-1}_2(D_2)$ we can also connect $x_1$ with $\bar{x}$, so we conclude for $m=2$.
\\
\noindent
\textbf{Step 3.3.} Let us conclude for $m\geq 3$. We define
\begin{align*}
B_{1,1}&:=\psi_1(D_0)\cap D_1,\\
B_{1,2}&:= B_{1,1}\cap \psi^{-1}_{2}(D_2),\\
B_{1,j}&:= B_{1,(j-1)} \cap \psi^{-1}_{j,\dots,2}(D_j) \quad \textnormal{for }j=3,\dots,m,
\end{align*}
Let us now consider the following sequence of sets $\bar{B}_i\subset B_i$ for $i=1,\dots,m-1$ defined as
\begin{align*}
\bar{B}_1&:= B_{1,m},\\
\bar{B}_i&:= \psi_i(\bar{B}_{i-1})\quad \textnormal{for }i=2,\dots,m-1.
\end{align*}
By properties of the diffeomorphism $\psi_i$  we have that $\mathcal{L}^n(B_i\setminus \bar{B}_i)=0$ for $i=1,\dots,m-1$. Let us now construct a sequence of $m-1$ points in $\Omega$ and a polygonal chain in $\Gamma_\Omega(u)$ that connects them. More precisely, let $x_1\in \bar{B}_1$ be any point, let us call with
\begin{align*}
x_2&:=\psi_2(x_1),
\end{align*}
while in case $m\geq 4$ we further set
\begin{align*}
x_i&:=\psi_{i,\dots,2}(x_{i-1})\quad \textnormal{for }i=3,\dots,m-1.
\end{align*}
Note that by construction, $x_i\in \bar{B}_i$ for $i=1,\dots,m-1$. Let us show that the curve $\gamma_i:[a_{i-1},a_i]\to \Omega$ defined as
\begin{align*}
\gamma_i(t)= x_{i-1} + (t-a_{i-1})\eta_i
\end{align*}
is in $\Gamma_\Omega(u)$ and connects $x_{i-1}$ with $x_i$ for $i=2,\dots,m-1$. Then, since by step 3.1 $x_1$ can be connected to $x$, and $x_{m-1}$ can be connected with $\bar{x}$, by connecting all the curves $\gamma_i$ we conclude the proof. Indeed, $x_{i-1}\in \bar{B}_{i-1}$ and so thanks to the way we constructed the sets $(\bar{B}_i)_i$ we deduce that $x_{i-1}\in D_{i-1}$, thus by definition of $D_{i-1}$ we have that $(x_{i-1})_{\eta_i}\in (G_{i-1})_{\eta_i}$ and this proves that $\gamma_i \in \Gamma_\Omega(u)$. By generality of $i=2,\dots,m-1$ we conclude.
\end{proof}
\noindent
Next result shows some useful properties of the function $u_\gamma$ that was defined in \eqref{def_u_gamma}. Roughly speaking, this result can be seen as a sort of ``polygonal chain counterpart'' of Proposition \ref{prop_GBV slicing}.

\begin{proposition}\label{prop_fine properties of ugamma}
Let $\gamma\in \Gamma_\Omega(u)$. Then $u_\gamma\in BV(I_\gamma^\circ)$, and the following three relations hold true:
\begin{itemize}
\item[i)] $\nabla u_\gamma(t)= \nabla u(y_i+(t+(t_i-a_{i-1}))\eta_i)\cdot\eta_i$ \quad \textnormal{for } $\mathcal{L}^1$-\textnormal{a.e. }$t\in (a_{i-1},a_i)$, $i=1,\dots,m$;\vspace*{0.2cm}
\item[ii)] \textnormal{If }$t\in I_{\gamma}^\circ$\textnormal{ is a jump point for }$u_\gamma$,\textnormal{ then }$\gamma(t)\in \Omega$\textnormal{ is a jump point for }$u$,\textnormal{ moreover}\vspace*{0.2cm}
$$[u_\gamma](t)=[u](\gamma(t))\quad \forall\,t \in J_{u_\gamma};$$\vspace*{0.01cm}
\item[iii)] $|D^c u_\gamma|(I^\circ_{\gamma})=\sum_{i=1}^m |D^c(u^{\eta_i}_{y_i})|((t_i,t_i+(a_{i}-a_{i-1})))$,\vspace*{0.01cm}
\end{itemize}
where we recall that $(y_i)_i$, $(a_i)_i$ and $(\eta_i)_i$ for $i=0,\dots,m$ are associated to $\gamma$ as described in Section \ref{subsec_A class of one-dimensional curves}.
\end{proposition}
\begin{proof}
We divide the proof in several steps.\\
\textbf{Step 1.} In this first part of the proof we will prove that $u_\gamma\in BV(I_{\gamma}^\circ)$. This will be done in two sub steps.\\
\textbf{Step 1.1.} In this step we will prove that $u_\gamma\in BV((a_{i-1},a_i))$, and $u_\gamma$ is a good representative in the equivalence class of $u_\gamma$ in $(a_{i-1},a_i)$ for every $i=1,\dots,m$. In particular, we will prove that setting $N(\gamma)=\{a_i:\, i=1,\dots,m-1  \}$, we have that 
\begin{align}\label{ugamma in BV preliminary}
u_\gamma\in BV(I_\gamma^\circ\setminus (N(\gamma))),
\end{align}
and so, since $\mathcal{L}^1(N(\gamma))=0$, $u\in L^1(I_\gamma^\circ)$. Let us fix any $i\in \{ 1,\dots,m \}$, then 
\begin{align*}
u_\gamma(t)\stackrel{\eqref{gamma di t intermediate}+\eqref{def_u_gamma}}{=} u\Low(y_i+(t+(t_i-a_{i-1}))\eta_i)\quad \forall\, t\in (a_{i-1},a_i).
\end{align*} 
Thus, applying the change of variable $s=t + (t_i-a_{i-1})$ we get that
\begin{align}\label{ugamma = u fettato in (ai-1 , ai)}
u_\gamma(t)=u_\gamma(s-(t_i-a_{i-1}))= u\Low(y_i+s\,\eta_i)=u^{\eta_i}_{y_i}(s)\quad \forall\, t\in (a_{i-1},a_i).
\end{align} 
Thus, by definition of $\Gamma_{\Omega}(u)$ we know that $u^{\eta_i}_{y_i}\in BV((t_i,t_i+(a_i-a_{i-1})))$ for all $i=1,\dots,m$, and thus $u_\gamma\in BV((a_{i-1},a_i))$, which implies \eqref{ugamma in BV preliminary}. As a consequence of \eqref{ugamma in BV preliminary}, \eqref{ugamma = u fettato in (ai-1 , ai)}, and of Proposition \ref{prop_GBV slicing} we also deduce the validity of relation i).\\
\textbf{Step 1.2.} In this step we will show that $u_\gamma\in BV(I_\gamma^\circ)$. For any $i\in\{1,\dots,m  \}$ there exists $\bar{\epsilon}>0$ such that for every $0<\epsilon<\bar{\epsilon}$
$$
[t_i,t_i +(a_i-a_{i-1})]\subset (t_i-\epsilon,t_i + (a_i-a_{i-1})+\epsilon)\subset \Omega^{\eta_i}_{y_i}.
$$
Thus, recalling the definitions of \emph{pointwise variation} and \emph{essential variation} (see Section \ref{section_fundamentals}) we get 
\begin{align*}
&\textnormal{pV}(u_\gamma,[a_{i-1},a_i])=\textnormal{pV}(u^{\eta_i}_{y_1},[t_i,t_i+(a_i-a_{i-1})])\leq \textnormal{pV}(u^{\eta_i}_{y_i}, (t_i-\epsilon,t_i+(a_i-a_{i-1})+\epsilon))\\
&= \textnormal{eV}(u^{\eta_i}_{y_i}, (t_i-\epsilon,t_i+(a_i-a_{i-1})+\epsilon)),
\end{align*}
where for the last equality sign, we used that $u^{\eta_i}_{y_i}\in BV((t_i-\epsilon,t_i+(a_i-a_{i-1})+\epsilon))$, together with the fact that $(u\Low)^{\eta_i}_{y_i}$ is a good representative for $u^{\eta_i}_{y_i}$ in $\Omega^{\eta_i}_{y_i}$ (see \cite[relation (3.24)]{AFP}). Passing to the $\inf$ among all $0<\epsilon<\bar{\epsilon}$, and by Borel regularity of the Radon measure $|Du^{\eta_i}_{y_i}|$ defined on $\mathbb{R}$ we conclude that
\begin{align}\label{pV < eV}
\textnormal{pV}(u_\gamma,[a_{i-1},a_i])\leq \textnormal{eV}(u^{\eta_i}_{y_i}, [t_i,t_i + (a_i-a_{i-1})])<\infty.
\end{align}
Thus,
\begin{align*}
\textnormal{pV}(u_\gamma,I_\gamma^\circ) \leq \sum_{i=1}^m\textnormal{pV}(u_\gamma,[a_{i-1},a_i])\stackrel{\eqref{pV < eV}}{\leq}
\sum_{i=1}^m\textnormal{eV}(u^{\eta_i}_{y_i},[t_i,t_i + (a_i-a_{i-1})])<\infty.
\end{align*}
Thus $u_\gamma\in BV(I_\gamma^\circ)$. This concludes Step 1. Moreover, thanks to \eqref{ugamma = u fettato in (ai-1 , ai)} and since $\mathcal{H}^0(N(\gamma))<\infty$ we also deduce the validity of relation iii).\\
\textbf{Step 2.} In this step we will show the validity of relation ii). We divide the proof in few sub steps. Let us first notice that thanks to step 1.1, relation \eqref{ugamma = u fettato in (ai-1 , ai)}, and Proposition \ref{prop_GBV slicing} we have that relation ii) holds true whenever $t\in J_{u_\gamma}\setminus N(\gamma)$. Thus we will focus our attention in the case of $t\in J_{u_\gamma}\cap N(\gamma)$.\\
\textbf{Step 2.1.} In this step we will show that, for any $i\in \{1,\dots, m-1  \}$, $t_i+(a_i-a_{i-1})\in \Omega^{\eta_i}_{y_i}$ is a jump point for $u^{\eta_i}_{y_i}$ if and only if $t_{i+1}\in \Omega^{\eta_{i+1}}_{y_{i+1}}$ is a jump point for $u^{\eta_{i+1}}_{y_{i+1}}$. We will prove only one implication since the opposite one follows by similar considerations. Let us assume that $t_i+(a_i-a_{i-1})\in \Omega^{\eta_i}_{y_i}$ is a jump point for $u^{\eta_i}_{y_i}$, namely
$$
t_i+(a_i-a_{i-1}) \in S_{u^{\eta_i}_{y_i}}.
$$
Thanks to point ii) of Proposition \ref{prop_GBV slicing} this implies that
$$
\gamma(a_i)=y_i+(t_i+(a_i-a_{i-1}))\eta_i \in S_u,
$$
where in the above relation we used \eqref{gamma di t intermediate} in the interval $[a_{i-1}, a_i]$ with $t=a_i$. In a similar way, using once again \eqref{gamma di t intermediate} in the interval $[a_i,a_{i+1}]$, and $t=a_i$ we get that
$$
\gamma(a_i)=y_{i+1}+t_{i+1}\eta_{i+1},
$$
which thanks to point ii) of Proposition \ref{prop_GBV slicing} implies that $t_{i+1}\in S_{u^{\eta_{i+1}}_{y_{i+1}}}$.
This concludes step 2.1.\\
\textbf{Step 2.2.} Let $m\geq 2$. In this step we will show that if $a_i\in N(\gamma)$ is a jump point for $u_\gamma$, for some $i\in\{1,\dots, m-1 \}$, then $t_i+(a_i-a_{i-1})\in \Omega^{\eta_i}_{y_i}$, and $t_{i+1}\in \Omega^{\eta_{i+1}}_{y_{i+1}}$ are jump points for $u^{\eta_i}_{y_i}$ and for $u^{\eta_{i+1}}_{y_{i+1}}$, respectively. We will prove this by contradiction. Let assume that $a_i$ is a jump point of $u_\gamma$, and suppose that $t_i+(a_i-a_{i-1})$ is not a jump point for $u^{\eta_i}_{y_i}$. Thus by Proposition \ref{prop_GBV slicing} we infer that the weak approximate limit of $u$ at $\gamma(a_i)$ exists and we call it $z\in \mathbb{R}$. Let us show that $z$ is the weak approximate limit of $u_\gamma$ at $a_i$. Indeed, for every $0<\rho<\min\{(a_{i+1}-a_i), (a_i-a_{i-1}) \}$, and for every $\epsilon>0$ we have that
\begin{align*}
&\frac{\mathcal{H}^1\left(\{t\in (a_i-\rho,a_i+\rho):\,|u_{\gamma}(t)-z|>\epsilon \} \right)}{\rho}= \frac{\mathcal{H}^1\left(\{t\in (a_i-\rho,a_i):\,|u_{\gamma}(t)-z|>\epsilon \} \right)}{\rho}\\
&+ \frac{\mathcal{H}^1\left(\{t\in  (a_i,a_i+\rho):\,|u_{\gamma}(t)-z|>\epsilon \} \right)}{\rho} 
\\&\stackrel{\eqref{ugamma = u fettato in (ai-1 , ai)}}{=}\frac{\mathcal{H}^1\left(\{s\in (t_i+ (a_i-a_{i-1})-\rho, t_i+(a_i-a_{i-1})):\,|u^{\eta_i}_{y_i}(s)-z|>\epsilon \} \right)}{\rho}
\\& + \frac{\mathcal{H}^1\left(\{s\in (t_{i+1},t_{i+1}+\rho):\,|u^{\eta_{i+1}}_{y_{i+1}}(s)-z|>\epsilon \} \right)}{\rho}
\\&\leq \frac{\mathcal{H}^1\left(\{s\in (t_i+ (a_i-a_{i-1})-\rho, t_i+ (a_i-a_{i-1})+\rho):\,|u^{\eta_i}_{y_i}(s)-z|>\epsilon \} \right)}{\rho}
\\& + \frac{\mathcal{H}^1\left(\{s\in (t_{i+1}-\rho, t_{i+1}+\rho):\,|u^{\eta_{i+1}}_{y_{i+1}}(s)-z|>\epsilon \} \right)}{\rho}.
\end{align*} 
Passing to the limit as $\rho\to 0^+$, since (thanks to step 2.1) $z$ is the weak approximate limit of both $u^{\eta_i}_{y_i}$, and of $u^{\eta_{i+1}}_{y_i+1}$ at $t_i+(a_i-a_{i-1})$ and at $t_{i+1}$, respectively we get that
$$
\lim_{\rho\to 0^+}\frac{\mathcal{H}^1\left(\{t\in (a_i-\rho,a_i+\rho):\,|u_{\gamma}(t)-z|>\epsilon \} \right)}{\rho}=0.
$$
Thus, thanks to \cite[Proposition 3.65]{AFP} we deduce that $z$ is the approximate limit of $u_\gamma$ at $a_i$, which is a contradiction. This concludes the proof of step 2.2.\\
\textbf{Step 2.3.} Let $a_i\in N(\gamma)$ be a jump point of $u_\gamma$. Thanks to step 2.2 we know that both $t_i+(a_i-a_{i-1})\in \Omega^{\eta_i}_{y_i}$, and $t_{i+1}\in \Omega^{\eta_{i+1}}_{y_{i+1}}$ are jump points for $u^{\eta_i}_{y_i}$ and for $u^{\eta_{i+1}}_{y_{i+1}}$, respectively.
Without loss of generality we can assume that 
\begin{align}\label{eq_eta+1 salta}
u\Upp(\gamma(a_i))=(u^{\eta_{i+1}}_{y_{i+1}})\Upp(t_{i+1})=\textnormal{aplim}(u^{\eta_{i+1}}_{y_{i+1}},[t_{i+1},\infty), t_{i+1}),
\end{align}
where for the first equality sign we used point iii) of Proposition \ref{prop_GBV slicing}. Since by step 2.2 also $t_i+(a_i-a_{i-1})$ is a jump point for $u^{\eta_i}_{y_i}$, we are left with only two possibilities, namely either 
\begin{align}\label{eq_eta salta ,ma ugamma no}
u\Upp(\gamma(a_i))=(u^{\eta_{i}}_{y_{i}})\Upp(t_{i}+(a_i-a_{i-1}))=\textnormal{aplim}(u^{\eta_{i}}_{y_{i}}, (-\infty, t_{i}+(a_i-a_{i-1})], t_{i}+(a_i-a_{i-1})),
\end{align}
or
\begin{align}\label{eq_eta salta , e ugamma pure}
u\Low(\gamma(a_i))=(u^{\eta_{i}}_{y_{i}})\Low(t_{i}+(a_i-a_{i-1}))=\textnormal{aplim}(u^{\eta_{i}}_{y_{i}}, (-\infty, t_{i}+(a_i-a_{i-1})], t_{i}+(a_i-a_{i-1})).
\end{align}
Goal of this step is to prove that if $a_i\in N(\gamma)$ is a jump point of $u_\gamma$ and \eqref{eq_eta+1 salta} is in force, then \eqref{eq_eta salta , e ugamma pure} holds true. Suppose by contradiction that \eqref{eq_eta+1 salta}, and \eqref{eq_eta salta ,ma ugamma no} hold true. Setting $z=u\Upp(\gamma(a_i))$ we have that
\begin{align*}
&\frac{\mathcal{H}^1\left(\{t\in (a_i-\rho,a_i+\rho):\,|u_{\gamma}(t)-z|>\epsilon \} \right)}{\rho}= \frac{\mathcal{H}^1\left(\{t\in (a_i-\rho,a_i):\,|u_{\gamma}(t)-z|>\epsilon \} \right)}{\rho}\\
&+ \frac{\mathcal{H}^1\left(\{t\in  (a_i,a_i+\rho):\,|u_{\gamma}(t)-z|>\epsilon \} \right)}{\rho} 
\\&\stackrel{\eqref{ugamma = u fettato in (ai-1 , ai)}}{=}\frac{\mathcal{H}^1\left(\{s\in (t_i+ (a_i-a_{i-1})-\rho, t_i+(a_i-a_{i-1})):\,|u^{\eta_i}_{y_i}(s)-z|>\epsilon \} \right)}{\rho}
\\& + \frac{\mathcal{H}^1\left(\{s\in (t_{i+1},t_{i+1}+\rho):\,|u^{\eta_{i+1}}_{y_{i+1}}(s)-z|>\epsilon \} \right)}{\rho}.
\end{align*} 
Thanks to  relations \eqref{eq_eta+1 salta}, and \eqref{eq_eta salta ,ma ugamma no}, we get that
$$
\lim_{\rho\to 0^+}\frac{\mathcal{H}^1\left(\{t\in (a_i-\rho,a_i+\rho):\,|u_{\gamma}(t)-z|>\epsilon \} \right)}{\rho}=0,
$$
thus $u\Upp(\gamma(a_i))$ is the approximate limit of $u_\gamma$ at $a_i$, and so $a_i\in N(\gamma)$ is not a jump point for $u_\gamma$. This contradicts our assumption, and concludes the step.\\
\textbf{Step 3.} We conclude. So let $a_i\in N(\gamma)$ be a jump point for $u_\gamma$, and assume that \eqref{eq_eta+1 salta} holds true, then by step 2.3 we know that \eqref{eq_eta salta , e ugamma pure} holds true. Moreover,
\begin{align*}
&(u_\gamma)\Upp (a_i)=\textnormal{aplim}(u_\gamma,(a_i,\infty),a_i)\stackrel{\eqref{ugamma = u fettato in (ai-1 , ai)}}{=}\textnormal{aplim}(u^{\eta_{i+1}}_{y_{i+1}},(t_{i+1},\infty), t_{i+1})\stackrel{\eqref{eq_eta+1 salta}}{=}(u^{\eta_{i+1}}_{y_{i+1}})\Upp(t_{i+1})=u\Upp(\gamma(a_i)),\vspace*{0.2cm}\\
&(u_\gamma)\Low(a_i)=\textnormal{aplim}(u_\gamma,(-\infty,a_i),a_i)\stackrel{\eqref{ugamma = u fettato in (ai-1 , ai)}}{=}\textnormal{aplim}(u^{\eta_{i}}_{y_{i}}, (-\infty, t_{i}+(a_i-a_{i-1})], t_{i}+(a_i-a_{i-1}))
\\ &\stackrel{\eqref{eq_eta salta , e ugamma pure}}{=}(u^{\eta_{i}}_{y_{i}})\Low(t_{i}+(a_i-a_{i-1}))=u\Low(\gamma(a_i)).
\end{align*}
This concludes the proof.
\end{proof} 

\noindent
Let us now focus our analysis to a $\mathcal{L}^n$-equivalent subset of $\Omega$, namely we introduce the set
\begin{align}\label{def_tilda Omega u}
\tilde{\Omega}_u:= \left\{x\in \Omega\setminus S_u:\,x \textnormal{ is suitable for connections in the sense of Definition \ref{def_suitable for connection}  } \right\}.
\end{align}

\begin{lemma}\label{lem_calculation |u_gamma (a0)-u_gamma(am)|}
Let $x_1,x_2\in \tilde{\Omega}_u$, and let $\gamma\in \Gamma_\Omega(u)$ be such that it connects $x_1,x_2$ in the sense of Definition \ref{def_gamma connects two points}. Then,
\begin{align}\label{eq_|u_gamma (a0)-u_gamma(am)| < eV}
|u_\gamma(a_m)-u_\gamma(a_0)|\leq|Du_\gamma|(I_\gamma^\circ),
\end{align}
where $a_0,a_m$ are associated to $\gamma$ as described in Section \ref{subsec_A class of one-dimensional curves}.
\end{lemma}
\begin{proof}
We divide the proof in few steps.\\
\textbf{Step 1.} In this step we will prove that
\begin{align}\label{se tacca bene}
u_\gamma(a_0)=\lim_{t\to a_0^+} (u_\gamma)^r(t),\quad u_\gamma(a_m)=\lim_{t\to a_m^-}(u_\gamma)^l(t),
\end{align}
where $(u_\gamma)^r$, and $(u_\gamma)^l$ are the right continuous, and left continuous representative of $u_\gamma$ in $I_\gamma^\circ=(a_0,a_m)$, respectively. Indeed,
\begin{align*}
\lim_{t\to a_0^+} (u_\gamma)^r(t)\stackrel{\eqref{ugamma = u fettato in (ai-1 , ai)}}{=}\lim_{s\to t_1^+}(u^{\eta_1}_{y_1})^r(s)=(u^{\eta_1}_{y_1})^r(t_1)=(u^{\eta_1}_{y_1})\Low(t_1)=u\Low(y_1+t_1\eta_1)\stackrel{\eqref{gamma di t intermediate}}{=}u_{\gamma}(a_0),
\end{align*}
where for the second, and third  equality signs we used the fact that $t_1\in  (\Omega^{\eta_1}_{y_1}\setminus S_{u^{\eta_1}_{y_1}})$, and thus $u^{\eta_1}_{y_1}$ is approximately continuous in $t_1$, while for the second last equality we used point iii) of Proposition \ref{prop_GBV slicing}. Following a similar argument to the one we have just used, we complete the proof of \eqref{se tacca bene}. This concludes the first step.\\
\textbf{Step 2.} Let $\epsilon>0$, we define the function $u_{\gamma,\epsilon}:(a_0-\epsilon,a_m+\epsilon)\to \mathbb{R}$ as
\begin{align}
u_{\gamma,\epsilon}(t)=
\begin{cases}
u_\gamma(a_0)\quad &\mbox{if }t\in (a_0-\epsilon,a_0],\vspace*{0.1cm}\\
u_{\gamma}(t)\quad &\mbox{if }t\in (a_0,a_m)\vspace*{0.1cm},\\
u_\gamma(a_m)\quad &\mbox{if }t \in [a_m,a_m+\epsilon).
\end{cases}
\end{align}
Let us show that $u_{\gamma,\epsilon}\in BV((a_0-\epsilon,a_m+\epsilon))$ and that any good representative of $u_{\gamma,\epsilon}$ in $(a_0-\epsilon,a_m+\epsilon)$ is continuous in both $a_0$ and $a_m$. Indeed, by construction it is not difficult to prove that $u_{\gamma,\epsilon}\in L^1((a_0-\epsilon,a_m+\epsilon))$, moreover
\begin{align*}
\textnormal{pV}(u_{\gamma,\epsilon}, (a_0-\epsilon,a_m+\epsilon))=\textnormal{pV}(u_\gamma,[a_0,a_m])\stackrel{\eqref{pV < eV}}{<}\infty,
\end{align*}
and so $u_{\gamma,\epsilon}\in BV((a_0-\epsilon,a_m+\epsilon))$. Let us now address the problem of the continuity at $a_0$ and $a_m$. Indeed,
\begin{align*}
(u_{\gamma,\epsilon})^r(a_0)= \lim_{t\to a_0^+}(u_{\gamma,\epsilon})^r(t)=\lim_{t\to a_0^+} (u_\gamma)^r(t)=u_\gamma(a_0),
\end{align*}
where for the last equality sign we used step 1. Finally,
\begin{align*}
\lim_{t\to a_0^-}(u_{\gamma,\epsilon})^l(t)=\lim_{t\to a_0^-}(u_{\gamma}(a_0))^l(t)=u_{\gamma}(a_0),
\end{align*}
which together with the relation above proves that $u_{\gamma,\epsilon}$ is approximately continuous in $a_0$. A similar argument to the one just used shows the approximate continuity of $u_{\gamma,\epsilon}$ in $a_m$. This concludes the second step.\\
\textbf{Step 3.} We conclude. Indeed,
\begin{align*}
|u_\gamma(a_m)-u_\gamma(a_0)|&=|(u_{\gamma,\epsilon})^l(a_m)-(u_{\gamma,\epsilon})^l(a_0)|=|Du_{\gamma,\epsilon}([a_0,a_m))|=|Du_{\gamma,\epsilon}(a_0) + Du_{\gamma,\epsilon}((a_0,a_m))|\\&=|Du_\gamma((a_0,a_m))|\leq |Du_\gamma|((a_0,a_m)),
\end{align*}
where for the first and fourth equality signs we used the fact that, thanks to step 2, $u_{\gamma,\epsilon}$ is approximate continuous in both $a_0$ and $a_m$ (so both $a_0$ and $a_m$ are not atoms of $Du_{\gamma,\epsilon}$), while for the second equality sign we used \cite[Theorem 3.28]{AFP}. This concludes the proof of the lemma.
\end{proof}

\section{The notion of singular vertical distance}\label{section_SVD}
\noindent
Let $\Omega\subset \mathbb{R}^n$ be as in \eqref{assumptions on Omega}. Given $u\in BV(\Omega)$ we introduce the notion of \emph{singular vertical distance} between two points of $\tilde{\Omega}_u$ (defined in \eqref{def_tilda Omega u}), and we prove some properties of such a notion of distance. These results will be then used to give a geometric characterization for \emph{minimally singular} functions (see Theorem \ref{thm_characterization minimally singular via SVD}).

\begin{definition}\label{def: SVD}
Let $u\in BV(\Omega)$. Given any $x_1, x_2 \in \tilde{\Omega}_u$, we call \emph{singular vertical distance} between $x_1$ and $x_2$ with respect to $u$ in $\Omega$, and we denote it with $\textnormal{SVD}_{u,\Omega}(x_1,x_2)\in [0,\infty)$, the quantity defined as
\begin{align}\label{eq: def SVD}
\textnormal{SVD}_{u,\Omega}(x_1,x_2):=
\begin{cases}
\inf \left\{|D^s u_\gamma|(I_\gamma^\circ):\, \gamma\in \Gamma_\Omega(u),\,\gamma\textnormal{ connects }x_1,x_2  \right\}\quad &\mbox{if } x_1\neq x_2;\vspace*{0.2cm}\\
0 \quad &\mbox{if }x_1=x_2;
\end{cases}
\end{align}
where the notion of $\gamma$ connecting $x_1$ and $x_2$ was given in Definition \ref{def_gamma connects two points}.
\end{definition}

\begin{remark}\label{rem_SVD for more points}
Let us observe that, more in general, we could define the singular vertical distance between $x_1$ and $x_2$ with respect to $u$ in $\Omega$ for every couple of points $x_1,x_2 \in \Omega$ such that they are both suitable for connections (see Definition \ref{def_suitable for connection}), thus including in this way also points that possibly are in $S_u$.
\end{remark}

\begin{remark}\label{rem: notion of connectedness}
Let us observe that the couple $(\tilde{\Omega}_u, \textnormal{SVD}_{u,\Omega})$ is a pseudometric space (see for instance \cite[Chapter 9]{Royden}). Indeed, let us check that $\textnormal{SVD}_{u,\Omega}:\tilde{\Omega}_u \times \tilde{\Omega}_u \to [0,\infty)$ is a pseudometric which by definition is equivalent to say that the following three properties hold true:
\begin{itemize}
\item[i)]$\textnormal{SVD}_{u,\Omega}(x,x)=0$ for every $x\in \tilde{\Omega}_u$;
\item[ii)]$\textnormal{SVD}_{u,\Omega}(x_1,x_2)=\textnormal{SVD}_{u,\Omega}(x_2,x_1)$ for every $x_1,x_2 \in \tilde{\Omega}_u$;
\item[iii)] $\textnormal{SVD}_{u,\Omega}(x_1,x_2)\leq \textnormal{SVD}_{u,\Omega}(x_1,x_3)+\textnormal{SVD}_{u,\Omega}(x_3,x_2)$ for every $x_1,x_2,x_3 \in \tilde{\Omega}_u$.
\end{itemize}
The first two relations follow by definition of $\textnormal{SVD}_{u,\Omega}$. Let us show that relation iii) holds true, i.e. let us show the validity of the triangular inequality for the notion of $\textnormal{SVD}_{u,\Omega}$. Indeed,
\begin{align*}
&\textnormal{SVD}_{u,\Omega}(x_1,x_3) + \textnormal{SVD}_{u,\Omega}(x_3,x_2)\\
&=\inf\left\{|D^s u_\gamma|(I_\gamma^\circ):\, \gamma\in \Gamma_\Omega(u),\, \gamma \textnormal{ connects }x_1, x_2,\textnormal{ and }x_3=\gamma(t) \textnormal{ for some }t\in I_\gamma^\circ   \right\}\\
&\geq \inf\left\{|D^s u_\gamma|(I_\gamma^\circ):\, \gamma\in \Gamma_\Omega(u),\, \gamma \textnormal{ connects }x_1, x_2 \right\}\\
&=\textnormal{SVD}_{u,\Omega}(x_1,x_2).
\end{align*}
Motivated by the above three properties, we can use the \emph{singular vertical distance} to induce an equivalence relation between points of $\tilde{\Omega}_u$: more precisely we say that $x_1$ is equivalent to $x_2$ (in $\tilde{\Omega}_u$), if and only if $\textnormal{SVD}_{u,\Omega}(x_1,x_2)=0$. Thus, for every $\bar{x}\in \tilde{\Omega}_u$ we denote with $\{ \bar{x}\}_{u,\Omega}$ the equivalence class of $\tilde{\Omega}_u$ having $\bar{x}$ as representative, namely
\begin{align}\label{def: (x) equivalence class}
\{ \bar{x} \}_{u,\Omega}=\left\{x\in\tilde{\Omega}_u:\, \textnormal{SVD}_{u,\Omega} (\bar{x},x)=0 \right\}.
\end{align}  
\end{remark}

\begin{remark}
Thanks to Remark \ref{rem: notion of connectedness} we have the following equivalent statement of Theorem \ref{thm_characterization minimally singular via SVD}: $u$ is \emph{minimally singular} if and only if there exists $\bar{x}\in \Omega$ such that $\{\bar{x} \}_{u,\Omega}=\Omega$ up to a set of $\mathcal{L}^n$-measure zero.
\end{remark}

\noindent
Next result represents the main ingredient in the proof of Theorem \ref{thm_characterization minimally singular via SVD}.
\begin{lemma}\label{lem_SVD is in GBV}
Let $u\in BV(\Omega)$, let $\bar{x}\in \tilde{\Omega}_u$, and let $b: \Omega \to [0,\infty)$ be defined as
\begin{align}\label{def: the smart baricenter}
b(x):=
\begin{cases}
\textnormal{SVD}_{u,\Omega}(\bar{x},x)\quad &\textnormal{for } x\in \tilde{\Omega}_u,\vspace*{0.2cm}\\
0\quad &\textnormal{otherwise}.
\end{cases}
\end{align}
Then $b\in GBV(\Omega)$, and it satisfies conditions i), ii), and iii) of Definition \ref{def_minimally singular functions} with respect to  the function $u$.
\end{lemma}
\begin{proof}
We divide the proof in few steps.\\
\noindent
\textbf{Step 1.} Goal of the present step is to prove that $b\in GBV(\Omega)$. Since $u\in BV(\Omega)$, thanks to Proposition \ref{prop_GBV slicing}, for every $\nu \in \mathbb{S}^{n-1}$ we have that $|D_\nu u|(\Omega)<\infty$. So let us fix any $\nu \in \mathbb{S}^{n-1}$, and let $y\in G_\nu\subset  \Omega_\nu$ so that $|D (u^\nu_y)|(\Omega^\nu_y)<\infty$. Since $\Omega$ is open, so it is $\Omega^\nu_y$. Thus, let us consider any connected component $(a,b)$ of $\Omega^\nu_y$, and let $t_1,t_2\in \Omega^\nu_y$ be such that $a<t_1<t_2<b$. We set $x_1= y+t_1 \nu$, $x_2= y+t_2 \nu$, and $x_t=y+t\,\nu$ for a generic $t\in (t_1,t_2)$. Since $\mathcal{L}^n(\Omega\setminus\tilde{\Omega}_u)=0$ we can assume without loss of generality that $\mathcal{L}^1(\Omega^\nu_y\setminus (\tilde{\Omega}_u)^\nu_y)=0$. Recalling Remark \ref{rem_SVD for more points}, let us define $\bar{b}:(a,b)\to [0,\infty)$ as 

\begin{align}\label{eq_Barbie baricenter}
\bar{b}(t):=
\textnormal{SVD}_{u,\Omega}(\bar{x}, x_t)+\frac{1}{2}[u^\nu_y](t)\quad \forall\, t\in \Omega^\nu_y.
\end{align}
Let us observe that according with Definition \ref{def_suitable for connection}, since $\mathcal{L}^1(\Omega^\nu_y\setminus (\tilde{\Omega}_u)^\nu_y)=0$, every point in $\Omega^\nu_y$ is suitable for connections . Moreover, by definition of $\bar{b}$, and thanks to $\mathcal{L}^1(\Omega^\nu_y\setminus (\tilde{\Omega}_u)^\nu_y)=0$, we have that $\bar{b}(t)= b^\nu_y (t)$ for $\mathcal{L}^1$-a.e. $t\in \Omega^\nu_y$. Let us now prove that 
\begin{align}\label{eq: Step 2.1 1 necessity}
|\bar{b}(t_1)-\bar{b}(t_2)|\leq |D^s u^\nu_y|((t_1,t_2))+\frac{1}{2}[u^\nu_y](t_1)+\frac{1}{2}[u^\nu_y](t_2)\quad \forall\,t_1,t_2 \in (a,b).
\end{align}
In order to prove \eqref{eq: Step 2.1 1 necessity}, let us first prove that for every $x_1,x_2 \in \Omega$ that are suitable for connections we get
\begin{align}\label{SVD triangular inequality}
&\textnormal{SVD}_{u,\Omega}(x_1,x_2)+ \frac{1}{2}[u](x_1)+ \frac{1}{2}[u](x_2) \nonumber
\\& \geq \left|\textnormal{SVD}_{u,\Omega}(\bar{x},x_1) + \frac{1}{2}[u](x_1)-\textnormal{SVD}_{u,\Omega}(\bar{x},x_2)-\frac{1}{2}[u](x_2)\right|.
\end{align}
Indeed, by Definition \ref{def: SVD} we get
\begin{align*}
&\textnormal{SVD}_{u,\Omega}(\bar{x},x_1) + [u](x_1)+ \textnormal{SVD}_{u,\Omega}(x_1,x_2)\\ 
& = \inf\left\{|D^su_\gamma|(I_\gamma^\circ):\, \gamma \in \Gamma_\Omega(u),\, \gamma \textnormal{ connects }\bar{x},x_2,\textnormal{ and } x_1=\gamma(t) \textnormal{ for some  } t\in I_\gamma^\circ  \right\}\\
& \geq \inf\left\{|D^su_\gamma|(I_\gamma^\circ):\, \gamma \in \Gamma_\Omega(u),\, \gamma \textnormal{ connects }\bar{x},x_2 \right\}\\
&= \textnormal{SVD}_{u,\Omega}(\bar{x},x_2).
\end{align*}
which is equivalent to
\begin{align}\label{eq_SVD triangular part 1}
&\textnormal{SVD}_{u,\Omega}(x_1,x_2)+ \frac{1}{2}[u](x_1)+ \frac{1}{2}[u](x_2) \nonumber
\\& \geq \textnormal{SVD}_{u,\Omega}(\bar{x},x_2)+\frac{1}{2}[u](x_2) -\textnormal{SVD}_{u,\Omega}(\bar{x},x_1) - \frac{1}{2}[u](x_1).
\end{align}
We also have that
\begin{align*}
\textnormal{SVD}_{u,\Omega}(\bar{x},x_2) + [u](x_2)+ \textnormal{SVD}_{u,\Omega}(x_2,x_1) \geq \textnormal{SVD}_{u,\Omega}(\bar{x},x_1),
\end{align*}
which is equivalent to
\begin{align}\label{eq_SVD triangular part 2}
&\textnormal{SVD}_{u,\Omega}(x_2,x_1)+ \frac{1}{2}[u](x_1)+ \frac{1}{2}[u](x_2) \nonumber
\\& \geq \textnormal{SVD}_{u,\Omega}(\bar{x},x_1)+\frac{1}{2}[u](x_1) -\textnormal{SVD}_{u,\Omega}(\bar{x},x_2) - \frac{1}{2}[u](x_2).
\end{align}
By Definition \ref{def: SVD} we have that $\textnormal{SVD}_{u,\Omega}(x_2,x_1)=\textnormal{SVD}_{u,\Omega}(x_1,x_2)$, hence combining \eqref{eq_SVD triangular part 1} and \eqref{eq_SVD triangular part 2}, we get \eqref{SVD triangular inequality}. Thus,
\begin{align*}
&|\bar{b}(t_1)-\bar{b}(t_2)|\stackrel{\eqref{eq_Barbie baricenter}}{=}\left|\textnormal{SVD}_{u,\Omega}(\bar{x},x_1)+ \frac{1}{2}[u](x_1)-\textnormal{SVD}_{u,\Omega}(\bar{x},x_2)-\frac{1}{2}[u](x_2)\right|\\
&\stackrel{\eqref{SVD triangular inequality}}{\leq} \textnormal{SVD}_{u,\Omega}(x_1,x_2)+\frac{1}{2}[u](x_1)+\frac{1}{2}[u](x_2)\stackrel{\eqref{eq: def SVD}}{\leq}|D^s u^\nu_y|((t_1,t_2))+\frac{1}{2}[u](x_1)+\frac{1}{2}[u](x_2)
\\
&=|D^s u^\nu_y|((t_1,t_2))+\frac{1}{2}[u^\nu_y](t_1)+\frac{1}{2}[u^\nu_y](t_2),
\end{align*}
where for the last inequality sign we used the fact that the segment connecting the two points $x_1$, and $x_2$ is a curve in $\Gamma_\Omega(u)$, and this proves \eqref{eq: Step 2.1 1 necessity}. Thus, recalling the definition of \emph{pointwise}, and \emph{essential} variation (see for instance \cite[Chapter 3.2]{AFP}) which we denote with $\textnormal{pV}$ and \textnormal{eV} respectively, and recalling that $b^\nu_y(t)=\bar{b}(t)$ for $\mathcal{L}^1$-a.e. $t\in (a,b)$ we get

\begin{align*}
&\textnormal{eV}(b^\nu_y,(a,b))\leq \textnormal{pV}(\bar{b},(a,b))= \sup\left\{\sum_{i=1}^{m-1} |\bar{b}(t_{i+1})-\bar{b}(t_i)|:\, m\geq 2,\, a<t_1<\dots<t_m<b  \right\}\\
&\stackrel{\eqref{eq: Step 2.1 1 necessity}}{\leq}\sup\left\{\sum_{i=1}^{m-1}|D^s u^\nu_y|((t_i,t_{i+1}))+ \frac{1}{2}[u^\nu_y](t_i)+ \frac{1}{2}[u^\nu_y](t_{i+1}) :\, m\geq 2,\, a<t_1<\dots<t_m<b  \right\}\\
&\leq |D^s u^\nu_y|((a,b))<\infty.
\end{align*} 
Repeating the same calculation for all the connected components of $\Omega^\nu_y$ we deduce that 
\begin{align}\label{eq_eddaje!}
\textnormal{eV}(b^\nu_y,\Omega^\nu_y)\leq |D^s u^\nu_y|(\Omega^\nu_y)<\infty.
\end{align}
Thus, $b^\nu_y$ is a bounded function in $\Omega^\nu_y$, hence since $\mathcal{L}^1(\Omega^\nu_y)<\infty$ for $\mathcal{L}^{n-1}$-a.e. $y \in \Omega_\nu$, $b^\nu_y\in L^1(\Omega^\nu_y)$, which together with $\textnormal{eV}(b^\nu_y,\Omega^\nu_y)<\infty$ implies that $b^\nu_y\in BV(\Omega^\nu_y)$. By generality of $\nu\in \mathbb{S}^{n-1}$, and of $y \in \Omega_\nu$, thanks to \cite[Proposition 4.35]{AFP} we conclude that $b\in GBV(\Omega)$. This concludes step 1.\\
\textbf{Step 2.} Goal of this step is to prove that $b$ satisfies conditions i), and ii) of Definition \ref{def_minimally singular functions}. Let us fix $M>0$. Similarly to what we have done at the beginning of step 1, let $\nu \in \mathbb{S}^{n-1}$ be such that $|D_\nu u|(\Omega)<\infty$. In particular, fixing any $y\in G_\nu$, thanks to \eqref{eq_eddaje!}, and repeating the argument presented in step 1 we get $|D (b^M)^\nu_y|(I)\leq |D^s u^\nu_y|(I)$ for every $I\subset \Omega^\nu_y$ open set, thus by \ Borel regularity we have that $|D (b^M)^\nu_y|(G)\leq |D^s u^\nu_y|(G)$ for every $G\subset \Omega^\nu_y$ Borel. As a consequence of that, we have that $D (b^M)^\nu_y <<|D^s u^\nu_y|$, and so by the Radon-Nikodym theorem there exists $f_{(M,\nu,y)}:\Omega^\nu_y\to [-1,1]$ with $f_{(M,\nu,y)}\in L^1(\Omega^\nu_y,|D^s u^\nu_y|)$, such that $D (b^M)^\nu_y = f_{(M,\nu,y)}|D^s u^\nu_y|$, namely
\begin{align}\label{ma che ne so 1}
D (b^M)^\nu_y(G)=\int_{G}f_{(M,\nu,y)}(t)\,d |D^s u^\nu_y|(t),\quad \forall\,G\subset \Omega^\nu_y \textnormal{ Borel}.
\end{align}
Let us observe that a similar expression to \eqref{ma che ne so 1} can be obtained with $d D^s u^\nu_y$ in place of $d |D^s u^\nu_y|$. Indeed, by polar decomposition of $|D^s u^\nu_y|$ we get that
$$
D^s u^\nu_y(G)=\int_{G}\frac{d D^s u^\nu_y}{d|D^s u^\nu_y|}(t)\,d|D^s u^\nu_y|(t),
$$
where $\frac{d D^s u^\nu_y}{d|D^s u^\nu_y|}(t)\in \mathbb{S}^0$, i.e. $\frac{d D^s u^\nu_y}{d|D^s u^\nu_y|}(t)\in \{-1,+1  \}$ for $|D^s u^\nu_y|$-a.e. $t\in \Omega^\nu_y$. Let us set
$$
g_{(M,\nu,y)}(t):= \frac{d D^s u^\nu_y}{d|D^s u^\nu_y|}(t) f_{(M,\nu,y)}(t).
$$ 
Thus, we get 
\begin{align}\label{ma che ne so 2}
D (b^M)^\nu_y(G)=\int_{G}g_{(M,\nu,y)}(t)\,d D^s u^\nu_y(t),\quad \forall\,G\subset \Omega^\nu_y \textnormal{ Borel}.
\end{align}
Hence, for every $\varphi\in C^1_c(\Omega)$, thanks to \cite[Theorem 3.107, relation (3.104)]{AFP} applied to $b^M$, and thanks to  \cite[Theorem 3.108, relation (3.108)]{AFP} applied to $u$ we get,
\begin{align*}
&\int_{\Omega}\varphi(x)\,dD_\nu (b^M)(x)=-\int_{\Omega}b^M(x)\frac{\partial \varphi}{\partial \nu}(x)\, dx = -\int_{\Omega_\nu}dy \int_{\Omega^\nu_y} (b^M)^\nu_y (t) (\varphi^\nu_y)'(t)\,dt\\
&=\int_{\Omega_\nu}dy\int_{\Omega^\nu_y}(\varphi^\nu_y)(t)\,d D(b^M)^\nu_y(t)\stackrel{\eqref{ma che ne so 2}}{=}\int_{\Omega_\nu}dy\int_{\Omega^\nu_y}(\varphi^\nu_y)(t)g_{(M,\nu,y)}(t)\,d D^s(u)^\nu_y(t)\\
&=\int_{\Omega} \varphi(x) g_{(M,\nu)}(x)\, dD^s_\nu u(x),
\end{align*}
where $g_{(M,\nu)}:\Omega\to [-1,1]$ is defined as $g_{(M,\nu)}(x):= g_{(M,\nu,y)}(t)$ whenever $x=y+t\nu$. Thus, we get 
\begin{align}\label{eq: step 2.2 star 1}
\int_{\Omega}\varphi(x)\,d D_\nu (b^M)(x)= \int_{\Omega}\varphi(x) g_{(M,\nu)}(x)\, d D^s_\nu u(x),
\end{align}
for every $\varphi\in C^1_c(\Omega)$.
Now let us set $D_{e_i}=D_i$, $g_{(M,e_i)}=g_{(M,i)}$ for $i=1,\dots,n$, and for every $\varphi\in C^1_c(\Omega;\mathbb{R}^n)$ we call $\tilde{\varphi}:\Omega\to \mathbb{R}^n$ the bounded Borel function defined as
$$
\tilde{\varphi}(x)=(g_{(M,1)}(x)\varphi_1(x),\dots,g_{(M,n)}(x)\varphi_n(x)),\quad \forall\, x\in \Omega.
$$
Then, thanks to \eqref{eq: step 2.2 star 1} that was obtained for a generic $\nu\in \mathbb{S}^{n-1}$, a direct computation shows that
\begin{align}\label{eq: step 2.2 star 2}
\int_{\Omega}\varphi(x)\cdot dD(b^M)(x)=\int_{\Omega}\tilde{\varphi}(x)\cdot dD^s u(x)\leq |D^s u|(\Omega),
\end{align}
for every $\varphi\in C^1_c(\Omega;\mathbb{R}^n)$, where for the last inequality we used that $|\tilde{\varphi}|\leq 1$. Passing to the $\sup$ on the left hand side among all $\varphi\in C^1_c(\Omega;\mathbb{R}^n)$ with $|\varphi|\leq 1$, gives that $|D(b^M)|(\Omega)\leq |D^s u|(\Omega)$. Repeating the same argument for every open set of $\Omega$, and by Borel regularity we finally get that
\begin{align}\label{eq: |Db^m|<< 1/2 |D^s v|}
|D(b^M)|(B)\leq |D^s u|(B)\quad \forall\, B\subset \Omega,
\end{align}
which proves conditions i), ii) of Definition \ref{def_minimally singular functions} for the function $(b^M)$. Thanks to \cite[Theorem 4.34]{AFP} we deduce that the same conditions hold true also for the function $b$. This concludes step 2.\\
\noindent
\textbf{Step 3.} In this step we will prove that $b$ satisfies condition  iii) of Definition \ref{def_minimally singular functions}. Let $M>0$. By \eqref{eq: |Db^m|<< 1/2 |D^s v|} we have that $D^c(b^M)<<|D^c u|$, and so by Radon-Nikodym theorem there exists $\mathcal{G}_{M}:\Omega\to\mathbb{R}^n$ with $\mathcal{G}_{M}\in L^1(\Omega,|D^c u|;\mathbb{R}^n)$,  such that $D^c(b^M)=\mathcal{G}_M|D^c u|$, namely
\begin{align}\label{eq: step 2.2 star 4.1}
D^c(b^M)(B)=\int_{B}\mathcal{G}_M(x)\, d|D^c u|(x) \quad \forall\,B\subset \Omega\textnormal{ Borel}.
\end{align}
Let us observe that thanks to \eqref{eq: |Db^m|<< 1/2 |D^s v|}, we have $|\mathcal{G}_M(x)|\leq 1$ for $|D^c u|$-a.e. $x\in\Omega$. Let us now observe the following fact: let $(E_i)_{i\in I}$ be a partition of $\Omega$ made of Borel sets, where $I\subset \mathbb{N}$ is a finite set of indexes, let $(c_i)_{i\in I}$ with $c_i\in\mathbb{R}$ for every $i\in I$ and $\sup_i|c_i|<\infty$, and let $(\eta_i)_{i\in I}$ be vectors in $\mathbb{S}^{n-1}$. Let $\varphi:\Omega\to \mathbb{R}^n$ be the bounded Borel function defined as
\begin{align}\label{piecewise constant direction function}
\varphi(x):= \sum_{i\in I}c_i\eta_i \chi_{E_i}(x)\quad \forall\,x\in \Omega,
\end{align}
then, we get
\begin{align}\label{distributional directional condition}
\int_{\Omega}|\varphi(x)\cdot dD^c b^M(x)|\leq  \int_{\Omega}|\varphi(x)\cdot dD^c u(x) |.
\end{align}
Indeed, thanks to \eqref{eq_eddaje!}, \cite[Theorem 3.103]{AFP}, and \cite[Theorem 3.108, relation (3.108)]{AFP} we get
\begin{align*}
|D^c_\nu (b^M)|(U)=\int_{U_\nu} |D^c (b^M)^\nu_y|(U^\nu_y)\,dy\stackrel{\eqref{eq_eddaje!}}{\leq} \int_{U_\nu} |D^c u^\nu_y|(U^\nu_y)\,dy=|D_\nu^c u|(U)\quad\forall\, U\subset \Omega \textnormal{ open},
\end{align*}
which by Borel regularity implies that
\begin{align}\label{eq: V_v (b) < 1/2 |D^s v| Borel}
|D^c_\nu (b^M)|(B)\leq |D_\nu^c u|(B)\quad \forall\, B\subset \Omega \textnormal{ Borel}.
\end{align}
Thus,
\begin{align*}
&\int_{\Omega}|\varphi(x)\cdot dD^c b^M(x)|= \int_{\Omega}\left|\left( \sum_{i\in I}c_i\eta_i \chi_{E_i}(x) \right)\cdot dD^c b^M(x)   \right| = \int_{\Omega} \sum_{i\in I}\left|\left(c_i\eta_i \chi_{E_i}(x) \right)\cdot dD^c b^M(x)   \right|\\ &=\int_{\Omega} \sum_{i\in I}\left|c_i\eta_i \cdot dD^c b^M(x)\chi_{E_i}(x)   \right|\stackrel{\eqref{eq: V_v (b) < 1/2 |D^s v| Borel}}{\leq}\int_{\Omega} \sum_{i\in I}\left|c_i\eta_i \cdot dD^c u(x)\chi_{E_i}(x)   \right| = \int_{\Omega}|\varphi(x)\cdot dD^c u(x)|.
\end{align*}
Thus, let us show that for every $\psi:\Omega \to \mathbb{R}^n$ bounded Borel function we have that
\begin{align}\label{what we need for Cantor}
\int_{\Omega} |\psi(x)\cdot dD^cb^M(x)| \leq \int_{\Omega}|\psi(x)\cdot dD^c u(x)|.
\end{align}
Indeed, thanks to the \emph{Simple Approximation Theorem} (see for instance \cite[Chapter 18]{Royden}) applied to each component of $\psi$, there exists a sequence of functions $\psi_h:\Omega\to \mathbb{R}^n$ as in \eqref{piecewise constant direction function}, such that $|\psi_h(x)|\leq |\psi(x)|$ for $|D^c u|$-a.e. $x\in \Omega$, such that $\lim_{h\to \infty}\psi_h(x)=\psi(x)$ for $|D^c u|$-a.e. $x\in \Omega$. Thus, by the Dominated Convergence Theorem we get that
\begin{align*}
&\int_{\Omega}|\psi(x)\cdot dD^c u(x)|= \int_{\Omega}\left|\psi(x)\cdot \frac{dD^c u}{d|D^c u|}(x)\right| d|D^c u|(x)= \int_{\Omega}\lim_{h\to \infty}\left|\psi_h(x)\cdot \frac{dD^c u}{d|D^c u|}(x)\right| d|D^c u|(x)\\
&=\lim_{h\to \infty}\int_{\Omega}\left|\psi_h(x)\cdot \frac{dD^c u}{d|D^c u|}(x)\right| d|D^c u|(x)\stackrel{\eqref{distributional directional condition}}{\geq} \lim_{h\to \infty}\int_{\Omega}\left|\psi_h(x)\cdot \frac{dD^c b^M}{d|D^cb^M|}(x)\right| d|D^c b^M|(x) \\
&=\int_{\Omega}\left|\psi(x)\cdot \frac{dD^c b^M}{d|Db^M|}(x)\right| d|D^c b^M|(x)= \int_{\Omega}|\psi(x)\cdot dD^c b^M(x)|,
\end{align*}
which proves \eqref{what we need for Cantor}. Thus, let $\psi:\Omega \to \mathbb{S}^{n-1}$ be a Borel function such that 
$$
\psi(x)\cdot \frac{dD^c u}{d|D^c u|}(x)=0\quad \textnormal{for }|D^c u|\textnormal{-a.e. }x\in\Omega.
$$
Thanks to \eqref{what we need for Cantor} we deduce that 
$$
\psi(x)\cdot \frac{dD^c b^M}{d|D^c b^M|}(x)=0\quad \textnormal{for }|D^c b^M|\textnormal{-a.e. }x\in\Omega.
$$
Thus, by generality of $\psi$, we deduce that
$$
\left|\frac{dD^c b^M}{d|D^c b^M|}(x)\cdot\frac{dD^c u}{d|D^c u|}(x)\right|=1  \quad \textnormal{for }|D^c b^M|\textnormal{-a.e. }x\in\Omega.
$$
Let us set
\begin{align*}
K_{b^M}&:=\left\{x\in \textnormal{spt}(|D^c u|):\, \frac{dD^c b^M}{d|D^c b^M|}(x) \in \mathbb{S}^{n-1}   \right\},\\
K_u&:=\left\{x\in \textnormal{spt}(|D^c u|):\, \frac{dD^c u}{d|D^c u|}(x) \in \mathbb{S}^{n-1}   \right\}.
\end{align*}
Let us observe that thanks to \eqref{eq: |Db^m|<< 1/2 |D^s v|} we have that $|D^c b^M|(K_{b^M}\setminus K_u)=0$. Let us define $\textnormal{sign}_M:\Omega \to [-1,1]$ as
\begin{align*}
\textnormal{sign}_M(x):=
\begin{cases}
\frac{dD^c b^M}{d|D^c b^M|}(x)\cdot \frac{dD^c u}{d|D^c u|}(x)\quad &\mbox{if } x\in K_{b^M}\cap K_u,\vspace*{0.2cm}\\
0\quad &\mbox{otherwise}.
\end{cases}
\end{align*}
Then, for every $B\subset \Omega$ Borel set we get
\begin{align*}
&D^c b^M (B)=\int_{B} dD^c b^M(x)=\int_{B}\frac{dD^c b^M}{d|D^c b^M|}(x)\, d|D^c b^M|(x)\stackrel{\eqref{eq: step 2.2 star 4.1}}{=}\int_{B} |\mathcal{G}_M(x)| \frac{dD^c b^M}{d|D^c b^M|}(x) d|D^c u|(x)\\
&=\int_{B}|\mathcal{G}_M(x)|\textnormal{sign}_M(x)\frac{dD^c u}{d|D^c u|}(x) d|D^c u|(x)= \int_{B}|\mathcal{G}_M(x)|\textnormal{sign}_M(x) dD^c u(x).
\end{align*}
We call
\begin{align*}
f_M(x) = |\mathcal{G}_M(x)|\textnormal{sign}_M(x)\quad\forall
\, x\in \Omega,
\end{align*}
and so thanks to the above calculation we get that
\begin{align}\label{eq_Cantor almost there 1}
D^c b^M (B):=\int_B f_M(x)\, d D^c u(x)\quad \forall\, B\subset \Omega\textnormal{ Borel}.
\end{align}
Let us define 
\begin{align*}
f(x):= \sup_{M>0} f_M(x)\quad \forall\, x\in \Omega.
\end{align*}
In order to conclude the step we need to show that for every $M>0$ we have that
\begin{align}\label{eq_Cantor almost there 2}
f(x)=f_M(x) \quad\textnormal{for }|D^c u|\textnormal{-a.e. }x\in \{b\leq M  \}^{(1)}.
\end{align}
Indeed putting together \eqref{eq_Cantor almost there 1} with \eqref{eq_Cantor almost there 2} we get condition iii) of Definition \ref{def_minimally singular functions}. In order to prove \eqref{eq_Cantor almost there 2} let us prove that for every $\overline{M}>M>0$ we have that
\begin{align}\label{f_ barM = f_M in (b<=M)}
f_{\overline{M}}(x)= f_{M}(x)\quad \textnormal{for }|D^c u|\textnormal{-a.e. }x\in \{|b|\leq M  \}^{(1)} .
\end{align}
Indeed, thanks to Lemma \ref{lem: cantor total variation of v on v=0 is null} and Lemma \ref{lem: 2.3 FIlippo page 18} for every $M>0$ such that $\{|b|\leq M  \}$ is a set of finite perimeter, we have that 
\begin{align}\label{eq_consequence of Lemma 2.2 Cagnetti}
D^c b^{\overline{M}}\mres \{|b|\leq M  \}^{(1)} = D^c b^M\mres \{|b|\leq M \}^{(1)}.
\end{align}
Thanks to \eqref{eq_consequence of Lemma 2.2 Cagnetti} together with \eqref{eq: step 2.2 star 4.1} we deduce that 
\begin{align}\label{e one}
\mathcal{G}_{\overline{M}}(x)=\mathcal{G}_M(x)\quad \textnormal{for }|D^c u|\textnormal{-a.e. }x\in \{|b|\leq M \}^{(1)},
\end{align}
while again thanks to \eqref{eq_consequence of Lemma 2.2 Cagnetti} and by definition of the function $\textnormal{sign}_M$ we also deduce that 
\begin{align}\label{e two}
\textnormal{sign}_{\overline{M}}(x)=\textnormal{sign}_M(x)\quad \forall\, x\in \{|b|\leq M \}^{(1)}.
\end{align}
Putting together \eqref{e one} and \eqref{e two} we get that \eqref{f_ barM = f_M in (b<=M)} holds true, which implies \eqref{eq_Cantor almost there 2}. This concludes step 3. Putting together step 1, step 2, and step 3 we conclude.
\end{proof}
\noindent
For the seek of completeness, in the following we gather some properties of the function $b:\Omega\to [0,\infty)$ defined in \eqref{def: the smart baricenter}. We start presenting a result that essentially says that if $\gamma\in \Gamma_\Omega(u)$ then the restriction of $b$ to $\gamma$, namely $b_\gamma$, satisfies some regularity properties, and in particular Proposition \ref{prop_fine properties of ugamma}
 and Lemma \ref{lem_calculation |u_gamma (a0)-u_gamma(am)|} applies to $b_\gamma$.

\begin{lemma}\label{lem_bgamma is BV}
Let $u\in BV(\Omega)$, let $\bar{x}\in \tilde{\Omega}_u$, and let $b:\Omega\to [0,\infty)$ be defined as in \eqref{def: the smart baricenter}. Then the following statements hold true:
\begin{itemize}
\item[i)] if $\gamma\in \Gamma_\Omega(u)$ then $b_\gamma\in BV(I_\gamma^\circ)$ where $b_\gamma(t)=\textnormal{SVD}_{u,\Omega}(\bar{x},\gamma(t))$ for every $t \in I_\gamma$, and the three conditions of Proposition \ref{prop_fine properties of ugamma} applies to $b_\gamma$. In particular 
\begin{align}\label{eq_eddaje super}
|D b_\gamma|(I_\gamma^\circ)\leq |D^s u_\gamma|(I_\gamma^\circ)\quad \forall\, \gamma\in \Gamma_\Omega(u).
\end{align}
 \vspace*{0.2cm}
\item[ii)] if $x_1,x_2\in \tilde{\Omega}_u$, and $\gamma\in \Gamma_\Omega(u)$ connects them in the sense of Definition \ref{def_gamma connects two points} then
\begin{align}\label{eq_calculation |b_gamma (a0)-b_gamma(am)|}
|b(x_1)-b(x_2)|\leq |Db_\gamma|(I_\gamma^\circ).
\end{align}
\end{itemize}
\end{lemma}
\begin{proof}
We divide the proof in two steps.\\
\textbf{Step 1.} In this step we prove the validity of $i)$. Let $\gamma\in \Gamma_\Omega(u)$, and consider $b_\gamma(t):= b\Low (\gamma(t))$ for every $t\in I_\gamma^\circ$. Thanks to relation \eqref{eq_eddaje!}, and Proposition \ref{prop_GBV slicing} for $GBV$ functions (see \cite[Proposition 4.35]{AFP}) a careful inspection of the proof of Proposition \ref{prop_fine properties of ugamma} shows that we can prove the validity of the three conditions listed in Proposition \ref{prop_fine properties of ugamma} for the function $b_\gamma$, which implies also the validity of \eqref{eq_eddaje super}. Since $b\Low(x)=b(x)$ for $\mathcal{L}^n$-a.e. $x\in \Omega$, and thanks to \eqref{eq_eddaje!} we conclude that $b_\gamma(t)=\textnormal{SVD}_{u,\Omega}(\bar{x},\gamma(t))$ for $\mathcal{L}^1$-a.e. $t\in I_\gamma^\circ$. This concludes the first step.\\
\textbf{Step 2.} We conclude. Let us first observe that if $x\in \tilde{\Omega}_u$, given any $\gamma\in \Gamma_\Omega(u)$ such that $\gamma(t)=x$  for some $t\in I_\gamma^\circ$, then thanks to step 1 we can apply relation ii) of Proposition \ref{prop_fine properties of ugamma} to $b_\gamma$ which implies that $[b_\gamma](t)=[b](x)=0$ which, thanks to \cite[Theorem 3.28]{AFP} implies that $t$ is a point of approximate continuity for $b_\gamma$. Then, following verbatim the argument presented in the proof of Lemma \ref{lem_calculation |u_gamma (a0)-u_gamma(am)|} we conclude.
\end{proof}
\noindent
Motivated by Lemma \ref{lem_bgamma is BV} we define the notion of \emph{singular vertical distance} for the function $b$ restricted to the curves in $\Gamma_\Omega(u)$, namely
\begin{align}\label{eq: def SVD^*}
\textnormal{SVD}^*_{b,\Omega}(\bar{x},x):=
\begin{cases}
\inf \left\{|D b_\gamma|(I_\gamma^\circ):\, \gamma\in \Gamma_\Omega(u),\,\gamma\textnormal{ connects }\bar{x},x  \right\}\quad &\mbox{if } \bar{x}\neq x;\vspace*{0.2cm}\\
0 \quad &\mbox{if }\bar{x}=x.
\end{cases}
\end{align}

\begin{lemma}\label{lem_SVDb=SVDu}
Let $u\in BV(\Omega)$, let $\bar{x} \in \tilde{\Omega}_u$, and let $b:\Omega\to [0,\infty)$ as in \eqref{def: the smart baricenter}. Then
\begin{align}
\textnormal{SVD}^*_{b,\Omega}(\bar{x},x) = \textnormal{SVD}_{u,\Omega}(\bar{x},x)\quad \forall\, x\in \tilde{\Omega}_u.
\end{align}
\end{lemma}
\begin{proof}
Indeed, let $x\in \tilde{\Omega}_u$, and let $\gamma\in \Gamma_\Omega(u)$ be such that it connects $\bar{x}$ and $x$. Then
\begin{align*}
\textnormal{SVD}_{u,\Omega}(\bar{x},x)\stackrel{\eqref{def: the smart baricenter}}{=} b(x)= |b(x)-b(\bar{x})|\stackrel{\eqref{eq_calculation |b_gamma (a0)-b_gamma(am)|}}{\leq}|D b_\gamma|(I_\gamma^\circ).
\end{align*}
Passing to the $\inf$ on the right hand side of the above relation among all $\gamma\in \Gamma_\Omega(u)$ connecting $\bar{x}$ and $x$ we get that $\textnormal{SVD}_{u,\Omega}(\bar{x},x)\leq  \textnormal{SVD}^*_{b,\Omega}(\bar{x},x)$. Let us prove the reverse inequality. Indeed,
\begin{align*}
\textnormal{SVD}^*_{b,\Omega}(\bar{x},x)\leq |D b_\gamma|(I^\circ_\gamma) \stackrel{\eqref{eq_eddaje super}}{\leq} |D^s u_\gamma|(I_\gamma^\circ).
\end{align*}
Passing to the $\inf$ on the right hand side of the above relation among all $\gamma\in \Gamma_\Omega(u)$ connecting $\bar{x}$ and $x$ we get that $\textnormal{SVD}_{u,\Omega}(\bar{x},x)\geq  \textnormal{SVD}^*_{b,\Omega}(\bar{x},x)$. This concludes the proof.
\end{proof}

\begin{lemma}\label{lem_b non scende lungo cammini ottimi}
Let $u\in BV(\Omega)$, let $\bar{x} \in \tilde{\Omega}_u$, and let $b:\Omega\to [0,\infty)$ as in \eqref{def: the smart baricenter}. Let $x\in \tilde{\Omega}_u$ and let $(\gamma_i)_i\subset \Gamma_\Omega(u)$ be such that $\lim_{i\to \infty}|D u_{\gamma_i}|(I^\circ_{\gamma_i})= \textnormal{SVD}_{u,\Omega}(\bar{x},x)$. Lastly, we define the set $N_{b,i}\subset I^\circ_{\gamma_i}$ as 
\begin{align}\label{eq_Nb,gamma}
N_{b,i}:=\left\{t \in I^\circ_{\gamma_i}:\, \frac{d Db_{\gamma_i}}{d |Db_{\gamma_i}|}(t)=-1 \right\},\quad \textnormal{for }i\in \mathbb{N}.
\end{align}
Then
\begin{align}\label{eq_b sale lungo i cammini ottimi}
\lim_{i\to \infty} |D b_{\gamma_i}|(I_{\gamma_i}^\circ)(N_{b,i})=0.
\end{align}
\end{lemma}
\begin{proof}
Let $x \in \tilde{\Omega}_u$ and let $\gamma\in \Gamma_\Omega(u)$ be connecting $\bar{x}$ and $x$ in the sense of Definition \ref{def_gamma connects two points}. Then,
\begin{align}\label{eq_siii?}
0\leq \textnormal{SVD}_{b,\Omega}^*(\bar{x},x)= \textnormal{SVD}_{u,\Omega}(\bar{x},x)=b(x)= b(x)-b(\bar{x})=D b_\gamma (I_\gamma^\circ)= |D b_\gamma (I_\gamma^\circ)|,
\end{align}
where for the first equality sign we used Lemma \ref{lem_SVDb=SVDu}, while for the second last equality sign we followed the argument used in step 3 of Lemma \ref{lem_calculation |u_gamma (a0)-u_gamma(am)|} applied to $b_\gamma$ (which can be done as explained in ii) of Lemma \ref{lem_bgamma is BV}). Moreover let us observe that calling with $N_{b,\gamma}$ the set defined in \eqref{eq_Nb,gamma} with respect to $\gamma$, by standard considerations we get
\begin{align}\label{eq_siii?1}
&|D b_\gamma|(I_\gamma^\circ)= |D b_\gamma|(I_\gamma^\circ\setminus N_{b,\gamma})+ |D b_\gamma|(N_{b,\gamma})= D b_\gamma(I_\gamma^\circ\setminus N_{b,\gamma})- D b_\gamma(N_{b,\gamma}) \nonumber\\
& D b_\gamma(I_\gamma^\circ)- 2 D b_\gamma(N_{b,\gamma})= |D b_\gamma(I_\gamma^\circ)|+ 2 |D b_\gamma|(N_{b,\gamma})\stackrel{\eqref{eq_siii?}} {=} \textnormal{SVD}^*_{b,\Omega}(\bar{x},x) + 2 |D b_\gamma|(N_{b,\gamma}).
\end{align}
Let us now consider a sequence of paths as in the statement of Lemma \ref{lem_b non scende lungo cammini ottimi}, so that for every sequence of positive real numbers $(\epsilon_i)_{i\in\mathbb{N}}$ such that $\lim_{i\to \infty} \epsilon_i =0$, up to relabelling, we get
\begin{align}\label{eq_siii?2}
|D b_{\gamma_i}|(I_{\gamma_i}^\circ) \leq \textnormal{SVD}^*_{b,\Omega}(\bar{x},x) + \epsilon_i,\quad \forall\,i\in \mathbb{N}.
\end{align} 
Thus,
\begin{align*}
\textnormal{SVD}^*_{b,\Omega}(\bar{x},x) + 2 |D b_\gamma|(N_{b,i})\stackrel{\eqref{eq_siii?1}}{=}|D b_\gamma|(I_{\gamma_i}^\circ) \stackrel{\eqref{eq_siii?2}}{\leq}\textnormal{SVD}^*_{b,\Omega}(\bar{x},x) + \epsilon_i \quad \forall\, i\in \mathbb{N}
\end{align*}
and so
\begin{align*}
2 |D b_\gamma|(N_{b,i}) \leq \epsilon_i \quad \forall\, i\in \mathbb{N}.
\end{align*}
This concludes the proof.
\end{proof}
\noindent
Roughly speaking, given a generic point $x\in \tilde{\Omega}_u$, the above result is telling us that along optimal sequences of paths connecting $\bar{x},x \in \tilde{\Omega}_u$, the function $b$ restricted to these paths, tends to be not decreasing. 

\noindent
Let us now conclude this section with a couple of open problems. The first one is about regularity properties of the function $b:\Omega\to [0,\infty)$ defined as in \eqref{def: the smart baricenter}, and it asks if it is possible to show that actually $b\in BV(\Omega)$ rather than simply $b\in GBV(\Omega)$. The second problem instead asks whether it is possible to decompose every function $u\in BV(\Omega)$ as a sum of two functions $\bar{u},b\in BV(\Omega)$ such that $\bar{u}$ is \emph{minimally singular}, both functions satisfy properties \eqref{eq_Salto b< 1/2 salto v}, and \eqref{eq_Cantor b< 1/2 Cantor v}, and in particular
\begin{align*}
\nabla \bar{u}(x)= \nabla u(x), \quad \nabla b (x)=0 \quad \textnormal{for }\mathcal{L}^n\textnormal{-a.e. }x\in \Omega.
\end{align*}
Roughly speaking this latter problem if answered positively, could be interpreted as a sort of weak $n$-dimensional generalization of the decomposition theorem for functions of bounded variation valid for $n=1$ that says that every function of bounded variation in an open and bounded interval, can be written as sum of its absolutely continuous part and its singular part (see for instance \cite[Corollary 3.33]{AFP}), where in our context we replace \emph{absolutely continuous} with \emph{minimally singular}. 
\begin{open}\label{open 1}
Let $u\in BV(\Omega)$, let $\bar{x} \in \tilde{\Omega}_u$, and let $b:\Omega\to [0,\infty)$ as in \eqref{def: the smart baricenter}. Is it always true that $b\in BV(\Omega)$? Observe that thanks to Lemma \ref{lem_SVD is in GBV} it will be sufficient to show that $b\in L^1(\Omega)$.
\end{open}

\begin{open}\label{open 2}
Let $u\in BV(\Omega)$. Prove or disprove that we can always decompose $u$ as
\begin{align*}
u(x)=\bar{u}(x) + b(x)\quad\textnormal{for }\mathcal{L}^n\textnormal{-a.e. }x\in \Omega,
\end{align*}
where $\bar{u}\in BV(\Omega)$ is \emph{minimally singular}, and $b\in BV(\Omega)$ satisfies conditions $i),\,ii),\,iii)$ of Definition \ref{def_minimally singular functions} with respect to  the function $u$.
\end{open}

\section{Characterization result for minimally singular functions}\label{section_minimally singular characterized via SVD}
\noindent
In this section we prove the main result of this work, which is a geometric characterization for \emph{minimally singular} functions using the notion of singular vertical distance that was introduced in the previous section. Let us first prove that when $n=1$ then \emph{minimally singular} functions coincide with absolutely continuous functions.

\begin{proposition}\label{prop_n=1}
Let $\Omega\subset \mathbb{R}$ be as in \eqref{assumptions on Omega}, and let $u\in BV(\Omega)$. Then $u$ is \emph{minimally singular} if and only if $u$ is absolutely continuous.
\end{proposition}
\begin{proof}
Since $\Omega\subset \mathbb{R}$ satisfies condition \eqref{assumptions on Omega} we have that $\Omega=(a,b)$ for some $a,b\in \mathbb{R}$. Let $u$ be absolutely continuous, and let $b\in GBV(\Omega)$ satisfying \eqref{eq_nabla b = 0}, \eqref{eq_Salto b< 1/2 salto v}, and \eqref{eq_Cantor b< 1/2 Cantor v}. Then, thanks to the Coarea formula applied to $b$ we get that 
\begin{align*}
\int_{\mathbb{R}}P\left(\{ b>t \};\Omega  \right)\, dt=0.
\end{align*}
So $b$ must be equivalent to a constant function in $\Omega$, which by generality of $b$ implies that $u$ is \emph{minimally singular}. Vice versa, suppose $u$ is \emph{minimally singular}, and let $b:\Omega\to \mathbb{R}$  be defined as $b(t):= |D^s u|((a,t))$ for every $t\in \Omega$. Then, thanks to \cite[Theorem 3.30]{AFP} we know that $b\in BV(\Omega)$, while by Fubini Theorem we get that $Db = |D^s u|$. Indeed, let $\varphi \in C^1_c(\Omega)$ then
\begin{align*}
&\int_a^b b(t) \varphi'(t)\,dt = \int_a^b \int_a^t \varphi'(t) \,d|D^s u|(s)\, dt = \int_a^b \int_a^b \varphi'(t) \chi_{(a,t)}(s) \,d|D^s u|(s)\, dt\\
&=\int_a^b \int_a^b \varphi'(t) \chi_{(a,t)}(s) \, dt \,d|D^s u|(s) = \int_a^b \int_s^b \varphi'(t) \, dt \,d|D^s u|(s)=- \int_a^b \varphi(s)\, d|D^s u|(s).
\end{align*}
Thus, since $Db = |D^s u|$, we get that $b$ satisfies \eqref{eq_nabla b = 0}, \eqref{eq_Salto b< 1/2 salto v}, and \eqref{eq_Cantor b< 1/2 Cantor v}. At the same time, since $u$ is \emph{minimally singular}, $b$ must be constant, which implies that $|D^s u|((a,b))=0$. This implies that $u$ is absolutely continuous and so we conclude the proof.
\end{proof}
\noindent
We are now ready to prove  Theorem \ref{thm_characterization minimally singular via SVD} which represents the main result of this work.

\begin{proof}[Proof of Theorem \ref{thm_characterization minimally singular via SVD}]
We divide the proof into two steps.\\
\textbf{Step 1 (Sufficiency).} In this first part of the proof we will show that if there exists $\bar{x}\in \tilde{\Omega}_u$ such that $\textnormal{SVD}_{u,\Omega}(\bar{x},x)=0$ for $\mathcal{L}^n$-a.e. $ x\in \tilde{\Omega}_u$, then $u$ is \emph{minimally singular}. We divide the argument in two sub-steps.\\
\textbf{Step 1.1.} Let $b\in GBV(\Omega)$ be a function satisfying condition \eqref{eq_nabla b = 0}, \eqref{eq_Salto b< 1/2 salto v}, and \eqref{eq_Cantor b< 1/2 Cantor v}. Goal of this first step is to prove that for every $M>0$, for every $\nu\in \mathbb{S}^{n-1}$, and for $\mathcal{L}^{n-1}$-a.e. $y\in \Omega_\nu$ we have 
\begin{align}\label{eq:step1.1 sufficiency}
\int_{G}d|D (b^M)^\nu_y|(t)\leq \int_{G}d|D^s (u)^\nu_y|(t)  \quad \forall\, G\subset \Omega^\nu_y \textnormal{ Borel}.
\end{align}
Thanks to the assumptions on the function $b$ we can construct a Borel function $f_{M}:\Omega \to [-1,1]$ defined as
\begin{align*}
f_{M}(x)=
\begin{cases}
 (\nu_{b^M}(x) \cdot \nu_u(x)) \frac{[b^M](x)}{[u](x)}\quad &\mbox{for } \mathcal{H}^{n-1}\textnormal{-a.e. } x\in J_{b^M},\vspace{0.2cm}\\
f^c(x)\quad &\mbox{for } |D^c u|\textnormal{-a.e. }x \in \{|b|<M  \}^{(1)},\vspace{0.2cm}\\
0\quad &\mbox{otherwise,}
\end{cases}
\end{align*}
where $f^c$ is the Borel function appearing in \eqref{eq_Cantor b< 1/2 Cantor v}. Thus, 
\begin{align}\label{eq: Db=f D^sv}
D b^M (B)=\int_B f_{M}(x)\, dD^s u(x),\quad \forall\, B\subset \Omega \textnormal{ Borel}.
\end{align}
Thanks to \eqref{eq: Db=f D^sv}, \cite[Remark 4.8]{maggiBOOK}, and by definition of directional distributional derivative $D_\nu u =D u\cdot\nu$, for every $\nu \in \mathbb{S}^{n-1}$ we get that
$$
|D_\nu b^M|(B)= \int_{B}|f_{M}(x)|\, d|D_\nu^s u|(x) \quad \forall\, B\subset \Omega \textnormal{ Borel}.
$$
Applying \cite[Theorem 3.103]{AFP}, and thanks to the above relation we then get
\begin{align*}
\int_{B_\nu}dy\int_{B^\nu_y}d|D (b^M)^\nu_y|(t)\leq  \int_{B_\nu}dy\int_{B^\nu_y}d|D^s (u)^\nu_y|(t)  \quad \forall\, B\subset \Omega \textnormal{ Borel}.
\end{align*}
From the above relation, \eqref{eq:step1.1 sufficiency} follows.\\
\textbf{Step 1.2.} We conclude. By assumption there exists $\bar{x}\in \tilde{\Omega}_u$ such that $\textnormal{SVD}_{u,\Omega}(\bar{x},x)=0$ for $\mathcal{L}^n$-a.e. $ x\in \tilde{\Omega}_u$,  thus thanks to Remark \ref{rem: notion of connectedness} we have that $\mathcal{L}^n (\{\bar{x}\}_{u,\Omega})= \mathcal{L}^n(\Omega)$ where $\{\bar{x}\}_{u,\Omega}$ was defined in \eqref{def: (x) equivalence class}. Let $x\in \{\bar{x}\}_{u,\Omega}$ be such that $\bar{x}\neq x$. Let $M>0$, then since $\mathcal{L}^n(\Omega)<\infty$, we have that $b^M\in BV(\Omega)$ , and up to changing $\bar{x}, x$ with other representatives of the class $\{\bar{x}\}_{u,\Omega}$, it is not restrictive to assume that $\bar{x},x \in \tilde{\Omega}_u\cap \tilde{\Omega}_{b^M}$. Let $(\gamma_i)_{i\in \mathbb{N}}\subset \Gamma_{\Omega}(u)$ with $\gamma_i:I_i \to \Omega$ be such that each curve $\gamma_i$ connects $\bar{x}, x$ for every $i\in \mathbb{N}$, and $\textnormal{SVD}_{u,\Omega}(\bar{x},x)= \inf_{i\in\mathbb{N}}|D^s u_{\gamma_i}|(I_i^\circ)$. Let us observe that thanks to \eqref{eq:step1.1 sufficiency}, and by definition of $\Gamma_\Omega(b^M)$ it is not restrictive to assume that $(\gamma_i)_{i\in \mathbb{N}}\subset \Gamma_{\Omega}(b^M)$. Then, setting $\bar{x}=\gamma_i(a^i_0)$, and $x=\gamma_i(a^i_m)$ where $[a^i_0, a^i_m]= I_i$, for every $i\in\mathbb{N}$, and recalling that each curve $\gamma_i$ can be seen as the continuous union of $m=m(i)\in\mathbb{N}$ affine curves, we get 
\begin{align*}
&|(b^M)\Low(\bar{x})-(b^M)\Low(x)|=|(b^M)\Low(\gamma_i(a_0^i))-(b^M)\Low(\gamma_i(a^i_m))|\stackrel{\eqref{def_u_gamma}}{=}|(b^M)_{\gamma_i}(a_0^i)-(b^M)_{\gamma_i}(a_m^i)|\\&\stackrel{\eqref{eq_|u_gamma (a0)-u_gamma(am)| < eV}}{\leq} |D (b^M)_{\gamma_i}|(I_i^\circ)= \sum_{j=1}^{m}|D (b^M)_{\gamma_i}|((a^i_{j-1},a^i_j)) + \sum_{j=1}^{m-1}|D (b^M)_{\gamma_i}|(a^i_j)\\& \stackrel{\eqref{ugamma = u fettato in (ai-1 , ai)}}{=}\sum_{j=1}^{m}|D(b^M)^{\eta^i_j}_{y^i_j}|((t^i_j,t^i_j+(a^i_j-a^i_{j-1}))) + \sum_{j=1}^{m-1}|D (b^M)_{\gamma_i}|(a^i_j)\\& \stackrel{\textnormal{Proposition}\, \ref{prop_fine properties of ugamma}}{=} \sum_{j=1}^{m}|D(b^M)^{\eta^i_j}_{y^i_j}|((t^i_j,t^i_j+(a^i_j-a^i_{j-1}))) + \sum_{j=1}^{m-1}[b^m](\gamma_i(a^i_j))\\& \stackrel{\eqref{eq:step1.1 sufficiency}}{\leq} \sum_{j=1}^{m}|D^s u^{\eta^i_j}_{y^i_j}|((t^i_j,t^i_j+(a^i_j-a^i_{j-1}))) + \sum_{j=1}^{m-1}[b^m](\gamma_i(a^i_j))\\& \leq \sum_{j=1}^{m}|D^s u^{\eta^i_j}_{y^i_j}|((t^i_j,t^i_j+(a^i_j-a^i_{j-1}))) + \sum_{j=1}^{m-1}[u](\gamma_i(a^i_j))= |D^su_{\gamma_i}|(I_i^\circ),
\end{align*}
where $y^i_j$, $t^i_j$, and $\eta^i_j$ are the quantities associated to $\gamma_i$ as described in Section \ref{subsec_A class of one-dimensional curves}, and for the last inequality sign we used Lemma \ref{lem_bgamma is BV}. Taking the $\inf$ in the right hand side among all indexes $i\in\mathbb{N}$ we get that $|(b^M)\Low(\bar{x})-(b^M)\Low(x)|=0$. By arbitrariness of $x \in \tilde{\Omega}_u\cap \tilde{\Omega}_{b^M}$ we conclude that $b^M$ is $\mathcal{L}^n$-equivalent to a constant function in $\Omega$. By generality of $M>0$ we conclude that $b$ is $\mathcal{L}^n$-equivalent to a constant function in $\Omega$. Thus, by generality of $b$, we deduce that $u$ is a \emph{minimally singular} function.\\
\textbf{Step 2 (Necessity).} 
In this step we will prove the remaining implication, namely that if $u$ is \emph{minimally singular}, then there exists $\bar{x}\in \tilde{\Omega}_u$ such that $\textnormal{SVD}_{u,\Omega}(\bar{x},x)=0$ for $\mathcal{L}^n$-a.e. $ x\in \tilde{\Omega}_u$. We will do it by contradiction, showing that if for every $\bar{x}\in \tilde{\Omega}_u$ the function $\textnormal{SVD}(\bar{x},\cdot)$ is not $\mathcal{L}^n$-equivalent to zero in $\Omega$, then we can construct a function $b\in GBV(\Omega)$ that satisfies conditions i), ii), and iii) of Definition \ref{def_minimally singular functions} with respect to  $u$, and it is not $\mathcal{L}^n$-equivalent to a constant function in $\Omega$. So, thanks to Remark \ref{rem: notion of connectedness}  for every $\bar{x}\in \tilde{\Omega}_u$ the set of those $x\in \tilde{\Omega}_u$ such that the singular vertical distance between $(\bar{x},x)$ is zero, has $\mathcal{L}^n$-measure strictly less than $\mathcal{L}^n(\Omega)$, namely for every $\bar{x}\in \tilde{\Omega}_u$ 
\begin{align}\label{eq: L^n(x) < L^n(Omega)}
\mathcal{L}^n\left(\{ \bar{x} \}_{u,\Omega}\right)<\mathcal{L}^n(\Omega),
\end{align}
where $\{ \bar{x} \}_{u,\Omega}$ was defined in \eqref{def: (x) equivalence class}. So, let us fix any $\bar{x}\in\tilde{\Omega}_u$, and let us consider the function $b:\Omega\to [0,\infty)$ defined as in \eqref{def: the smart baricenter}. Let us prove that $b$ is not $\mathcal{L}^n$-equivalent to a constant in $\Omega$. Let $M>0$ be such that $\mathcal{L}^n(\{|b|\leq M  \})>0$. Let us prove that 
\begin{align}\label{eq_SVD_b kind of equal to SVD_u}
\textnormal{SVD}_{b^M,\Omega}(\bar{x},x)=\textnormal{SVD}_{u,\Omega}(\bar{x},x)\quad \forall\, x\in \{|b|\leq M  \}\cap \tilde{\Omega}_{b^M}\cap \tilde{\Omega}_u.
\end{align}
Indeed, thanks to Lemma \ref{lem_calculation |u_gamma (a0)-u_gamma(am)|} we have that for every $\gamma\in \Gamma_\Omega(b^M)$
\begin{align*}
\textnormal{SVD}_{u,\Omega}(\bar{x},x)=|b(x)-b(\bar{x})|=|b^M(x)-b^M(\bar{x})|\leq |D(b^M)_{\gamma}|(I_\gamma^\circ),
\end{align*}
where we used that both $x$ and $\bar{x}$ belong to $\tilde{\Omega}_{b^M}$ (see for instance Lemma \ref{lem_bgamma is BV}). Since thanks to step 2.2 $Db^M=D^s b^M$, passing to the $\inf$ in the right hand side among all $\gamma\in \Gamma_\Omega(b^M)$ connecting $\bar{x},x$ in the sense of Definition \ref{def_gamma connects two points}, we deduce that $\textnormal{SVD}_{b^M,\Omega}(\bar{x},x)\geq\textnormal{SVD}_{u,\Omega}(\bar{x},x)$. On the other hand, thanks to Lemma \ref{lem_bgamma is BV} we get that 
\begin{align*}
\textnormal{SVD}_{b^M,\Omega}(\bar{x},x)\leq |D^s (b^M)_\gamma|(I_\gamma^\circ)\leq |D^s u_\gamma|(I_\gamma^\circ)\quad   \forall\, \gamma\in \Gamma_u(\Omega).
\end{align*}
So, passing to the $\inf$ on the right hand side among all $\gamma\in \Gamma_u(\Omega)$ we prove that $\textnormal{SVD}_{b^M,\Omega}(\bar{x},x)\leq\textnormal{SVD}_{u,\Omega}(\bar{x},x)$. This proves \eqref{eq_SVD_b kind of equal to SVD_u}. So, let us now suppose by contradiction that there exists $c>0$ such that $b(x)= \textnormal{SVD}_{u,\Omega}(\bar{x},x)= c$ for $\mathcal{L}^n$-a.e. $x\in \Omega$. Thanks to \eqref{eq_SVD_b kind of equal to SVD_u} this implies that for every $M>c$ we have that $\textnormal{SVD}_{b^M,\Omega}(\bar{x},x)= c$ for $\mathcal{L}^n$-a.e. $x\in \tilde{\Omega}_{b^M}\cap \tilde{\Omega}_{u}$. At the same time, since $b(x)=c$ for $\mathcal{L}^n$-a.e.  $x\in \Omega$ it is not difficult to see that $\textnormal{SVD}_{b^M,\Omega}(\bar{x},x)=0$ for $\mathcal{L}^n$-a.e. $x\in \Omega$, which is a contradiction. So the only remaining possibility is that $c=0$, but this would imply that $\textnormal{SVD}_{u,\Omega}(\bar{x},x)=0$ for $\mathcal{L}^n$-a.e. $x\in \Omega$ which is a contradiction by \eqref{eq: L^n(x) < L^n(Omega)}. All in all, thanks to Lemma \ref{lem_SVD is in GBV} and to relation \eqref{eq: L^n(x) < L^n(Omega)} we have constructed a function $b\in GBV(\Omega)$ that is not $\mathcal{L}^n$-equivalent to a constant, and it satisfies conditions i), ii), and iii) of Definition \ref{def_minimally singular functions} with respect to $u$. This implies that $u$ is not \emph{minimally singular}, which is a contradiction with our assumption made at the beginning of step 2. This concludes the proof.
\end{proof}

\section{Application: the rigidity problem for Steiner's perimeter inequality}\label{section_application to Steiner}
\noindent
In this section we apply the theory of \emph{minimally singular} functions developed in the previous sections to  prove Theorem \ref{thm_rigidity by perugini}. We then conclude by commenting on the topological nature of the set $\{ v>0 \}$ for a given $v:\mathbb{R}^n\to [0,\infty)$ that satisfies \eqref{cond_ F[v] has finite perimeter}, and \eqref{cond_ F[v] is indecomposable}, and we prove Lemma \ref{lem_open sets exists for n=1}, Proposition \ref{prop_no open set equivalency}, and Lemma \ref{lem_sufficient condition for the existence of the open set}.

\begin{proof}[Proof of Theorem \ref{thm_rigidity by perugini}]
We divide the proof into two steps.\\
\textbf{Step 1.} Let $\Omega\subset\mathbb{R}^n$ be as in \eqref{CCF condition for Omega}. Let us prove that conditions $i)$ and $ii)$ of Theorem \ref{thm_rigidity by perugini} are indeed equivalent. Let us assume that \emph{rigidity} over $\Omega$ holds true, namely that  for every $E\in \mathcal{M}_\Omega(v)$ there exists $c\in \mathbb{R}$ such that $b_{E,\Omega}(x)=c$ for $\mathcal{L}^n$-a.e. $x\in \Omega$. Thanks to the generality of $E\in\mathcal{M}_\Omega(v)$, by Proposition \ref{prop_equality cases over open sets}, and Definition \ref{def_minimally singular functions}, this  implies that $\frac{1}{2} v$ is \emph{minimally singular}, thus $v$ is \emph{minimally singular}. Thanks to Theorem \ref{thm_characterization minimally singular via SVD} we conclude. Vice versa,  let us assume that $v\in BV(\Omega)$ is \emph{minimally singular} thus $\frac{1}{2} v$ is \emph{minimally singular} too. If $E\in \mathcal{M}_\Omega(v)$ thanks to Proposition \ref{prop_equality cases over open sets} we have that $b_{E,\Omega}\in GBV(\Omega)$ and satisfies  \eqref{eq_nabla b for E = 0 PROP}, \eqref{eq_Salto b< 1/2 salto v for E PROP}, and \eqref{eq_Cantor b< 1/2 Cantor v for E PROP}. Since $\frac{1}{2} v$ is \emph{minimally singular}, then there exists a constant $c\in \mathbb{R}$ such that $b_{E,\Omega}(x)=c$ for $\mathcal{L}^n$-a.e. $x\in \Omega$. Thus \emph{rigidity} over $\Omega$ holds true.\\
\noindent
\textbf{Step 2.} We conclude. Let us show that if there exists $\Omega\subset\mathbb{R}^n$ as in \eqref{CCF condition for Omega} such that $\mathcal{H}^{n-1}(\{ v\Low>0 \}\setminus \Omega)=0$, then $\mathcal{M}(v)=\mathcal{M}_\Omega(v)$. Clearly $\mathcal{M}(v)\subset \mathcal{M}_\Omega(v)$, thus let us focus on the opposite inclusion. First let us notice that for every $E\subset \mathbb{R}^{n+1}$ set of finite perimeter, thanks to De Giorgi structure theorem (see \cite[Theorem 15.9]{maggiBOOK}), we have that
\begin{align}\label{eq_P(E;BxR)=0 se H^(n-1)(B)= 0}
P(E;B\times\mathbb{R})=0\quad \forall\, B\subset \mathbb{R}^n \textnormal{ Borel set such that }\mathcal{H}^{n-1}(B)=0.
\end{align}
Let $E\in \mathcal{M}_\Omega(v)$ then
\begin{align*}
&P(E)=P(E;\{ v\Low>0 \}\times \mathbb{R}) + P(E;\{ v\Low=0 \}\times \mathbb{R})=P(E;(\{ v\Low>0 \}\setminus \Omega)\times \mathbb{R})\\
&+ P(E;\Omega\times\mathbb{R}) + P(E;\{ v\Low=0 \}\times \mathbb{R})\stackrel{\eqref{eq_P(E;BxR)=0 se H^(n-1)(B)= 0}}{=}P(E;\Omega\times\mathbb{R}) + P(E;\{ v\Low=0 \}\times \mathbb{R})\\
&= P(F[v];\Omega\times\mathbb{R}) + P(E;\{ v\Low=0 \}\times \mathbb{R})\stackrel{\eqref{eq_P(E;BxR)=0 se H^(n-1)(B)= 0}}{=}P(F[v];\{ v\Low >0 \}\times\mathbb{R}) + P(E;\{ v\Low=0 \}\times \mathbb{R})\\
&=P(F[v];\{ v\Low >0 \}\times\mathbb{R}) + P(F[v];\{ v\Low=0 \}\times \mathbb{R}) = P(F[v]),
\end{align*}
where for the second last equality sign we used
 \cite[Proposition 3.8]{CagnettiColomboDePhilippisMaggiSteiner}. Thus, putting together the first step with what we have proved so far, we conclude.
\end{proof}
\noindent
The following result shows that \emph{minimally singular} functions  with null approximate distributional gradient, must be equivalent to a constant.
\begin{proposition}\label{prop_characterization of constant functions}
Let $\Omega\subset \mathbb{R}^n$ be as in \eqref{assumptions on Omega}, and let $u\in BV(\Omega)$ be such that $\nabla u=0$ $\mathcal{L}^n$-a.e. in $\Omega$. Then, $u$ is \emph{minimally singular} if and only if $u$ is $\mathcal{L}^n$-equivalent to a constant function in $\Omega$.
\end{proposition}

\begin{proof}[Proof of Proposition \ref{prop_characterization of constant functions}] If $u$ is $\mathcal{L}^n$-equivalent to a constant function then for every $\gamma\in \Gamma_\Omega(u)$ we have that $u_\gamma=c$ $\mathcal{L}^1$-a.e. in $I_\gamma^\circ$, thus $\textnormal{SVD}_{u,\Omega}(\bar{x},x)=0$ for every $x\in \tilde{\Omega}_u$ which, thanks to Theorem \ref{thm_characterization minimally singular via SVD}, implies that $u$ is \emph{minimally singular}. Let us now focus on the remaining implication. Thanks to Theorem \ref{thm_characterization minimally singular via SVD} there exists $\bar{x}\in\tilde{\Omega}_u$ such that $\textnormal{SVD}_{u,\Omega}(\bar{x},x)=0$ for $\mathcal{L}^n$-a.e. $x\in \tilde{\Omega}_u$. Let $x\in \tilde{\Omega}_u$ be such that $\textnormal{SVD}_{u,\Omega}(\bar{x},x)=0$. Let $(\gamma_i)_{i\in \mathbb{N}}\subset \Gamma_\Omega(u)$ with $\gamma_i:I_i \to \Omega$ be such that $\gamma_i$ connects $\bar{x}, x$ for every $i\in \mathbb{N}$, and $\textnormal{SVD}_{u,\Omega}(\bar{x},x)= \inf_{i\in\mathbb{N}}|D^s u_{\gamma_i}|(I_i^\circ)$. Then, setting $\bar{x}=\gamma_i(a^i_0)$, and $x=\gamma_i(a^i_m)$ where $[a^i_0, a^i_m]= I_i$, for every $i\in\mathbb{N}$, we get
\begin{align}\label{proof,lemmaCONST,step1,a}
|u\Low (\bar{x})-u\Low(x)|=|u_{\gamma_i}(a^i_0)-u_{\gamma_i}(a^i_m)|\stackrel{\eqref{eq_|u_gamma (a0)-u_gamma(am)| < eV}}{\leq} |D u_{\gamma_i}|(I_i^\circ)= |D^a u_{\gamma_i}|(I_i^\circ) + |D^s u_{\gamma_i}|(I_i^\circ).
\end{align} 
We know apply Proposition \ref{prop_fine properties of ugamma} to each component of the curve $\gamma_i$. Indeed we recall that by definition of $\gamma_i\in \Gamma_\Omega(u)$ (see Section \ref{subsec_A class of one-dimensional curves}) each curve $\gamma_i$ can be seen as the continuous union of $m(i)\in\mathbb{N}$ affine curves, and thus
\begin{align}\label{proof,lemmaCONST,step1,b}
|D^a u_{\gamma_i}|(I_i^\circ)=\int_{I_i^\circ}|\nabla u_{\gamma_i}(s)|\, ds =  \sum_{j=1}^{m(i)}\int_{a^i_{j-1}}^{a^i_j}|\nabla u(y^i_j+  (s+(t^i_j-a^i_{j-1}))\eta^i_j)\cdot\nu^i_j |\, ds = 0,
\end{align}
where, $y^i_j$, $t^i_j$, and $\eta^i_j$ are the quantities associated to $\gamma_i$ as described in Section \ref{subsec_A class of one-dimensional curves}, and for the last inequality we used the fact that by assumption $\nabla u(x)= 0$ for $\mathcal{L}^n$-a.e. $x\in\Omega$. Thus, combining together \eqref{proof,lemmaCONST,step1,a} with \eqref{proof,lemmaCONST,step1,b}, we get $
|u\Low (\bar{x})-u\Low(x)|\leq |D^s u_{\gamma_i}|(I_i^\circ)$. Taking the $\inf$ among all $i\in\mathbb{N}$ in the right hand side, we get that $|u\Low (\bar{x})-u\Low(x)|=0$. By arbitrariness of $x\in \tilde{\Omega}_u$ we conclude that 
$$
|u\Low (\bar{x})-u\Low(x)|=0\quad \textnormal{for }\mathcal{L}^n\textnormal{-a.e. }x\in \tilde{\Omega}_u.
$$
Since $\Omega=\tilde{\Omega}_u$ up to a set of $\mathcal{L}^n$-measure zero, we conclude.
\end{proof}
\noindent
With the above proposition at hand we can quickly prove the following result.
\begin{corollary}\label{cor_rigidity for constant functions}
Let $v:\mathbb{R}^n\to [0,\infty)$ be a Lebesgue measurable function satisfying \eqref{cond_ F[v] has finite perimeter}, and let $\Omega \subset \mathbb{R}^n$ be as in \eqref{CCF condition for Omega}. Let us further assume that $\nabla v(x) = 0$ for $\mathcal{L}^n$-a.e. $x\in \Omega$. Then the following statements are equivalent:
\begin{itemize}
\item[i)] for every $E \in \mathcal{M}_\Omega(v)$ there exists $t \in \mathbb{R}$ so that $\mathcal{H}^{n+1}\left( \left(E \Delta (t e_{n+1}+ F[v])\right) \cap (\Omega\times \mathbb{R}) \right)=0$;\vspace*{0.1cm}
\item[ii)] there exists $c\in \mathbb{R}$ such that $v(x)=c$ for $\mathcal{L}^n$-a.e. $x\in \Omega$.
\end{itemize}
Moreover, if there exists $\Omega\subset \mathbb{R}^n$ as in \eqref{CCF condition for Omega} such that $\mathcal{H}^{n-1}(\{ v\Low>0 \}\setminus \Omega  )=0$, then statement $i)$ can be substituted with the following one
\begin{itemize}
\item[$\textnormal{i}^*)$] for every $E \in \mathcal{M}(v)$ there exists $t \in \mathbb{R}$ so that $\mathcal{H}^{n+1}\left( E \Delta (t e_{n+1}+ F[v]) \right)=0$.
\end{itemize}
\end{corollary}
\begin{proof}[Proof of Corollary \ref{cor_rigidity for constant functions}]
Combine Proposition \ref{prop_characterization of constant functions} with Theorem \ref{thm_rigidity by perugini}.
\end{proof}
\noindent 
Let us now focus on the topological nature of $\{ v>0 \}$ whenever $v:\mathbb{R}^n\to [0,\infty)$ is a Lebesgue measurable function that satisfies \eqref{cond_ F[v] has finite perimeter}, and \eqref{cond_ F[v] is indecomposable}. Let us begin proving the following result that says that if an open set $\Omega$ is not essentially disconnected by its complement $\Omega^c$, then $\Omega$ is also connected.
\begin{lemma}\label{lem_Omega^c non disconnette Omega}
Let $\Omega\subset \mathbb{R}^n$ be an open set, and let us assume that $\Omega^c$ does not essentially disconnect $\Omega$, namely
\begin{align}\label{eq_Omega thesis}
\mathcal{H}^{n-1} \left( \left( \Omega^{(1)}\cap \partial^e \Omega_+ \cap \partial^e \Omega_- \right)\setminus \Omega^c \right)>0
\end{align}
for every non trivial Borel partition $\{\Omega_+,\Omega_-  \}$ of $\Omega$. Then $\Omega$ is a connected set.
\end{lemma}
\begin{proof}
We argue by contradiction. Suppose that $\Omega$ is not connected, namely there exist $\Omega_+,\Omega_-\subset \Omega$ open sets such that $\Omega=\Omega_+ \cup \Omega_-$ and $\Omega_+\cap \Omega_- = \emptyset$. Observe that $\{\Omega_+,\Omega_-\}$ is a non trivial Borel partition of $\Omega$ according with \eqref{non trivial Borel partition}. Let us prove that
\begin{align}\label{Omega+ e Omega- non rompono i cojoni}
\partial^e \Omega_+ \subset \Omega^c,\quad \partial^e \Omega_- \subset \Omega^c.
\end{align}
Let us start noticing that, since $\Omega_+$ is an open set, $\Omega_+\subset \Omega_+^{(1)}$ and the same is true also for $\Omega_-$, thus
\begin{align}\label{eq_Omega step 1}
\partial^e\Omega_+ \cap \Omega_+ = \emptyset,\quad  \partial^e\Omega_- \cap \Omega_- = \emptyset.
\end{align}
In order to prove \eqref{Omega+ e Omega- non rompono i cojoni} it remains to verify the following conditions 
\begin{align}\label{eq_Omega step 2}
\partial^e \Omega_+ \cap \Omega_-=\emptyset,\quad \partial^e \Omega_- \cap \Omega_+=\emptyset.
\end{align}
Suppose by contradiction that $x\in \partial^e \Omega_+ \cap \Omega_-$, then since $\Omega_-$ is open there exists $\rho>0$ such that $B_\rho(x)\subset \Omega_-$, while by definition of essential boundary (see Section \ref{section_fundamentals}) we can find a decreasing sequence $(\rho_i)_i\subset (0,\rho)$ such that $\mathcal{H}^n(B_{\rho_i}(x)\cap \Omega_+)>0$ for every $i\in\mathbb{N}$ which implies that $B_{\rho_i}(x)\subset \Omega_+\cap \Omega_-$ contradicting the fact that $\Omega_+\cap \Omega_-=\emptyset$. Thus combining \eqref{eq_Omega step 1} with \eqref{eq_Omega step 2} we get \eqref{Omega+ e Omega- non rompono i cojoni}. Thus we found a non trivial Borel partition of $\Omega$ such that 
\begin{align*}
\mathcal{H}^{n-1} \left( \left( \Omega^{(1)}\cap \partial^e \Omega_+ \cap \partial^e \Omega_- \right)\setminus \Omega^c \right)\stackrel{\eqref{Omega+ e Omega- non rompono i cojoni}}{=}0,
\end{align*}
which contradicts \eqref{eq_Omega thesis}. This concludes the proof.
\end{proof}
\noindent
Let us now prove Lemma \ref{lem_open sets exists for n=1}.
\begin{proof}[Proof of Lemma \ref{lem_open sets exists for n=1}]
We divide the proof in few steps.\\
\textbf{Step 1.} In this step we prove that 
\begin{align}\label{eq_n=1 step 1}
x\in \{  v\Low >0 \}\quad \textnormal{ if and only if }\quad \mathcal{L}^1\left((-\infty, x)\cap \{  v>0\}  \right) \mathcal{L}^1\left((x, \infty)\cap \{  v>0\}  \right) > 0.
\end{align}
If $x\in \{ v\Low >0 \}$ then thanks to \eqref{eq: 2.9 Filippo page 15} we have that $x\in \{ v>0 \}^{(1)}$ which implies 
\begin{align*}
\lim_{\rho \to 0^+} \frac{\mathcal{L}^1((x-\rho, x)\cap \{ v>0 \})}{2\rho} = \frac{1}{2}.
\end{align*}
Thus, by definition of limit we get that for every $\epsilon >0$ there exists $\bar{\rho}>0$ such that
\begin{align*}
\mathcal{L}^1((x-\rho, x)\cap \{ v>0 \}) > \rho (1-\epsilon)\quad \forall\, \rho \in (0,\bar{\rho}],
\end{align*}
which for $\epsilon < 1$ guarantees that $\mathcal{L}^1((x-\rho, x)\cap \{ v>0 \})>0$ for every $\rho \in (0,\bar{\rho})$, and thus $\mathcal{L}^1\left((-\infty, x)\cap \{  v>0\}  \right)>0$. Applying the same argument we have just used we can prove that also $\mathcal{L}^1\left((x, \infty)\cap \{  v>0\}  \right) > 0$, and so we proved that 
$$
\mathcal{L}^1\left((-\infty, x)\cap \{  v>0\}  \right) \mathcal{L}^1\left((x, \infty)\cap \{  v>0\}  \right) > 0.
$$
Let us now prove the remaining implication. Let $x\in \mathbb{R}$ be such that 
$$
\mathcal{L}^1\left((-\infty, x)\cap \{  v>0\}  \right) \mathcal{L}^1\left((x, \infty)\cap \{  v>0\}  \right) > 0.
$$
Consider the non trivial Borel partition $\{ G_+,G_- \}$ of $\{ v>0 \}$ (see relation \eqref{non trivial Borel partition}), where
\begin{align*}
G_+:= (x,\infty)\cap \{ v>0 \},\quad G_-:= (-\infty,x ) \cap \{  v>0  \}. 
\end{align*}
Let us show that
\begin{align}\label{eq_n=1 other implication}
\partial^e G_+ \cap \partial^e G_- =\{ x \}. 
\end{align}
First let us note that $\partial^e G_+ \cap \partial^e G_- \neq \emptyset$. Indeed, if that was not the case then
\begin{align*}
0= \mathcal{H}^0 \left( \{ v>0 \}^{(1)}\cap \partial^e G_+ \cap \partial^e G_- \right) \geq \mathcal{H}^0 \left( \left(\{ v>0 \}^{(1)}\cap \partial^e G_+ \cap \partial^e G_- \right) \setminus \{ v\Low =0 \} \right)\stackrel{\eqref{cond_ F[v] is indecomposable}  }{>} 0,
\end{align*}
which is clearly a contradiction with \eqref{cond_ F[v] is indecomposable}. Thus let $y\in \partial^e G_+ \cap \partial^e G_-$ and suppose by contradiction that $y \neq x$. With no loss of generality we consider $y>x$. Then, since $\textnormal{dist}(y, G_-)= d > 0$ we get that
\begin{align*}
\lim_{\rho \to 0^+} \frac{\mathcal{L}^1( (y-\rho, x+ \rho)\cap G_- )  }{2\rho}=0,
\end{align*}
and thus $y\notin \partial^e G_-$, which is a contradiction with $y \in \partial^e G_+ \cap \partial^e G_-$. This proves \eqref{eq_n=1 other implication}. Finally, suppose by contradiction that $x\in \{ v\Low = 0 \}$ then we would have that 
$$
\mathcal{H}^0 \left( \left(\{ v>0 \}^{(1)}\cap \partial^e G_+ \cap \partial^e G_- \right) \setminus \{ v\Low =0 \} \right)=0,
$$
which again contradicts the validity of \eqref{cond_ F[v] is indecomposable}. So $x\in \{ v\Low >0 \}$ and this concludes the first step. \\
\textbf{Step 2.} Let $x,y \in \{ v\Low >0 \}$ with $x < y$, then $z \in \{ v\Low >0 \}$ for every $z \in (x,y)$. Indeed, let $z \in (x, y)$ then thanks to step 1 we have that
\begin{align*}
&\mathcal{L}^1 \left( (-\infty, z) \cap \{ v>0 \} \right) \geq \mathcal{L}^1 \left( (-\infty, x) \cap \{ v>0 \} \right) >0,\\
&\mathcal{L}^1 \left( (z, \infty) \cap \{ v>0 \} \right) \geq \mathcal{L}^1 \left( (y, \infty) \cap \{ v>0 \} \right) >0,
\end{align*}
thus 
\begin{align*}
\mathcal{L}^1 \left( (-\infty, z) \cap \{ v>0 \} \right) \mathcal{L}^1 \left( (z, \infty) \cap \{ v>0 \} \right) >0
\end{align*}
and so by step 1 we conclude that $z \in \{ v\Low > 0 \}$. This concludes the second step.\\
\textbf{Step 3.} Let us conclude. Let's prove that $\{ v\Low >0 \}$ coincides with its interior,
namely that for every $x \in \{ v\Low >0 \}$ there exists $d > 0$ such that
\begin{align}\label{eq_step 3 n=1}
(x-d, x+d) \subset \{ v\Low > 0  \}.
\end{align}
Let $x \in \{ v\Low > 0 \}$ then by step 1 we have that  $(-\infty, x)\cap \{ v\Low > 0 \}\neq \emptyset$, and $(x, \infty)\cap \{ v\Low > 0 \}\neq \emptyset$. So let $y_- \in (-\infty, x)\cap \{ v\Low > 0 \}$ and $y_+ \in (x, \infty)\cap \{ v\Low > 0 \}$. Setting $d:= \min \{|x- y_-|, |x- y_+| \}$, thanks to step 2 we have that 
$$
(x-d,x+d) \subset \{ v\Low > 0 \}
$$
which proves \eqref{eq_step 3 n=1}. Thus we have that $\{ v\Low > 0 \}$ coincides with its interior, and so $\{ v\Low > 0\}$ is an open set. Thanks to step 2 or using Lemma \ref{lem_Omega^c non disconnette Omega} we conclude that $\{ v\Low > 0 \}$ is an open and connected set. Since $\mathcal{L}^1 (\{ v> 0  \})< \infty$ we conclude that there exists an open and bounded interval $(a,b) \subset \mathbb{R}$ such that $\{ v\Low >0  \} = (a,b)$.
\end{proof}

\noindent
Before proving Proposition \ref{prop_no open set equivalency}, we need an intermediate result. Let us first recall the definition of an indecomposable set of finite perimeter (see for instance \cite{AmbrosioCasellesMasnau}). Let $m\in\mathbb{N}$, with $m\geq 1$, and let $G\subset\mathbb{R}^m$ be a set of finite perimeter. We say that $G$ is \emph{indecomposable} if for every non trivial partition of $G$ of sets of finite perimeter $\{ G_+,G_- \}$ as in \eqref{non trivial Borel partition}, we have that $P(G)< P(G_+)+P(G_-)$.

\begin{lemma}\label{lem_C indecomponibile, allora 1_c fa quello che deve.} Let $E\subset \mathbb{R}^n$ be an indecomposable set of finite perimeter and finite volume. Then, $\{ \chi_E\Low = 0 \}$ does not essentially disconnect $E$.
\end{lemma}
\begin{proof}
Since $E$ is an indecomposable set of finite perimeter and finite volume, thanks to \cite[Remark 2.3]{ccdpmGAUSS} we have that $E$ is essentially connected, namely for every non trivial Borel partition $\{ G_+, G_- \}$ of $E$ (see definition \eqref{def_does not essentially disconnect}) we have
\begin{align}\label{quello da cui partiamo}
\mathcal{H}^{n-1}(E^{(1)}\cap \partial^e G_+ \cap \partial^e G_-)>0.
\end{align}
Let us show that 
\begin{align}\label{sotto e densita 1 so ugual1}
\{ \chi_E\Low >0 \}=E^{(1)}.
\end{align}
Thanks to relation \eqref{eq: 2.9 Filippo page 15} we are only left to prove that $E^{(1)}\subset \{ \chi_E\Low >0 \}$. Let $x\in E^{(1)}$, and observe that $E= \{ \chi_E > t \}$ for every $0\leq t <1$. Then, 
\begin{align*}
\lim_{\rho \to 0^+}\frac{\mathcal{L}^n(B_\rho(x)\cap \{ \chi_E > t \} )}{\mathcal{L}^n(B_\rho(x))}=\lim_{\rho \to 0^+}\frac{\mathcal{L}^n(B_\rho(x)\cap E)}{\mathcal{L}^n(B_\rho(x))},\quad \forall\, 0\leq t <1,
\end{align*}
which by definition of $\chi_E\Low$ implies that $\chi_E\Low (x)=1$, and this proves \eqref{sotto e densita 1 so ugual1}. Let $\{G_+,G_-  \}$ be a non trivial Borel partition for $E$, then
\begin{align*}
&\mathcal{H}^{n-1}\left((E^{(1)}\cap \partial^e G_+ \cap \partial^e G_-)\setminus \{ \chi_E\Low = 0 \}\right) \stackrel{\eqref{sotto e densita 1 so ugual1}}{=}\mathcal{H}^{n-1}\left( (\{ \chi_E\Low >0 \}\cap \partial^e G_+ \cap \partial^e G_-)\setminus \{ \chi_E\Low = 0 \}\right)\\
&=\mathcal{H}^{n-1}\left( \{ \chi_E\Low >0 \}\cap \partial^e G_+ \cap \partial^e G_-\right)\stackrel{\eqref{sotto e densita 1 so ugual1}}{=}\mathcal{H}^{n-1}\left( E^{(1)}\cap \partial^e G_+ \cap \partial^e G_-\right)\stackrel{\eqref{quello da cui partiamo}}{>}0,
\end{align*}
and this concludes the proof.
\end{proof}
\begin{remark}
Note that thanks to Lemma \ref{lem_C indecomponibile, allora 1_c fa quello che deve.} and \cite[Theorem 4.3]{CagnettiColomboDePhilippisMaggiSteiner} we have that $E\subset\mathbb{R}^n$ is an indecomposable set of finite perimeter and finite volume if and only if $F[\chi_E]\subset\mathbb{R}^{n+1}$ is an indecomposable set of finite perimeter and finite volume.
\end{remark}
\noindent
Now we prove Proposition \ref{prop_no open set equivalency}. The proof of this result is deeply inspired by \cite[Example 12.25]{maggiBOOK}.

\begin{proof}[Proof of Proposition \ref{prop_no open set equivalency}]
We divide the proof in several steps.\\
\textbf{Step 1.} In this step we construct a set $K\subset\mathbb{R}^n$ of finite perimeter and finite volume with empty interior. Let $(x_h)_{h\in\mathbb{N}}\subset B_1(0)$ be a countable and dense sequence of points, and let $(r_h)_{h\in \mathbb{N}}\subset (0,\epsilon)$, with $\epsilon < n\omega_n = \mathcal{L}^n(B_1(0))$, be such that $n\omega_n \sum_{h\in\mathbb{N}} r_h^{n-1}\leq 1$.  As shown in \cite[Example 12.25]{maggiBOOK}, we have that the set $E:= \cup_{h\in\mathbb{N}} B_{r_h}(x_h)\subset B_1(0)$ is an open set such that $\mathcal{L}^n(E) \leq \epsilon$, and $P(E)<  1$. Let us consider the set $K\subset \mathbb{R}^n$ defined as
$$
K:= \overline{B_1(0)}\cap E^c.
$$
Then by construction, we have that $K$ is a close set of finite volume, while thanks to \cite[Exercise 12.9, and Theorem 16.3]{maggiBOOK} we have that $P(K)<\infty$. Finally let us note that, since $\bar{E}= \overline{B_1(0)}$, we have that $K=\partial E$, and so $K$ has empty interior. This concludes the first step.\\
\textbf{Step 2.} We need to work with an indecomposable component of $K$. Since the set $K$ may be a decomposable set, we apply the \emph{Decomposition Theorem} for sets of finite perimeter (see \cite[Theorem 1]{AmbrosioCasellesMasnau}) to find an indecomposable set $C\subset \mathbb{R}^n$ of finite perimeter with the property that $\mathcal{L}^n(C\setminus K)=0$. Let us note that thanks to standard BV-theory (see for instance Section \ref{section_fundamentals}) $\chi_C\in BV(\mathbb{R}^n)$ with $\mathcal{L}^n(C)<\infty$, while thanks to Lemma \ref{lem_C indecomponibile, allora 1_c fa quello che deve.} we have that $\{\chi_C\Low = 0  \}$ does not essentially disconnect $C$.\\
\textbf{Step 3.} Let $v:\mathbb{R}^n\to [0,\infty)$ be the Lebesgue measurable function defined as $v(x)=\chi_C(x)$ for every $x\in\mathbb{R}^n$. Let us show that \emph{rigidity} for $v$ holds true. First, let us note that thanks to step 2 we have that $v$ satisfies both \eqref{cond_ F[v] has finite perimeter}, and \eqref{cond_ F[v] is indecomposable}. Second, let us show that 
\begin{align*}
|D^s v|(\{ v\Low >0 \})= 0.
\end{align*}
Indeed, thanks to the Coarea formula we get
\begin{align*}
0= \int_{\mathbb{R}}\mathcal{H}^{n-1}(\{ v\Low >0 \}\cap \partial^e\{ v>t \}) dt = |D v|(\{ v\Low >0 \}) \geq |D^s v|(\{ v\Low >0 \}),
\end{align*}
where for the first equality sign we used the fact that since $\{ v\Low >0  \}\subset C^{(1)}$, and since $\{ v> t \}= C$ for every $0 \leq t < 1$, while $\{ v> t \}=\mathbb{R}^n$ for $t<0$, and $\{ v> t \}=\emptyset$ for $t > 1$, we have that $\{ v\Low >0 \}\cap \partial^e\{ v>t \}=\emptyset$ for $\mathcal{L}^1$-a.e. $t\in \mathbb{R}$. Thus, applying \cite[Theorem 1.16]{CagnettiColomboDePhilippisMaggiSteiner} we have that \emph{rigidity} for $v$ holds true. This concludes the step.\\
\textbf{Step 4.} We are left to prove that for every $\Omega\subset \mathbb{R}^n$ open set we have $\mathcal{L}^n(C\Delta\Omega)>0$. Let us show that for every $x\in C^{(1)}$ we have that 
\begin{align}\label{intermediate step 4 the real one}
0<\mathcal{L}^n(B_\rho(x) \cap C) < \mathcal{L}^n(B_\rho(x))\quad \forall\,\rho >0.
\end{align}
Since $x\in C^{(1)}$ we have that $\mathcal{L}^n(B_\rho(x) \cap C)>0$ for every $\rho > 0$. In order to prove the remaining inequality, let us observe that if $x\in K$ then, by construction of the set $E$, and since $K=\partial E$, we have that $B_\rho(x)\cap E\neq \emptyset$, and so $\mathcal{L}^n(B_\rho(x) \cap K) < \mathcal{L}^n(B_\rho(x))$ for every $\rho >0$. Let $x\in C^{(1)}$, and let $\rho>0$. Since $\mathcal{L}^n(B_\rho(x) \cap C)>0$, also $\mathcal{L}^n(B_\rho(x) \cap K)>0$, so we can find a point $z\in B_\rho(x) \cap K$, and a radius $r>0$ such that $B_r(z)\subset B_\rho(x)$. Thus
\begin{align*}
&\mathcal{L}^n(B_\rho(x) \cap C)\leq\mathcal{L}^n(B_\rho(x) \cap K)= \mathcal{L}^n(B_r(z) \cap K) + \mathcal{L}^n((B_\rho(x)\setminus B_r(z)) \cap K)\\&< \mathcal{L}^n(B_r(z)) + \mathcal{L}^n(B_\rho(x)\setminus B_r(z)) = \mathcal{L}^n(B_\rho(x)),
\end{align*}
which proves \eqref{intermediate step 4 the real one}. Let $\Omega\subset \mathbb{R}^n$ be an open set. Since $\{v>0  \}= C^{(1)}$ up to a set of $\mathcal{L}^n$-measure zero, thanks to \eqref{intermediate step 4 the real one} we conclude that $\mathcal{L}^n(\{ v>0 \}\Delta
 \Omega)>0$. Observe that this implies that $\mathcal{H}^{n-1}(\{ v\Low >0 \}\Delta \Omega)= \infty$.
\end{proof}
\noindent
Lastly, we conclude proving Lemma \ref{lem_sufficient condition for the existence of the open set}.
\begin{proof}[Proof of Lemma \ref{lem_sufficient condition for the existence of the open set}] Let us consider the set $\Omega := \mathbb{R}^n\setminus \overline{\{v\Low = 0  \}}$. Thanks to \eqref{cond_la chiusa di v = 0 meno v = 0 minuscola} we have that $\mathcal{H}^{n-1}(\overline{\{ v\Low=0 \}}\setminus \{ v\Low = 0 \})=0$. Thus, thanks to \cite[Remark 1.5]{CagnettiColomboDePhilippisMaggiSteiner} we get that $\Omega^c=\overline{\{ v\Low=0 \}}$ does not essentially disconnect $\{ v\Low>0 \}$. Moreover, by construction we have that 
\begin{align*}
&\mathcal{H}^{n-1}(\{ v\Low > 0 \} \Delta \Omega)=\mathcal{H}^{n-1}(\{ v\Low>0 \}\setminus \Omega)=\mathcal{H}^{n-1}(\{ v\Low > 0 \}\cap \overline{\{ v\Low = 0 \}})\\
&=\mathcal{H}^{n-1}(\{ v\Low > 0 \}\cap (\overline{\{ v\Low = 0 \}}\setminus \{ v\Low = 0 \})) \leq \mathcal{H}^{n-1}(\overline{\{ v\Low = 0 \}}\setminus \{ v\Low = 0 \})=0.
\end{align*}
Thus we have that $\Omega^{(1)}= \{ v\Low>0 \}^{(1)}$ and so by definition of essential connectedness $\Omega^c$ does not essentially disconnect $\Omega$. Finally since $\Omega^c$ does not essentially disconnect $\Omega$, and since $\Omega$ is an open set,  thanks to Lemma \ref{lem_Omega^c non disconnette Omega} we get that $\Omega$ is also connected. This concludes the proof.
\end{proof}

\def\cprime{$'$}

\end{document}

%% file: OmegaconnessoUP_nobarycenter.pdf_tex
\begingroup%
  \makeatletter%
  \providecommand\color[2][]{%
    \errmessage{(Inkscape) Color is used for the text in Inkscape, but the package 'color.sty' is not loaded}%
    \renewcommand\color[2][]{}%
  }%
  \providecommand\transparent[1]{%
    \errmessage{(Inkscape) Transparency is used (non-zero) for the text in Inkscape, but the package 'transparent.sty' is not loaded}%
    \renewcommand\transparent[1]{}%
  }%
  \providecommand\rotatebox[2]{#2}%
  \newcommand*\fsize{\dimexpr\f@size pt\relax}%
  \newcommand*\lineheight[1]{\fontsize{\fsize}{#1\fsize}\selectfont}%
  \ifx\svgwidth\undefined%
    \setlength{\unitlength}{689.73234211bp}%
    \ifx\svgscale\undefined%
      \relax%
    \else%
      \setlength{\unitlength}{\unitlength * \real{\svgscale}}%
    \fi%
  \else%
    \setlength{\unitlength}{\svgwidth}%
  \fi%
  \global\let\svgwidth\undefined%
  \global\let\svgscale\undefined%
  \makeatother%
  \begin{picture}(1,0.52107905)%
    \lineheight{1}%
    \setlength\tabcolsep{0pt}%
    \put(0,0){\includegraphics[width=\unitlength,page=1]{OmegaconnessoUP_nobarycenter.pdf}}%
    \put(0.0971261,0.34039725){\color[rgb]{0,0,0}\makebox(0,0)[lt]{\lineheight{1.25}\smash{\begin{tabular}[t]{l}\tiny{$F[v]$}\end{tabular}}}}%
    \put(0.61811409,0.34650764){\color[rgb]{0,0,0}\makebox(0,0)[lt]{\lineheight{1.25}\smash{\begin{tabular}[t]{l}\tiny{$E$}\end{tabular}}}}%
    \put(0.29027286,0.1166792){\color[rgb]{1,0,0}\makebox(0,0)[lt]{\lineheight{1.25}\smash{\begin{tabular}[t]{l}\tiny{$\{  v\Low = 0 \}$}\end{tabular}}}}%
  \end{picture}%
\endgroup%

%% file: Vertical_nobarycenter.pdf_tex
\begingroup%
  \makeatletter%
  \providecommand\color[2][]{%
    \errmessage{(Inkscape) Color is used for the text in Inkscape, but the package 'color.sty' is not loaded}%
    \renewcommand\color[2][]{}%
  }%
  \providecommand\transparent[1]{%
    \errmessage{(Inkscape) Transparency is used (non-zero) for the text in Inkscape, but the package 'transparent.sty' is not loaded}%
    \renewcommand\transparent[1]{}%
  }%
  \providecommand\rotatebox[2]{#2}%
  \newcommand*\fsize{\dimexpr\f@size pt\relax}%
  \newcommand*\lineheight[1]{\fontsize{\fsize}{#1\fsize}\selectfont}%
  \ifx\svgwidth\undefined%
    \setlength{\unitlength}{846.28048532bp}%
    \ifx\svgscale\undefined%
      \relax%
    \else%
      \setlength{\unitlength}{\unitlength * \real{\svgscale}}%
    \fi%
  \else%
    \setlength{\unitlength}{\svgwidth}%
  \fi%
  \global\let\svgwidth\undefined%
  \global\let\svgscale\undefined%
  \makeatother%
  \begin{picture}(1,0.57813243)%
    \lineheight{1}%
    \setlength\tabcolsep{0pt}%
    \put(0,0){\includegraphics[width=\unitlength,page=1]{Vertical_nobarycenter.pdf}}%
    \put(0.07042,0.48319013){\color[rgb]{0,0,0}\makebox(0,0)[lt]{\lineheight{1.25}\smash{\begin{tabular}[t]{l}\tiny{$F[v]$}\end{tabular}}}}%
    \put(0.61840952,0.46211361){\color[rgb]{0,0,0}\makebox(0,0)[lt]{\lineheight{1.25}\smash{\begin{tabular}[t]{l}\tiny{$E$}\end{tabular}}}}%
  \end{picture}%
\endgroup%

%% file: Notminimallysingular.pdf_tex
\begingroup%
  \makeatletter%
  \providecommand\color[2][]{%
    \errmessage{(Inkscape) Color is used for the text in Inkscape, but the package 'color.sty' is not loaded}%
    \renewcommand\color[2][]{}%
  }%
  \providecommand\transparent[1]{%
    \errmessage{(Inkscape) Transparency is used (non-zero) for the text in Inkscape, but the package 'transparent.sty' is not loaded}%
    \renewcommand\transparent[1]{}%
  }%
  \providecommand\rotatebox[2]{#2}%
  \newcommand*\fsize{\dimexpr\f@size pt\relax}%
  \newcommand*\lineheight[1]{\fontsize{\fsize}{#1\fsize}\selectfont}%
  \ifx\svgwidth\undefined%
    \setlength{\unitlength}{983.56953685bp}%
    \ifx\svgscale\undefined%
      \relax%
    \else%
      \setlength{\unitlength}{\unitlength * \real{\svgscale}}%
    \fi%
  \else%
    \setlength{\unitlength}{\svgwidth}%
  \fi%
  \global\let\svgwidth\undefined%
  \global\let\svgscale\undefined%
  \makeatother%
  \begin{picture}(1,0.40647659)%
    \lineheight{1}%
    \setlength\tabcolsep{0pt}%
    \put(0,0){\includegraphics[width=\unitlength,page=1]{Notminimallysingular.pdf}}%
    \put(0.38381349,0.11564081){\color[rgb]{0,0,1}\makebox(0,0)[lt]{\lineheight{0.40000001}\smash{\begin{tabular}[t]{l}\tiny{$\tiny{S_u}$}\end{tabular}}}}%
    \put(0.09885322,0.05131895){\color[rgb]{0,0,0}\makebox(0,0)[lt]{\lineheight{1.25}\smash{\begin{tabular}[t]{l}\tiny{$\Omega$}\end{tabular}}}}%
    \put(0.61748786,0.05067479){\color[rgb]{0,0,0}\makebox(0,0)[lt]{\lineheight{1.25}\smash{\begin{tabular}[t]{l}\tiny{$\Omega$}\end{tabular}}}}%
    \put(0.90393564,0.11664509){\color[rgb]{0,0,1}\makebox(0,0)[lt]{\lineheight{0.40000001}\smash{\begin{tabular}[t]{l}\tiny{$S_b$}\end{tabular}}}}%
    \put(0.42924398,0.21884312){\color[rgb]{0,0,0}\makebox(0,0)[lt]{\lineheight{1.25}\smash{\begin{tabular}[t]{l}\tiny{$d$}\end{tabular}}}}%
    \put(0.12341969,0.00304359){\color[rgb]{0,0,0}\makebox(0,0)[lt]{\lineheight{1.25}\smash{\begin{tabular}[t]{l}\tiny{$x_1$}\end{tabular}}}}%
    \put(0.02301941,0.08580301){\color[rgb]{0,0,0}\makebox(0,0)[lt]{\lineheight{1.25}\smash{\begin{tabular}[t]{l}\tiny{$x_2$}\end{tabular}}}}%
    \put(-0.00030639,0.13814642){\color[rgb]{0,0,0}\makebox(0,0)[lt]{\lineheight{1.25}\smash{\begin{tabular}[t]{l}\tiny{$t$}\end{tabular}}}}%
    \put(0.63652367,0.00304359){\color[rgb]{0,0,0}\makebox(0,0)[lt]{\lineheight{1.25}\smash{\begin{tabular}[t]{l}\tiny{$x_1$}\end{tabular}}}}%
    \put(0.53612339,0.08580301){\color[rgb]{0,0,0}\makebox(0,0)[lt]{\lineheight{1.25}\smash{\begin{tabular}[t]{l}\tiny{$x_2$}\end{tabular}}}}%
    \put(0.51432258,0.13814642){\color[rgb]{0,0,0}\makebox(0,0)[lt]{\lineheight{1.25}\smash{\begin{tabular}[t]{l}\tiny{$t$}\end{tabular}}}}%
    \put(0.79617382,0.21873463){\color[rgb]{0,0,0}\makebox(0,0)[lt]{\lineheight{1.25}\smash{\begin{tabular}[t]{l}\tiny{$d$}\end{tabular}}}}%
  \end{picture}%
\endgroup%

%% file: Minimallysingular.pdf_tex
\begingroup%
  \makeatletter%
  \providecommand\color[2][]{%
    \errmessage{(Inkscape) Color is used for the text in Inkscape, but the package 'color.sty' is not loaded}%
    \renewcommand\color[2][]{}%
  }%
  \providecommand\transparent[1]{%
    \errmessage{(Inkscape) Transparency is used (non-zero) for the text in Inkscape, but the package 'transparent.sty' is not loaded}%
    \renewcommand\transparent[1]{}%
  }%
  \providecommand\rotatebox[2]{#2}%
  \newcommand*\fsize{\dimexpr\f@size pt\relax}%
  \newcommand*\lineheight[1]{\fontsize{\fsize}{#1\fsize}\selectfont}%
  \ifx\svgwidth\undefined%
    \setlength{\unitlength}{520.82519766bp}%
    \ifx\svgscale\undefined%
      \relax%
    \else%
      \setlength{\unitlength}{\unitlength * \real{\svgscale}}%
    \fi%
  \else%
    \setlength{\unitlength}{\svgwidth}%
  \fi%
  \global\let\svgwidth\undefined%
  \global\let\svgscale\undefined%
  \makeatother%
  \begin{picture}(1,0.81990025)%
    \lineheight{1}%
    \setlength\tabcolsep{0pt}%
    \put(0,0){\includegraphics[width=\unitlength,page=1]{Minimallysingular.pdf}}%
    \put(0.61796435,0.08303638){\color[rgb]{0,0,0}\makebox(0,0)[lt]{\lineheight{1.25}\smash{\begin{tabular}[t]{l}\tiny{$\Omega$}\end{tabular}}}}%
    \put(0.4920822,0.16113796){\color[rgb]{0,0,0}\makebox(0,0)[lt]{\lineheight{1.25}\smash{\begin{tabular}[t]{l}\tiny{$\gamma$}\end{tabular}}}}%
    \put(0.71860425,0.20663681){\color[rgb]{0,0,1}\makebox(0,0)[lt]{\lineheight{1.25}\smash{\begin{tabular}[t]{l}\tiny{$S_u$}\end{tabular}}}}%
    \put(0.20357554,0.1077089){\color[rgb]{0,0,0}\makebox(0,0)[lt]{\lineheight{1.25}\smash{\begin{tabular}[t]{l}\tiny{$\bar{x}$}\end{tabular}}}}%
    \put(0.38407237,0.29404428){\color[rgb]{0,0,0}\makebox(0,0)[lt]{\lineheight{1.25}\smash{\begin{tabular}[t]{l}\tiny{$x$}\end{tabular}}}}%
    \put(0.2503527,0.00574776){\color[rgb]{0,0,0}\makebox(0,0)[lt]{\lineheight{1.25}\smash{\begin{tabular}[t]{l}\tiny{$x_1$}\end{tabular}}}}%
    \put(0.06075143,0.15915753){\color[rgb]{0,0,0}\makebox(0,0)[lt]{\lineheight{1.25}\smash{\begin{tabular}[t]{l}\tiny{$x_2$}\end{tabular}}}}%
    \put(-0.0005786,0.26664714){\color[rgb]{0,0,0}\makebox(0,0)[lt]{\lineheight{1.25}\smash{\begin{tabular}[t]{l}\tiny{$t$}\end{tabular}}}}%
    \put(0.62828326,0.61861909){\color[rgb]{0,0.50196078,0.50196078}\makebox(0,0)[lt]{\lineheight{1.25}\smash{\begin{tabular}[t]{l}\tiny{$(\gamma, u_\gamma)$}\end{tabular}}}}%
  \end{picture}%
\endgroup%

%% file: Curve.pdf_tex
\begingroup%
  \makeatletter%
  \providecommand\color[2][]{%
    \errmessage{(Inkscape) Color is used for the text in Inkscape, but the package 'color.sty' is not loaded}%
    \renewcommand\color[2][]{}%
  }%
  \providecommand\transparent[1]{%
    \errmessage{(Inkscape) Transparency is used (non-zero) for the text in Inkscape, but the package 'transparent.sty' is not loaded}%
    \renewcommand\transparent[1]{}%
  }%
  \providecommand\rotatebox[2]{#2}%
  \newcommand*\fsize{\dimexpr\f@size pt\relax}%
  \newcommand*\lineheight[1]{\fontsize{\fsize}{#1\fsize}\selectfont}%
  \ifx\svgwidth\undefined%
    \setlength{\unitlength}{727.7434019bp}%
    \ifx\svgscale\undefined%
      \relax%
    \else%
      \setlength{\unitlength}{\unitlength * \real{\svgscale}}%
    \fi%
  \else%
    \setlength{\unitlength}{\svgwidth}%
  \fi%
  \global\let\svgwidth\undefined%
  \global\let\svgscale\undefined%
  \makeatother%
  \begin{picture}(1,1.01418172)%
    \lineheight{1}%
    \setlength\tabcolsep{0pt}%
    \put(0,0){\includegraphics[width=\unitlength,page=1]{Curve.pdf}}%
    \put(0.71453536,0.06189558){\color[rgb]{0,0,0}\makebox(0,0)[lt]{\lineheight{1.25}\smash{\begin{tabular}[t]{l}$\Omega_{\eta_2}$\end{tabular}}}}%
    \put(0.34288229,0.33929915){\color[rgb]{0,0,0}\makebox(0,0)[lt]{\lineheight{1.25}\smash{\begin{tabular}[t]{l}$\eta_2$\end{tabular}}}}%
    \put(0.88732753,0.05131214){\color[rgb]{0,0,0}\makebox(0,0)[lt]{\lineheight{1.25}\smash{\begin{tabular}[t]{l}$H_{\eta_2}$\end{tabular}}}}%
    \put(0.48537124,0.15012994){\color[rgb]{0,0,0}\makebox(0,0)[lt]{\lineheight{1.25}\smash{\begin{tabular}[t]{l}$y_2$\end{tabular}}}}%
    \put(0.24077643,0.8599598){\color[rgb]{0,0,0}\makebox(0,0)[lt]{\lineheight{1.25}\smash{\begin{tabular}[t]{l}$\Omega_{y_2}^{\eta_2}$\end{tabular}}}}%
    \put(0.47434196,0.39567107){\color[rgb]{0,0,0}\makebox(0,0)[lt]{\lineheight{1.25}\smash{\begin{tabular}[t]{l}$\bar{x}$\end{tabular}}}}%
    \put(0.74272132,0.70326565){\color[rgb]{0,0,0}\makebox(0,0)[lt]{\lineheight{1.25}\smash{\begin{tabular}[t]{l}$x$\end{tabular}}}}%
    \put(0.60179149,0.44101371){\color[rgb]{0,0,0}\makebox(0,0)[lt]{\lineheight{1.25}\smash{\begin{tabular}[t]{l}$\gamma(a_1)$\end{tabular}}}}%
    \put(0.64958507,0.57091422){\color[rgb]{0,0,0}\makebox(0,0)[lt]{\lineheight{1.25}\smash{\begin{tabular}[t]{l}$\gamma(a_2)$\end{tabular}}}}%
    \put(0.1900314,0.59174729){\color[rgb]{0,0,0}\makebox(0,0)[lt]{\lineheight{1.25}\smash{\begin{tabular}[t]{l}$t_2$\end{tabular}}}}%
    \put(0.82159381,0.40237169){\color[rgb]{0,0,0}\makebox(0,0)[lt]{\lineheight{1.25}\smash{\begin{tabular}[t]{l}$\Omega$\end{tabular}}}}%
    \put(0.27743001,0.21296306){\color[rgb]{0,0,0}\makebox(0,0)[lt]{\lineheight{1.25}\smash{\begin{tabular}[t]{l}$O$\end{tabular}}}}%
  \end{picture}%
\endgroup%

%% file: Admissible1.pdf_tex
\begingroup%
  \makeatletter%
  \providecommand\color[2][]{%
    \errmessage{(Inkscape) Color is used for the text in Inkscape, but the package 'color.sty' is not loaded}%
    \renewcommand\color[2][]{}%
  }%
  \providecommand\transparent[1]{%
    \errmessage{(Inkscape) Transparency is used (non-zero) for the text in Inkscape, but the package 'transparent.sty' is not loaded}%
    \renewcommand\transparent[1]{}%
  }%
  \providecommand\rotatebox[2]{#2}%
  \newcommand*\fsize{\dimexpr\f@size pt\relax}%
  \newcommand*\lineheight[1]{\fontsize{\fsize}{#1\fsize}\selectfont}%
  \ifx\svgwidth\undefined%
    \setlength{\unitlength}{1077.81252334bp}%
    \ifx\svgscale\undefined%
      \relax%
    \else%
      \setlength{\unitlength}{\unitlength * \real{\svgscale}}%
    \fi%
  \else%
    \setlength{\unitlength}{\svgwidth}%
  \fi%
  \global\let\svgwidth\undefined%
  \global\let\svgscale\undefined%
  \makeatother%
  \begin{picture}(1,0.39959402)%
    \lineheight{1}%
    \setlength\tabcolsep{0pt}%
    \put(0,0){\includegraphics[width=\unitlength,page=1]{Admissible1.pdf}}%
    \put(0.34529596,0.3595426){\color[rgb]{0,0,0}\makebox(0,0)[lt]{\lineheight{1.25}\smash{\begin{tabular}[t]{l}\tiny{$B$}\end{tabular}}}}%
    \put(0.14505384,0.20481009){\color[rgb]{0,0,0}\makebox(0,0)[lt]{\lineheight{1.25}\smash{\begin{tabular}[t]{l}\tiny{$B^1$}\end{tabular}}}}%
    \put(0.25427679,0.22880604){\color[rgb]{0,0,0}\makebox(0,0)[lt]{\lineheight{1.25}\smash{\begin{tabular}[t]{l}\tiny{$B^2$}\end{tabular}}}}%
    \put(0.18958847,0.18115103){\color[rgb]{0,0,0}\makebox(0,0)[lt]{\lineheight{1.25}\smash{\begin{tabular}[t]{l}\tiny{$O$}\end{tabular}}}}%
    \put(0.93212459,0.35449317){\color[rgb]{0,0,0}\makebox(0,0)[lt]{\lineheight{1.25}\smash{\begin{tabular}[t]{l}\tiny{$B$}\end{tabular}}}}%
    \put(0.67666649,0.20743404){\color[rgb]{0,0,0}\makebox(0,0)[lt]{\lineheight{1.25}\smash{\begin{tabular}[t]{l}\tiny{$\tilde{V}^1=V^1$}\end{tabular}}}}%
    \put(0.72278048,0.14763305){\color[rgb]{0,0,0}\makebox(0,0)[lt]{\lineheight{1.25}\smash{\begin{tabular}[t]{l}\tiny{$\tilde{V}^2$}\end{tabular}}}}%
    \put(0.78412987,0.18203049){\color[rgb]{0,0,0}\makebox(0,0)[lt]{\lineheight{1.25}\smash{\begin{tabular}[t]{l}\tiny{$O$}\end{tabular}}}}%
    \put(0.76498029,0.24030552){\color[rgb]{0,0,0}\makebox(0,0)[lt]{\lineheight{1.25}\smash{\begin{tabular}[t]{l}\tiny{$V^2\ni (x_1,x_2)$}\end{tabular}}}}%
    \put(0.86427394,0.18208314){\color[rgb]{0,0,0}\makebox(0,0)[lt]{\lineheight{1.25}\smash{\begin{tabular}[t]{l}\tiny{$x_1\in V^1$}\end{tabular}}}}%
    \put(0.43074003,0.18410856){\color[rgb]{0,0,0}\makebox(0,0)[lt]{\lineheight{1.25}\smash{\begin{tabular}[t]{l}\tiny{$O$}\end{tabular}}}}%
    \put(0.53748058,0.18741835){\color[rgb]{0,0,0}\makebox(0,0)[lt]{\lineheight{1.25}\smash{\begin{tabular}[t]{l}\tiny{$x_1$}\end{tabular}}}}%
    \put(0.45383275,0.2982962){\color[rgb]{0,0,0}\makebox(0,0)[lt]{\lineheight{1.25}\smash{\begin{tabular}[t]{l}\tiny{$x_2$}\end{tabular}}}}%
    \put(0.41915578,0.31352921){\color[rgb]{0,0,0}\makebox(0,0)[lt]{\lineheight{1.25}\smash{\begin{tabular}[t]{l}\tiny{$x_3$}\end{tabular}}}}%
  \end{picture}%
\endgroup%